\documentclass[a4paper, 10pt]{amsart}

\usepackage{preamble}

\makeatletter

\renewcommand{\tocsection}[3]{%
\indentlabel{\@ifnotempty{#2}{\parbox[b]{3ex}{\bfseries\ignorespaces#1 #2}}}\bfseries#3} 

\renewcommand{\tocsubsection}[3]{%
\indentlabel{\@ifnotempty{#2}{\hspace{1.3em}\parbox[b]{5ex}{\ignorespaces#1 #2}}}#3}

\renewcommand{\tocsubsubsection}[3]{%
\indentlabel{\@ifnotempty{#2}{\hspace{3.9em}\parbox[b]{5ex}{\ignorespaces#1 #2}}}#3}

\makeatother

\newcommand{\locadjL}[1]{\mathbin{% for localisation adjunctions
    \begin{tikzpicture}[baseline=-0.8ex]
	\coordinate (source) at (0ex,0ex);
	\coordinate (target) at (3ex,0ex);
	\draw[-stealth] ([yshift=0.7ex]source) to node[above]{\scriptsize $#1$}([yshift=0.7ex]target);
    \draw[left hook-stealth] ([yshift=-0.85 ex]target) to ([yshift=-0.85 ex]source);
	\node at (1.5ex,0  ex) {$\scriptscriptstyle \perp$};
    \end{tikzpicture}
}}

\newcommand{\colim}{\operatornamewithlimits{colim}}

\newcommand{\id}{\mathrm{id}}
\newcommand{\im}{\mathrm{im}}
\newcommand{\cofib}{\mathrm{cofib}}
\newcommand{\fib}{\mathrm{fib}}
\newcommand{\map}{\mathrm{map}}
\newcommand{\Map}{\mathrm{Map}}
\newcommand{\Fun}{\mathrm{Fun}}
\newcommand{\FunEx}{\Fun^{\Ex}}
\newcommand{\FunL}{\Fun^{L}}
\newcommand\FunInt{\underline{\Fun}^{\Ex}}
\newcommand\catop[1]{{#1}^{\mathrm{op}}}
\newcommand{\Ar}{\mathrm{Ar}}
\newcommand{\Ind}{\mathrm{Ind}}
\newcommand{\End}{\mathrm{End}}
\newcommand{\Ex}{\mathrm{ex}}
\newcommand{\Exc}{\mathrm{Exc}}
\newcommand{\Nat}{\mathrm{Nat}}
\newcommand{\Eq}{\mathrm{Eq}}

\newcommand{\TwAr}{\mathrm{TwAr}}
\newcommand{\Un}{\mathrm{Un}}
\newcommand{\fgt}{\mathrm{fgt}}
\newcommand{\Perf}{\mathrm{Perf}}
\newcommand{\Mod}{\mathrm{Mod}}
\newcommand{\sqz}{\mathrm{sqz}}

\newcommand{\Spaces}{\mathrm{Sps}}
\newcommand{\SpacesFinStar}{\Spaces^{\fin}_*}
\newcommand{\Sp}{\mathrm{Sp}}
\newcommand{\Spfin}{\Sp^{\fin}}
\newcommand{\LoopInfty}{\Omega^\infty}
\newcommand{\OmegaInfty}{\Omega^\infty}
\newcommand{\SigmaInftyP}{\Sigma^\infty_+}
\newcommand{\gen}{\mathrm{gen}}
\newcommand{\PgcSp}{\mathrm{PgcSp}}
\newcommand{\PgcSpGen}{\PgcSp^{\gen}}
\newcommand{\Alg}{\mathrm{Alg}}
\newcommand{\Bimod}{\mathrm{Bimod}}

\newcommand{\St}{\mathrm{St}}
\newcommand{\Lace}{\mathrm{Lace}}
\newcommand{\Pair}{\mathrm{Pair}}
\newcommand{\baBar}{\mathrm{Bar}}

\newcommand{\nexcPart}[1]{\mathrm{P}_{#1}}
\newcommand{\fbwnexcPart}[1]{\mathrm{P}^{\mathrm{tgt}}_{#1}}

\newcommand{\Kth}{\mathrm{K}}
\newcommand{\THH}{\mathrm{THH}}
\newcommand{\TC}{\mathrm{TC}}
\newcommand{\uTHH}{\mathrm{uTHH}}
\newcommand{\TR}{\mathrm{TR}}

\newcommand\Klace{\Kth^{\lace}}

\newcommand\LaceEq{\Lace^{\simeq}}
\newcommand\SpLaceEq{\SigmaInftyP\LaceEq}

\newcommand{\Free}{\mathrm{Free}}
\newcommand{\ev}{\mathrm{ev}}
\newcommand{\tr}{\mathrm{tr}}

\newcommand\E{\mathbb{E}}
\newcommand\Z{\mathbb{Z}}
\newcommand\bS{\mathbb{S}}

\newcommand\Acal{\mathcal{A}}
\newcommand\Bcal{\mathcal{B}}
\newcommand\Ccal{\mathcal{C}}
\newcommand\Dcal{\mathcal{D}}
\newcommand\Ecal{\mathcal{E}}
\newcommand\Fcal{\mathcal{F}}
\newcommand\Gcal{\mathcal{G}}
\newcommand\Ical{\mathcal{I}}
\newcommand\Jcal{\mathcal{J}}
\newcommand\Kcal{\mathcal{K}}
\newcommand\Mcal{\mathcal{M}}
\newcommand\Pcal{\mathcal{P}}
\newcommand\Scal{\mathcal{S}}
\newcommand\Zcal{\mathcal{Z}}

\newcommand{\cat}[1]{\mathrm{#1}}

\newcommand{\Cat}{\cat{Cat}}
\newcommand{\CAT}{\cat{CAT}}
\newcommand{\CatEx}{\cat{Cat}^{\Ex}}

\newcommand{\PrCat}[1]{\cat{Pr}^{\mathrm{#1}}}
\newcommand{\PrExCat}[1]{\cat{Pr}^{\mathrm{#1}}_{\Ex}}

\newcommand{\tangentT}[1][]{
    \ifthenelse{ \equal{#1}{} }{
        \mathcal{T}
    }{
        \mathcal{T}_{#1}
    }
}
\newcommand{\TCatEx}[1][]{
    \ifthenelse{ \equal{#1}{} }{
        \mathcal{T}\CatEx
    }{
        \mathcal{T}_{#1}\CatEx
    }
}

\newcommand{\tgt}{\mathrm{tgt}}

\newcommand{\lace}{\mathrm{lace}}

\newcommand{\add}{\mathrm{add}}
\newcommand{\cyc}{\mathrm{cyc}}

\newcommand{\op}{{^{\mathrm{op}}}}
\newcommand{\BiMod}{\mathrm{BiMod}}
\newcommand{\Om}{\Omega}

\newcommand{\Sig}{\Sigma}
\newcommand{\lax}{\mathrm{lax}}
\newcommand{\opl}{\mathrm{opl}}

\newcommand{\st}{\mathrm{st}}
\newcommand{\cont}{\mathrm{cont}}

\newcommand{\adj}{\mathrel{\substack{\longrightarrow \\[-.6ex] \stackrel{\upvdash}{\longleftarrow}}}}
\newcommand{\cocolon}{\nobreak \mskip6mu plus1mu \mathpunct{}\nonscript\mkern-\thinmuskip {:}\mskip2mu \relax}

\let\core\relax
\newcommand{\core}{(-)^\simeq}

\newcommand{\fin}{\mathrm{fin}}
\newcommand{\mon}{\mathrm{mon}}
\newcommand{\surj}{\mathrm{srj}}
\newcommand{\ad}{\mathrm{ad}}
\renewcommand{\L}{\mathrm{L}}
\newcommand{\Spaf}{\mathrm{Sp}^{\fin}}
\newcommand{\CatL}{\cat{Cat}^{\lace}}
\newcommand{\Catb}{\Cat^{\mathrm{b}}}
\newcommand{\cof}{\mathrm{cof}}

\newcommand{\cart}{\mathrm{cart}}
\newcommand{\nil}{\mathrm{nil}}
\newcommand{\bex}{\mathrm{fex}}
\newcommand{\vloc}{\mathrm{vloc}}

\title[The linear approximation of algebraic K-theory]{Trace methods for stable categories I:\\ The linear approximation of algebraic K-theory}

\author{Yonatan Harpaz, Thomas Nikolaus, Victor Saunier}

\begin{document}
\begin{abstract}
    We study algebraic K-theory and topological Hochschild homology in the setting of bimodules over a stable category, a datum we refer to as a laced category. We show that in this setting both K-theory and THH carry universal properties, the former defined in terms of additivity and the latter via trace properties. We then use these universal properties in order to construct a trace map from laced K-theory to THH, and show that it exhibits THH as the first Goodwillie derivative of laced K-theory in the bimodule direction, generalizing the celebrated identification of stable K-theory by Dundas-McCarthy, a result which is the entryway to trace methods.
\end{abstract}

\maketitle

\tableofcontents

\section{Introduction}

In \cite{DundasMcCarthy}, Dundas-McCarthy solved a conjecture of Goodwillie identifying \textit{stable K-theory} with topological Hochschild homology as defined by Bökstedt, up to a shift. This result opened trace methods, once limited to the rational world, to integral considerations, and is the stepping stone to the celebrated Dundas-Goodwillie-McCarthy theorem \cite{DGMBook}, stating that for every map $f:A\to B$ of connective ring spectra such that $\pi_0(f)$ is surjective with nilpotent kernel, the associated square of spectra
\[
    \begin{tikzcd}
        \Kth(A)\arrow[r]\arrow[d] & \TC(A)\arrow[d] \\
        \Kth(B)\arrow[r] & \TC(B)
    \end{tikzcd}
\]
is cartesian, where $\TC$ denotes topological cyclic homology, an invariant constructed from the cyclotomic structure on $\THH$. This theorem, uncovering the local structure of K-theory, has been instrumental for a flurry of results, the most recent being the computation of $\Kth(\Z/p^n\Z)$ by Antieau-Krause-Nikolaus \cite{AntieauKrauseNikolaus} and the recent disproof of the telescope conjecture by Burklund-Hahn-Levy-Schlank \cite{BurklundHahnLevySchlank}. It allows one to access the K-theory of complicated rings by building on foundational computations in K-theory — such as those by Quillen or Suslin \cite{QuillenFiniteFields, Suslin} for respectively finite and algebraically closed fields — and computations of the more amenable $\TC$ via the machinery of spectral sequences and the theory of cyclotomic spectra. \\

In a parallel development, the seminal work of Quillen in the 70's and Waldhausen in the 80's made it evident that algebraic K-theory is best viewed not as an invariant of rings, but rather as one of categories, or better yet, of higher categories. This approach has enjoyed tremendous growth in the last decade after Barwick~\cite{BarwickHigher}, working with Waldhausen $\infty$-categories, and Blumberg-Gepner-Tabuada~\cite{BlumbergGepnerTabuada}, working with stable $\infty$-categories, showed that the higher categorical setting grants algebraic K-theory a universal property. The latter of the two approaches is the one we adopt in the present paper, and to simplify terminology we will simply call (stable) $\infty$-categories (stable) categories.

This universal property opens a paradigm shift for K-theory: it is no longer an object defined by involved constructions but a functor universally performing an operation on stable categories.
Coming back to the Dundas-McCarthy theorem, 
one could hope to understand the identification of stable K-theory with topological Hochschild homology from the point of view of universal properties. According to Dundas himself in \cite[Section 3.2]{DundasHistory}, although the proof of \cite{DundasMcCarthy} is following a radically different philosophy, premises of such a program were envisioned by Schwänzl, Staffeldt and Waldhausen in \cite{SchwanzlStaffeldtWaldhausen}. \\

The goal of this article is precisely to fulfill this vision: we will prove that stable K-theory and topological Hochschild homology coincide generally for all stable categories and all choices of coefficients. In particular, we will recover the result of Dundas-McCarthy by plugging the category of compact $R$-module spectra for a connective ring spectrum $R$. Moreover, our proof will run at a high-level of abstraction, by providing universal properties for both functors and then showing that the two properties for which they are universal end up coinciding. Let us now run through an exposition of our methods. 

Stable K-theory $\Kth^S$ of a ring spectrum $R$ with coefficients in a $R$-bimodule $M$ is usually defined as the exact approximation of the functor $M\mapsto \Kth(R\oplus M\varepsilon)$, the K-theory of the square-zero extension of $R$ by $M$. Dundas and McCarthy proved in \cite{DundasMcCarthy} the identification $\Kth^S(R, M)\simeq\THH(R, \Sigma M)$. Alternatively, to define stable K-theory, one could consider the exact approximation of $\Kth(\End(R, M))$, the K-theory of the category $\End(R, M)$ of $M$-parameterized $R$-endomorphisms, whose objects are compact $R$-module spectra $N$ equipped with a $R$-linear map $N\to M\otimes_R N$. If $R$ and $M$ are both connective then there is an equivalence of categories between $\Perf(R\oplus M)$ and $\End(R, \Sigma M)$, yielding an identification of this two points of view. This idea already features in Dundas-McCarthy's proof, and we note that passing along this equivalence removes the shift from the identification of stable $\Kth$-theory with $\THH$. \\

Let us then recast that story in the categorical world: for a presentable category $\Ecal$, its tangent bundle $\tangentT \Ecal$ is defined as the category $\Exc(\SpacesFinStar, \Ecal)$ of excisive functors from finite pointed spaces into $\Ecal$. There is a bicartesian fibration $\fgt\colon\tangentT \Ecal\to\Ecal$ classifying the functor $X\mapsto\Sp(\Ecal_{/X})$, sending an object $X$ to the stabilization of the overcategory $\Ecal_{/X}$. The collection of functors $\OmegaInfty_{/X}\colon\Sp(\Ecal_{/X})\to\Ecal_{/X}$ upgrades to a functor:
\[
    \begin{tikzcd}
        \sqz\colon\tangentT\Ecal\arrow[r] & \Ar(\Ecal)
    \end{tikzcd}
\]
which universally characterizes $\tangentT\Ecal$. For $\Ecal=\Alg_{\E_1}(\Sp)$, the category of $\E_1$-ring spectra, Lurie has shown that the tangent bundle of $\Ecal$ identifies with the category of pairs $(R,M)$ with $R$ a $\E_1$-ring spectra and $M$ a $R$-bimodule, such that the above functor is indeed the square-zero extension functor $(R,M)\mapsto [R\oplus M \to R]$. 

Let us now consider the case $\Ecal = \CatEx$, the category of stable categories and exact functors between them. Recall that a bimodule on a stable category $\Ccal$ is a bi-exact functor $M\colon \Ccal\op \times \Ccal \to \Sp$, or equivalently, an exact functor $M\colon \Ccal \to \Ind(\Ccal)$, where $\Ind(\Ccal) := \FunEx(\Ccal\op,\Sp)$ is the Ind-completion of $\Ccal$. Somewhat abusively, we will pass freely between these two points of view without changing the notation.

Denote by $\BiMod(\Ccal)$ the category of $\Ccal$-bimodules. 
Given a bimodule $M\in \BiMod(\Ccal)$, there is a category $\Lace(\Ccal,M)$ whose points are $M$-laced objects in $\Ccal$, 
i.e., pairs $(X, f)$ where $X$ is an object of $\Ccal$ and $f\colon X\to M(X)$ is a morphism in $\Ind(\Ccal)$. A central example is as follows: if $M$ is an $R$-bimodule for $R$ a ring spectrum, then $M\otimes_R-\colon\Perf(R)\to\Mod(R)$ is a $\Perf(R)$-bimodule. In this case, we have an identification:
$$
    \Lace(\Perf(R), M\otimes_R-)\simeq\End(R, M)
$$
In general, we may view $\Lace(\Ccal,M)$ as an object of $\CatEx_{/\Ccal}$ by remembering the obvious forgetful functor. Our first result is the following:

\begin{thm}
    Let $\Ccal$ be a stable category, then the functor
    \[ \Lace(\Ccal,-)\colon \BiMod(\Ccal) \to \CatEx_{/\Ccal} \]
    exhibits $\BiMod(\Ccal)$ as the stabilization of $\CatEx_{/\Ccal}$. In particular, it induces an equivalence 
    \[ \BiMod(\Ccal) \simeq \Sp(\CatEx_{/\Ccal}) \] 
    under which the functor $\Lace(\Ccal,-)$ identifies with $\Om^{\infty}_{/\Ccal}$.
\end{thm}

Letting $\Ccal$ vary, the above theorem assembles into an equivalence
\[ \TCatEx \simeq \int_{\Ccal \in \CatEx} \BiMod(\Ccal) \]
between the tangent bundle of $\CatEx$ and the unstraightening of the functor $\Ccal\mapsto\BiMod(\Ccal)$, which explicitly is the category of pairs $(\Ccal,M)$ consisting of a stable category $\Ccal$ equipped with a $\Ccal$-bimodule $M$. This may be considered as a stable categorical counterpart of Lurie's identification of $\tangentT\Alg_{\E_1}(\Sp)$ in terms of rings equipped with a bimodule, see \cite[Theorem 7.3.4.13]{HA}. 

Let us fix some terminology: we call such pairs $(\Ccal, M)\in\TCatEx$ \textit{laced categories}. 
Laced functors $(f, \alpha)\colon(\Ccal, M)\to(\Dcal, N)$ are arrows in the tangent bundle $\TCatEx$; explicitly, this is the datum of an exact functor $f\colon\Ccal\to\Dcal$ and a natural transformation $\alpha\colon M\Rightarrow N\circ(\catop{f}\times f)$.

The following definition generalizes what is usually known as K-theory of parameterized endomorphisms. In particular, for $M=\id$, this is the K-theory of endomorphisms.

\begin{defi}
    Let $(\Ccal,M)$ be a laced category. We define the \textit{laced K-theory} of $(\Ccal, M)$ to be the spectrum $\Klace(\Ccal, M):=\Kth(\Lace(\Ccal, M))$. 
\end{defi}

Our first main result is that $\Klace$ enjoys a universal property of a similar flavour to the one of \cite{BlumbergGepnerTabuada}, whose results imply that $\SigmaInftyP\core\to\Kth$ is the initial natural transformation whose target is an \textit{additive invariant}, where $\core$ denotes the largest subgroupoid of a category. The precise meaning of additivity for functors on the tangent bundle is worked out in \S\ref{lace-additivity}. 

\begin{thm} \label{KlaceIsUnivAdditiveIntro}
    The natural transformation $\SpLaceEq\Rightarrow\Klace$ of functors $\TCatEx\to\Sp$ exhibits laced K-theory as the initial additive invariant under $\SpLaceEq$.
\end{thm}

If $\Ccal$ is a stable category, stable K-theory $\Kth^S(\Ccal, -)$ is the exact approximation of $\Klace(\Ccal, -)$ in the sense of Goodwillie calculus. These approximations behave well in families so that they assemble into a functor, which we still call stable K-theory
$$
    \fbwnexcPart{1}\Klace\colon \TCatEx\longrightarrow\Sp
$$
which maps $(\Ccal, M)$ to the value at $M$ of the exact approximation of $\Klace(\Ccal, -)$ (the $\tgt$ stands for tangentially).
It follows from the previous theorem that stable K-theory is also given by a universal property under $\SpLaceEq$, namely the composite 
\[
    \begin{tikzcd}
        \SpLaceEq\arrow[r] & \Klace\arrow[r] & \fbwnexcPart{1}\Klace
    \end{tikzcd}
\]
exhibits its target as the initial additive \textit{tangentially exact} invariant. Here, a functor $F\colon \TCatEx\to\Sp$ is said to be tangentially exact if the restriction to each fibre $F(\Ccal, -)\colon \BiMod(\Ccal)\to\Sp$ is exact. \\

Let us now pass to topological Hochschild homology. There are multiple ways of defining $\THH$ for laced categories, let us pick the simpler one to express, which is akin to writing $R\otimes_{\catop{R}\otimes R}M$ for ring spectra. If $(\Ccal, M)$ is a laced category, we let $\THH$ to be the following coend:
\[
    \THH(\Ccal, M):=\int^{X\in\Ccal} M(X, X).
\]
By the Bousfield-Kan formula, this coend can be shown to coincide with the usual cyclic bar construction with many objects, used notably in \cite[Section 10]{BlumbergGepnerTabuada}. In fact, in the laced world, this cyclic bar construction is an instance of a more general phenomenon, which we now explain. 

First, we note that the category $\TCatEx$ has all small limits and colimits. In fact, we will show that it is presentable, and compactly-generated by two objects. Using Day convolution, it inherits a closed symmetric monoidal structure which makes the bicartesian fibration $\fgt\colon\TCatEx\to\CatEx$ into a monoidal functor. 

If $(\Ccal, M)$ is a laced category, we let $(\Ccal, M)^{([1], *)}$ denote the new laced category whose underlying category is the arrow category $\Ccal^{[1]}$ and whose bimodule is given as
\[
(X_0 \to X_1, Y_0 \to Y_1) \mapsto M(X_1,Y_0) \ .
\]
As the notation suggests, this construction is an instance of a cotensor. More specifically, $\TCatEx$ is tensored and cotensored over $\Cat^b$, the unstraightening of the functor $\Ical\mapsto\Fun(\catop{\Ical}\times\Ical, \Spaces)$ and the category descibed above is the 
cotensor  by $([1], *)$, where $\ast\colon\catop{[1]}\times[1]\to\Spaces$ is the terminal functor. We write $d_0, d_1\colon (\Ccal, M)^{([1], *)}\to(\Ccal, M)$ for the laced functors induced by the faces $d_0, d_1\colon [0]\to[1]$. 

The laced functors $d_0$ and $d_1$ are generating examples of \textit{strict trace equivalences}. Indeed, we can consider the objects of the mapping space
\[
    \Map_{\TCatEx}\left( (\Ccal, M), (\Dcal, N)^{([1],*)} \right)
\]
as homotopies between their boundary points, we call such maps \textit{trace homotopies}. A strict trace equivalence is then a laced functor which has an inverse up to trace homotopy. The projection $s_0\colon [1]\to[0]$ shows that $d_0, d_1$ are strict trace equivalences with a common inverse induced by $s_0$, and more generally, the collection of maps $(\Ccal, M)^{([m], *)}\to(\Ccal, M)^{([n], *)}$ generates as a saturated class the collection of strict trace equivalences, in the sense that inverting such maps is equivalent to inverting all strict trace equivalences. 

We say that a functor $F\colon \TCatEx\to\Ecal$ is \textit{trace-like} if it inverts strict trace equivalences. If $\Ecal$ is presentable, we let
\[
    \cyc(F)(\Ccal, M):=\left|F\left( (\Ccal, M)^{([n], *)} \right)\right|
\]
denote the cyclic bar construction of $F$. The above characterization implies that the natural transformation $F\Rightarrow\cyc(F)$ induced by the inclusion of the 0-simplices has the following universal property:

\begin{prop}
    The natural transformation $F\Rightarrow\cyc(F)$ exhibits its target as the initial trace-like functor under $F$.
\end{prop}

Let us introduce the following unstable version of $\THH$, denoted $\uTHH(\Ccal, M)$ by taking the coend of the functor $\OmegaInfty M\colon \catop{\Ccal}\times\Ccal\to\Spaces$. It holds that $\uTHH(\Ccal, M)\simeq\cyc(\LaceEq)$, so that there is a natural transformation $\LaceEq\to\uTHH$ exhibiting unstable THH as the initial trace-like functor under $\LaceEq$. The colimit comparison map for $\OmegaInfty$ provides a natural transformation $\uTHH\to\OmegaInfty\THH(\Ccal, M)$ and we are able to show:

\begin{thm}
    The induced natural transformation $\eta\colon \SigmaInftyP\uTHH\Rightarrow\THH(\Ccal, M)$ exhibits its target as the initial tangentially exact invariant under its source. 
    
    As a consequence, the composite $\SpLaceEq\Rightarrow\THH$ exhibits topological Hochschild homology as the initial trace-like tangentially exact invariant under its source.
\end{thm}

In spirit, this universal property of $\THH$ it shares many of the ideas that led to the universal property of K-theory in \cite{BlumbergGepnerTabuada}; 
indeed, just like the $S_\bullet$-construction, the $\cyc$ construction is the geometric realization of a simplicial object whose faces are generators of arrows one wants to control (strict trace equivalences for $\THH$, split-Verdier projection for K-theory --- see \cite[Lemma 2.2.8]{HermKII}). 
On the other hand, it exemplifies the difference between K-theory and $\THH$: the universal property of the latter is to be found in a world with coefficients whereas the former is already universal as an absolute invariant, before the introduction of coefficients. \\

Now that we have characterized stable K-theory and $\THH$ as universal objects under the same functor, namely $\SpLaceEq$, the result we are after is a formal consequence of the following theorem:

\begin{thm}
    Let $F\colon \TCatEx\to\Ecal$ be a tangentially exact functor with target a stable category. Then, the following are equivalent:
    \begin{enumerate}
        \item[(i)] $F$ is additive
        \item[(ii)] $F$ is trace-like
    \end{enumerate}
    In particular, $\THH$ is additive.
\end{thm}

The implication $(ii) \Rightarrow (i)$ is obtained by adapting a classical argument, notably found in \cite[Section 5.2]{Kaledin}, where it is presented as the gist of the proof of the localization theorem for Hochschild homology proven by \cite{Keller}. A close variant of this argument also appears in \cite[Theorem 3.4]{HoyoisScherotzkeSibilla}. Comparing universal properties, we then obtain our main result:

\begin{cor}[Stable K-theory is THH]\label{cor:main}
    There is a natural, canonical transformation $\Klace\Rightarrow\THH$ which exhibits its target as the initial tangentially exact invariant under its source. In particular, $\THH$ is stable K-theory.
\end{cor}

The present paper is a precursor for our upcoming work \cite{HarpazNikolausSaunier}, where we 
show that $\THH$ lifts to a functor $\TCatEx\to\PgcSpGen$, a genuine version of the category of polygonic spectra introduced in \cite{KrauseMcCandlessNikolaus}, and that this structure restricts in particular to a genuine cyclotomic structure on $\THH(\Ccal, \map)$; in particular, we are able to recover the cyclotomic trace from K-theory to topological cyclic homology from those results, in a higher categorical version of Schlichtkrull's ideas~\cite{SchlichtkrullCyclotomicTrace}. We then deduce that the Goodwillie-Taylor tower of laced K-theory is given by the limit of the truncation tower of $\TR(\Ccal, M)$, in particular recovering the main result of Lindenstrauss-McCarthy in \cite{LindenstraussMcCarthyTaylorTower}. In fact, those statements will fit in a general framework applicable to any Verdier localising invariant on $\CatEx$ in place of K-theory. We note that one of the shortcomings of the formalism of the present paper is that it does not explain or allow to construct the $S^1$-action on $\THH$; this part will also be covered by the forthcoming paper.

This paper and its successor \cite{HarpazNikolausSaunier} grew out of a manuscript of the second author \cite{MFO} prepared for the 2018 Arbeitsgemeinschaft: Topological Cyclic Homology in Oberwolfach. The second author would like to apologize for the extended delay in bringing these ideas to publication.

Finally, let us note that in his thesis \cite{RamziDM}, Maxime Ramzi obtains a similar result to Corollary~\ref{cor:main}, although his proof is quite different and more in line with the program of \cite{SchwanzlStaffeldtWaldhausen}. 
We also note that Sam Raskin has observed similar phenomena to the one we describe here \cite{Raskin}.

\subsection{Conventions}
As we have done throughout most of this introduction, we will use the higher categorical framework of $\infty$-categories developed by Lurie in \cite{HTT, HA, SAG}, and we will omit the $\infty$ in front of our categories and the homotopy in front of our (co)limits. We denote $\Spaces$ the category of spaces, and $\Sp$ the category of spectra; we let $\Cat$ be the category of (small) categories and $\CatEx$ be the subcategory of stable categories and exact functors between them. If $\Ccal$ is stable, $\Ind(\Ccal)$ denotes its Ind-construction, given explicitly by $\FunEx(\catop{\Ccal}, \Sp)$; we will write $\map_\Ccal$ for the mapping spectra of $M$ and $\Map_\Ccal$ for the mapping space. Moreover, $\Kth$ will stand for the \textit{connective} K-theory functor from $\CatEx$ to spectra. 

\subsection{Acknowledgements}
Many ideas in this series of papers are originally due to Kaledin and we would like to acknowledge this.
We wish to thank Christian Ausoni, Bjorn Dundas, David Gepner, Fabian Hebestreit, Lars Hesselholt, Kaif Hilman, Dimitry Kaledin, Dominik Kirsten, Christian Kremer, Jordan Levin, Zhoufang Mao, Jonas McCandless, Maxime Ramzi, Sam Raskin, John Rognes, and Peter Scholze for fruitful discussions, exchanges and invitations. 
The first and third authors were supported by the European Research Council as part of the project \emph{Foundations of Motivic Real K-Theory} (ERC grant no. 949583).
The second author was funded by the Deutsche Forschungsgemeinschaft (DFG, German Research Foundation) – Project-ID 427320536 – SFB 1442, as well as under Germany’s Excellence Strategy EXC 2044 390685587, Mathematics M\"unster: Dynamics–Geometry–Structure. He would also like to thank the Mathematische Forschunsinsitut Oberwolfach for hosting the 2018 Arbeitsgemeinschaft: Topological Cyclic Homology which prompted the work on this subject.
The third author was further supported by the German Research Foundation through the research centre ‘Integral structures in Geometry and Representation theory’ (grant no. 491392403 - TRR 358) at the University of Bielefeld.

%%%%%%%%%%%%%%%%%%%%%%%%%%%%%%%%%%%%%%%%%%%%%%%%%%%%%%%%%%%%%%%%%%%%%%%%%%%%%%%%%%%%%%%%%%%%%%%%%%%%%%%%%%%%%%%%

\section{Laced categories} 

Let us begin by briefly recalling some facts about tangent categories and the tangent bundle of a sufficiently nice category, which were introduced in \cite[\S 7.3]{HA}. Let $\Ecal$ be an category with finite limits and denote by $\SpacesFinStar$ the category of finite pointed spaces, that is, the smallest full subcategory of $\Spaces_*$ containing $S^0$ and stable under finite colimits. There is a cartesian fibration
\[
    \begin{tikzcd}
        \pi\colon \Exc(\SpacesFinStar, \Ecal)\ar[r] & \Ecal\vspace{-0.5em}
    \end{tikzcd}
\]
given by the evaluation at $\ast \in \SpacesFinStar$, where the left-hand-side category is the category of excisive functors $\SpacesFinStar\to\Ecal$. Since $\ast$ is terminal in $\SpacesFinStar$ the fibre of $\pi$ over some $X\in\Ccal$ is naturally equivalent to the category $\Exc_*(\SpacesFinStar, \Ecal_{/X})$ of \emph{reduced} excisive functors to $\Ecal_{/X}$, which is a model for 
the stabilization of $\Ecal_{/X}$ or, equivalently, for the stabilization of the category $(\Ecal_{/X})_* = \Ecal_{X//X}$ of pointed objects in $\Ecal_{/X}$, which are objects over $X$ equipped with a section. In addition, If $\Ecal$ is presentable then the fibres of $\pi$ are again presentable and $\pi$ is also a cocartesian fibration since all the cartesian transition functors have left adjoints.

Following the philosophy of Goodwillie calculus, stable categories can be considered as the linear objects of higher category theory, and so, viewing $\Ecal_{/X}$ as a type of a neighbourhood of $X$ in $\Ecal$, its linearization $\tangentT[X]\Ecal:=\Sp(\Ecal_{/X})$ is called the \textit{tangent category of $\Ecal$ at $X$}. Furthering the analogy with manifolds, the cartesian fibration assembling the various tangent categories is called the \textit{tangent bundle of $\Ecal$}, and denoted by 
\[ \tangentT\Ecal:=\Exc(\SpacesFinStar, \Ecal).\] 
Evaluation along $S^0\to*$ provides a functor $\sqz$, the square-zero extension functor, given by:
\begin{align*}
    \sqz\colon\tangentT\Ecal& \longrightarrow \Ecal^{\Delta^1} := \Fun(\Delta^1,\Ecal) \\
         (X, M)& \longmapsto [\OmegaInfty_{/X}M\to X]
\end{align*}
which characterizes the tangent bundle as the stable envelope, in the sense of \cite[Section 7.3.1]{HA}, of the target projection $t\colon\Ar(\Ecal)\to\Ecal$. 
A classical and important example is the tangent bundle of the category of $\E_{\infty}$-ring spectra: by \cite[Theorem 7.3.4.14]{HA}, it is the cartesian fibration classified by the functor $\Alg_{\E_{\infty}}(\Sp)\op \to \Cat$ associating to an $\E_{\infty}$-ring spectrum $R$ its category of $R$-modules. 
Under this equivalence, $\sqz$ is the usual square-zero extension functor $(R,M) \mapsto R \oplus M$ of $\E_{\infty}$-ring spectra. 
If instead of $\Alg_{\E_\infty}(\Sp)$ we take $\Alg_{\E_1}(\Sp)$ then the same holds with the notion of an $R$-module replaced by that of an $R$-\emph{bimodule}.

Our goal in this section is to study the tangent bundle of $\CatEx$, the category of stable categories and exact functors between them. 
The main idea is that the tangent bundle of $\CatEx$ admits a description similar to the one just recalled for $\E_1$-rings, 
obtained by extending the notion of a bimodule from ring spectra to stable categories. 

\subsection{Bimodules on a stable category}

Recall that for an $\infty$-category $\Ccal$, its Ind-completion $\Ind(\Ccal)$ is given by the smallest full subcategory of $\Fun(\Ccal\op,\Spaces)$ containing the image of the Yoneda embedding and closed under filtered colimits. If $\Ccal$ admits finite colimits then we may also identify $\Ind(\Ccal)$ with the $\infty$-category of left exact functors $\Ccal\op \to \Spaces$ (\cite[Corollary 5.3.5.4]{HTT}), and if $\Ccal$ is furthermore is stable then this is equivalent,  
by the universal property of $\Sp$, to the $\infty$-category of exact functors $\Ccal\op \to \Sp$. In particular, $\Ind(\Ccal)$ is again stable and admits small colimits. Furthermore, the Yoneda embedding $\Ccal \to \Ind(\Ccal)$ exhibits $\Ind(\Ccal)$ as the universal such recipient of an exact functor from $\Ccal$.

For a stable category $\Ccal$, the notion of a \emph{$\Ccal$-bimodule} can be defined as either one of the following four equivalent notions:
\begin{enumerate}
    \item An exact functor $\Ccal\op \otimes \Ccal \to \Sp$.
    \item A biexact functor $\Ccal\op \times \Ccal \to \Sp$.
    \item An exact functor $\Ccal \to \Ind(\Ccal)$.
    \item A colimit-preserving functor $\Ind(\Ccal) \to \Ind(\Ccal)$.
\end{enumerate}

The passage between the first two types of functors is via the defining property of the tensor product in $\CatEx$. The passage between the second and third is by restricting the currying equivalence $\Fun(\Ccal\op\times \Ccal,\Sp) \simeq \Fun(\Ccal,\Fun(\Ccal\op,\Sp))$ to functors satisfying the relevant exactness conditions, and the passage between the third and the fourth is via the universal property of $\Ind(\Ccal)$. In the present paper it will be useful to keep in mind all four forms a $\Ccal$-bimodule can take, and we will freely shift between them, often without indicating it explicitly in the notation. We will write $\BiMod(\Ccal)$ for the (stable) category of $\Ccal$-bimodules.

\begin{ex} \label{ExampleRingSp1}
    Let $R$ be an $\E_1$-ring spectrum, and denote by $\Perf(R)$ the category of perfect $R$-modules, i.e., compact objects of $\Mod_R$ (to fix ideas, let us consider that modules always mean left modules). Then any $R$-bimodule $M$ determines a bimodule on the stable category $\Perf(R)$ in the form of an exact functor
    \[
        F_M\colon \Perf(R) \to \Mod(R) \quad\quad N\mapsto M \otimes_R N.
    \]
    The association $M\mapsto F_M$ upgrades to an equivalence of categories  
    \[ \BiMod_R\to\FunEx(\Perf(R),\Mod(R)) = \BiMod(\Perf(R))\]
    with inverse given by $F\mapsto F(R)$, where $F(R)$ is viewed as an $R\op$-module in $R$-modules using the identification $\End_R(R) = R\op$. 
\end{ex}

We now come to one of the principal notions of the present paper.
\begin{defi}
A \emph{laced category} is a pair $(\Ccal,M)$ where $\Ccal$ is a stable category and $M$ is a $\Ccal$-bimodule. Given two laced categories $(\Ccal,M),(\Dcal,N)$, a laced functor $(\Ccal,M) \to (\Dcal,N)$ is a pair $(f,\eta)$ where $f\colon \Ccal \to \Dcal$ is an exact functor and $\eta\colon M \Rightarrow N \circ (f\op \times f)$ is a natural transformation.
\end{defi}
The collection of laced categories and laced 
functors between them forms a category, which we denote by $\CatL$. One way to construct it is to note that the association $\Ccal \mapsto \BiMod(\Ccal)$ is contravariantly functorial in $\Ccal$ via restriction (when bimodules are viewed via one of the first two definitions above). One may then define $\CatL$ to be the total category of the cartesian fibration
\[ \pi\colon \CatL \to \CatEx \]
classified by $\Ccal \mapsto \BiMod(\Ccal)$. 
In \S\ref{subsec:tangent} below we will prove that $\CatL$ is naturally equivalent (over $\CatEx$) to the tangent bundle $\tangentT\CatEx$ of $\CatEx$. 
Before that, we dedicate the rest of the present subsection to study the basic properties of $\CatL$.

\begin{prop}\label{prop:limits-and-colimits}
    The category $\CatL$ has all small limits and colimits and the canonical functor $\pi\colon \CatL\to\CatEx$ preserves both.
\end{prop}
\begin{proof}
Recall that $\CatEx$ admits small limits and colimits, see \cite[6.1.1]{HermKI}. Now the functor $\pi\colon \CatL\to\CatEx$ is a bicartesian fibration with cartesian transition functors given by restriction and cocartesian transition functors given by left Kan extension (see~\cite[Lemma 1.4.1]{HermKI}). The fibres of $\pi$, which are the various bimodule categories, admit all small colimits, which the cocartesian transition functors automatically preserves (being left adjoints), and all small limits, which the cartesian transition functors automatically preserve (being right adjoints). Combining \cite[Corollary 4.3.1.11 and Proposition 4.3.1.5.(2)]{HTT} we now conclude that $\CatEx$ admits limits and colimits and that these are preserved by $\pi$.
\end{proof}

\begin{rmq}\label{rmq:explicit}
    The results of \cite{HTT} used in the proof of Proposition~\ref{prop:limits-and-colimits} also indicate the procedure to compute a colimit in $\CatL$. First, compute the colimit of the underlying diagram of stable categories, and denote by $\Ccal \in \CatEx$ the result. Then, left Kan extend every bimodule so that they become bimodules over $\Ccal$. This yields a diagram in $\BiMod(\Ccal)$ whose colimit $M$ yields a laced category $(\Ccal, M)$ which is the wanted colimit. 
    
    To compute a limit, the process is the same but using restriction of bimodules instead of left Kan extensions.
\end{rmq}

\begin{cons}\label{cons:laced-objects}
    Let $\Ccal$ be a stable category $M \in \BiMod(\Ccal)$ a bimodule.
    The underlying space valued functor $\Omega^{\infty}M \colon \Ccal\op \times \Ccal \to \Spaces$ can be unstraightened, first contravariantly in the first coordinate and then covariantly in the second, to yield a bifibration
    \[ 
        \begin{tikzcd}
            \Ccal & \ar[l] \int^{X \in \Ccal}_{Y \in \Ccal} \Omega^{\infty}M(X,Y) \ar[r] & \Ccal ,
        \end{tikzcd}
    \]
    see, e.g., \cite[\S 7.1]{HermKI}. We then define $\Lace(\Ccal,M)$ to be the pullback of this bifibration along the diagonal, that is, the category sitting in the fibre square
    \[
        \begin{tikzcd}
            \Lace(\Ccal,M) \ar[r]\ar[d] & \int^{X \in \Ccal}_{Y \in \Ccal} \Om^{\infty} M(X,Y) \ar[d] \\
        \Ccal \ar[r] & \Ccal \times \Ccal
        \end{tikzcd}
    \]
    We note that the left vertical map is neither a cartesian nor a cocartesian fibration in general. We may identify objects in $\Lace(\Ccal,M)$ with pairs $(X,\alpha)$ where $X$ is an object of $\Ccal$ and $\alpha\in \Om^{\infty}M(X,Y)$. We refer to these as \emph{laced objects} in $(\Ccal,M)$. 
\end{cons}

While not immediately visible from this description, the category $\Lace(\Ccal,M)$ is stable. To see this, it is useful to pass to an equivalent description of $\Lace(\Ccal,M)$ which arises from viewing $\Ccal$-bimodules as exact functors $M\colon \Ccal \to \Ind(\Ccal)$. Unwinding the correspondence between the two definitions of a $\Ccal$-bimodule, we see that the total object $\int^{X \in \Ccal}_{Y \in \Ccal} \Om^{\infty} M(X,Y)$ can be identified with the pullback along $j \times M:\Ccal \times \Ccal \rightarrow \Ind(\Ccal) \times \Ind(\Ccal)$ of the arrow bifibration 
$$ 
    \begin{tikzcd}
        \Ind(\Ccal)^{\Delta^{\{0\}}} &\arrow[l, "s"'] \Ind(\Ccal)^{\Delta^1} \arrow[r, "t"] & \Ind(\Ccal)^{\Delta^{\{1\}}}
    \end{tikzcd}
$$
We hence conclude that $\Lace(\Ccal,M)$ also sits in a fibre square of the form
\begin{equation}\label{eq:lace}
    \begin{tikzcd}
        \Lace(\Ccal,M) \ar[r]\ar[d] & \Ind(\Ccal)^{\Delta^1} \ar[d, "{(s, t)}"] \\
        \Ccal \ar[r, "{(j,M)}"] & \Ind(\Ccal) \times \Ind(\Ccal),
    \end{tikzcd}
\end{equation}
from which we see that $\Lace(\Ccal,M)$ is stable. 

\begin{ex} \label{SqZeroVSLace}
    Let $R$ be an $\E_1$-ring spectrum and $M$ a bimodule, and denote $F_M$ the associated $\Perf(R)$-bimodule of Example \ref{ExampleRingSp1}. Then, $\Lace(\Perf(R), F_M)$ is the category of compact modules $N$ equipped with a natural transformation $N\to M\otimes_{R} N$. It is often called the category of $M$-parametrized endomorphisms. 

    Now suppose that $R, M$ are connective. Then, $\Lace(\Perf(R), F_{\Sigma M})$ is generated by a single object, namely the pair $(R, 0\colon R\to \Sigma M)$. To see this, note that the inclusion $\Lace(\Perf(R),F_M)) \to \Lace(\Mod(R),F_M)$ takes values in compact objects and hence extends to a fully-faithful colimit preserving embedding
    \[ \Ind(\Lace(\Perf(R),F_M)) \to \Lace(\Mod(R),F_M) .\]
    In addition, if $N$ is a compact $R$-module equipped with an $M$-parameterized endomorphism $T\colon N \to \Sigma M \otimes_A N$ then $T$ is nilpotent, that is, the colimit 
    $$
        \colim\left(N \to \Sigma M \otimes_A N \to \Sigma^2 M \otimes_A M \otimes_A N\to \hdots\right)
    $$ 
    vanishes, as can be seen by connectivity considerations. The above functor hence takes values in the full subcategory of $\Lace^{\nil}(\Mod(R),F_M)\subseteq \Lace(\Mod(R),F_M)$ spanned by the nilpotent $M$-parameterized endomorphisms.
    
    Suppose $N$ is a $R$-module (not necessarily compact) equipped with a nilpotent $M$-parameterized endomorphism $T\colon N \to \Sigma M \otimes_A N$ in the above sense and such that:
    $$
        \map\left((R,0),(N,T)\right)\simeq 0
    $$
    then $T$ must be an equivalence, since this mapping spectra is also $\fib(T)$. This implies that $N = 0$ by the nilpotency assumption.  
    This means that $(R,0)$ is a compact generator of $\Lace^{\nil}(\Mod(R),F_M)$ and since $(R,0)$ is contained in $\Lace(\Perf(R),F_M)$ we conclude that the latter is exactly the full subcategory of compact objects in $\Lace^{\nil}(\Mod(R),F_M)$.
    
    Consequently, $\Lace(\Perf(R), F_{\Sigma M})$ is equivalent to $\Perf(A)$, where $A$ is the endomorphism ring spectrum of $0\colon R\to \Sigma M$. But this endomorphism ring spectrum is none other than the square-zero extension $R\oplus M$, so that we recover a special case of a result in \cite{Barkan}. When $M$ is compact, this can also be viewed as an instance of Koszul duality between the square-zero extension $R \oplus M$ and the cofree coalgebra generated by $\Sigma M$.
    
    In particular, we already have at the level of categories the result of Dundas-McCarthy in \cite{DundasMcCarthy}, comparing $\Kth(M\oplus R)$ and $\Kth(\Lace(\Perf(R), F_{\Sigma M}))$. 
\end{ex}

Our next step is constructing an explicit left adjoint to the functor $\Lace\colon \CatL \to \CatEx$ of Construction~\ref{cons:laced-objects}.
For this, let us now introduce an alternative approach to the construction of $\CatL$.
To begin, one first forms the oplax arrow category $\Ar^{\opl}(\PrExCat{L}) := \Fun^{\opl}(\Delta^1,\PrExCat{L})$ of the 2-category $\PrExCat{L}$ whose objects are the stable presentable categories and whose morphisms are the colimits preserving functors (equivalently, the left adjoint functors). Explicitly, the objects of $\Ar^{\opl}(\PrCat{L})$ are given by arrows $F\colon \Ccal \to \Dcal$ in $\PrExCat{L}$, and the morphisms from $F\colon \Ccal \to \Dcal$ to $F'\colon \Ccal'\to \Dcal'$ are given by oplax squares
\[
    \begin{tikzcd}
        \Ccal \ar[d, "F"]\ar[r] & \Ccal' \ar[d, "F'"]   \\
        \Dcal \ar[r] \ar[ur, Rightarrow, shorten >= 2.8ex, shorten <= 2.8ex] & \Dcal'. 
    \end{tikzcd}
\]
Formally, the definition of $\Ar^{\opl}(\PrExCat{L})$ can be made using general 2-categorical constructions of functor categories and oplax transformations between them (for example, see~\cite{GagnaHarpazLanariGray, GagnaHarpazLanariFibrations}). Alternatively, in this particular case one can also identify $\Ar^{\opl}(\PrExCat{L})$ with the full subcategory of $\CAT_{/\Delta^1}$ consisting of the cartesian fibrations $\Mcal \to \Delta^1$ whose fibres are stable and presentable and whose cartesian monodromy is colimit preserving. 

Now the fibre of the projection 
\[ \Ar^{\opl}(\PrExCat{L}) \xrightarrow{(s,t)} \PrExCat{L} \times \PrExCat{L} \]
over $(\Ccal,\Dcal)$ is given by $\Fun^L(\Ccal,\Dcal)$; more precisely, $(s,t)$ is exactly the orthofibration classified by the functor 
\[ \Fun^L(-,-)\colon (\PrExCat{L})\op \times \PrExCat{L} \longrightarrow \Cat \]
see \cite[Theorem 7.21]{HaugsengHebestreitLinskensNuiten}. We note that if we restrict this bifibration in each variable to the subcategory $\Pr^{\cont}_{\st} \subseteq \PrExCat{L}$ containing all objects and just the strongly continuous morphisms in $\PrExCat{L}$, i.e., those whose right adjoint is also colimit preserving, then the covariant dependence in the second entry is also contravariant (by post-composing with the right adjoint), and the contravariant dependence in the first entry is 
also covariant (bu pre-composing with the right adjoint). 

The base change of $(s,t)$ to $\Pr^{\cont}_{\st} \times \Pr^{\cont}_{\st}$ is then not just an orthofibration, but also a cartesian and a cocartesian fibration. Such strongly continuous morphisms arise, for example, from any exact functor $f\colon \Ccal \to \Dcal$ between small stable categories upon taking Ind-completion. In particular, the base change of $(s,t)$ along $(\Ind,\Ind)$ is a model for the cartesian and cocartesian fibration $P \to \CatEx \times \CatEx$ classified by the functor $(\Ccal,\Dcal) \mapsto \Fun^L(\Ind(\Ccal),\Ind(\Dcal))$, and we get that the category $\CatL$ of laced categories sits in a pullback square of the form
\begin{equation}\label{eq:categorical-lace}
    \begin{tikzcd}[column sep=50pt]
        \CatL\ar[d]\ar[r] & \Ar^{\opl}(\PrExCat{L}) \ar[d] \\
        \CatEx \ar[r, "{(\Ind,\Ind)}"] & \PrExCat{L} \times \PrExCat{L} .
    \end{tikzcd}
\end{equation}
This yields a construction of $\CatL$ as a fibre product, rather than via unstraightening. We point out the resemblance between the squares~\eqref{eq:categorical-lace} and~\eqref{eq:lace}, or rather, \eqref{eq:lace} can be considered as a once decategorified version of~\eqref{eq:categorical-lace}.

We now exploit the above description of $\CatL$ to construct a left adjoint to $\Lace$. For this consider the natural transformation of diagrams
\begin{equation}\label{eq:transformation-of-diagrams}
    \begin{tikzcd}[row sep = 10pt]
        & \PrExCat{L} \ar[dd, "\id"] & & & \Ar^{\opl}(\PrExCat{L}) \ar[dd, "{(s, t)}"] \\
        && \Rightarrow && \\
        \CatEx \ar[r] & \PrExCat{L} & & \CatEx \ar[r] & \PrExCat{L} \times \PrExCat{L}
    \end{tikzcd}
\end{equation}
which implements the diagonal map everywhere, that is, on the top right corner it is given by restriction along $\Delta^1 \to \Delta^0$, on the bottom right corner by restriction along $\Delta^0 \coprod \Delta^0 \to \Delta^0$, and on the bottom left corner it is just the identity. We then write
$$ 
    \begin{tikzcd}
        \L \colon \CatEx\simeq\CatEx \times_{\PrExCat{L}} \PrExCat{L} \arrow[r]& \CatEx \times_{(\PrExCat{L} \times \PrExCat{L})}\Ar^{\opl}(\PrExCat{L}) \simeq\CatL 
    \end{tikzcd}
$$
for the induced functor on fibre products. Unwinding the definitions, $\L$ is given by the formula $\Ccal \mapsto (\Ccal,\id)$ (or, if we consider bimodules as biexact functors, by the formula $\Ccal \mapsto (\Ccal,\map_{\Ccal})$).

\begin{prop}\label{prop:left-adj-to-lace}
The functor $\L$ defined above is left adjoint to the functor $\Lace$ of Construction~\ref{cons:laced-objects}.
\end{prop}
\begin{proof}
Each component of the natural transformation~\eqref{eq:transformation-of-diagrams} admits a right adjoint in the form of the corresponding oplax limit (indexed over $\Delta^1$, $\Delta^0 \coprod \Delta^0$ or $\Delta^0$); indeed, this is exactly the defining property of oplax limits, where we note that for $\Delta^0$ and $\Delta^0 \coprod \Delta^0$ there is no difference between oplax limits and limits (or between oplax natural transformations and natural transformations). To avoid confusion, let us note that these right adjoints do not assemble to a natural transformation in general, only to a \emph{lax} transformation 
\[
\begin{tikzcd}[row sep = 10pt]
& \PrExCat{L} \ar[dd] & & & \Ar^{\opl}(\PrExCat{L}) \ar[dd] \\
&& \stackrel{\lax}{\Leftarrow} && \\
\CatEx \ar[r] & \PrExCat{L} & & \CatEx \ar[r] & \PrExCat{L} \times \PrExCat{L} ,
\end{tikzcd}
\]
which, in turn, yields an induced right adjoint on the level of lax fibre products
\[ 
\begin{tikzcd}
\CatEx \times^{\lax}_{\PrExCat{L}} \PrExCat{L} & \ar[l, "\Psi"'] \CatEx \times^{\lax}_{\PrExCat{L} \times \PrExCat{L}} \Ar^{\opl}(\PrExCat{L}).
\end{tikzcd}
\] 
Here, the left hand side can be described as the category of diagrams
\[ 
\begin{tikzcd}
& \Ecal \ar[d, "f"] \\
\Ccal \ar[r, "g"] & \Dcal
\end{tikzcd}
\]
of stable categories and exact functors such that in addition $\Ccal$ is small, $\Ecal$ and $\Dcal$ are presentable, and $f$ is colimit preserving. Let $\Zcal \subseteq \CatEx \times^{\lax}_{\PrExCat{L}} \PrExCat{L}$ be the full subcategory spanned by those diagrams as above whose limit (that is, the fibre product of $f$ and $g$) is small. Then the composed functor 
\[ 
\begin{tikzcd}
\CatEx \times^{\lax}_{\PrExCat{L}} \PrExCat{L} & \ar[l, "\Psi"'] \CatEx \times^{\lax}_{\PrExCat{L} \times \PrExCat{L}} \Ar^{\opl}(\PrExCat{L}) & \ar[l] \CatEx \times_{\PrExCat{L} \times \PrExCat{L}} \Ar^{\opl}(\PrExCat{L}) \ar[ll,"\Psi'"',bend right = 20]
\end{tikzcd}
\] 
takes values in $\Zcal$, since presentable categories are locally small. At the same time, the functor
\[
\begin{tikzcd}
\CatEx = \CatEx \times_{\PrExCat{L}} \PrExCat{L} \ar[r] & \CatEx \times^{\lax}_{\PrExCat{L}} \PrExCat{L}
\end{tikzcd}
\]
also takes values in $\Zcal$, and the inclusion $\CatEx \hookrightarrow \Zcal$ admits a right adjoint $\Phi\colon \Zcal \to \CatEx$ given by realizing the fibre product. We hence conclude that the induced functor
\[
\begin{tikzcd}
\CatEx = \CatEx \times_{\PrExCat{L}} \PrExCat{L} \ar[r, "\L"] & \CatEx \times_{\PrExCat{L} \times \PrExCat{L}} \Ar^{\opl}(\PrExCat{L}) = \CatL
\end{tikzcd}
\]
admits a right adjoint by composing $\Psi'$ and $\Phi$. Unwinding the definitions, for a laced category $(\Ccal,F\colon \Ind(\Ccal) \to \Ind(\Ccal))$, the stable category $\Phi(\Psi'(\Ccal,F))$ naturally sits in a diagram 
\[ 
\begin{tikzcd}
\Phi(\Psi'(\Ccal,F))\ar[d]\ar[r] & \lim^{\opl}(F) \ar[r]\ar[d] & \Ind(\Ccal)^{\Delta^1} \ar[d] \\
\Ccal \ar[r, "{(j,j)}"] & \Ind(\Ccal) \times \Ind(\Ccal) \ar[r, "\id \times F"] & \Ind(\Ccal)^{\Delta^{\{0\}}} \times \Ind(\Ccal)^{\Delta^{\{1\}}}
\end{tikzcd}
\]
in which both squares are pullback, and so we obtain an identification $\Phi(\Psi(\Ccal,F)) = \Lace(\Ccal,F)$.
\end{proof}

Finally, let us conclude this section by the following result:

\begin{thm} \label{PresentableTCatEx}
    The category $\CatL$ is presentable and generated under colimits by the  
    objects $(\Spfin, 0)$ and $(\Spfin, \id)$, which are furthermore compact.
\end{thm}
\begin{proof}
    We begin by checking that the proposed generators are indeed compact. The object $(\Spfin, 0)$ corepresents the composite of $\fgt\colon\CatL\to\CatEx$ and $\core\colon\CatEx\to\Spaces$, the first is colimit preserving by Proposition~\ref{prop:limits-and-colimits} and the second is a composite of the forgetful functor $\CatEx \to \Cat$, which preserves filtered colimits by \cite[Proposition 1.1.4.6]{HA}, and the functor $\Cat \to \Spaces$ represented by the compact object $\ast \in \Cat$; therefore $(\Spfin, 0)$ is compact. 
    
    For $(\Spfin, \id)$, we have to show that $\LaceEq$ commutes with filtered colimits. Since we already argued that $\core$ is filtered colimits preserving it will suffice to show that $\Lace\colon \CatEx \to \Cat$ is so. Recall from Construction \ref{cons:laced-objects} that there is a pullback square of spaces
    $$
        \begin{tikzcd}
            \Lace(\Ccal, M)\arrow[r]\arrow[d] & \int^{X \in\Ccal}_{Y\in \Ccal}\OmegaInfty M(X, Y)\arrow[d, "p"] \\
            \Ccal\arrow[r, "\Delta"] & \Ccal\times\Ccal
        \end{tikzcd}
    $$
    Since pullbacks of commute with filtered colimits in $\Cat$, it suffices to show the other three terms in the above square preserve filtered colimits as functors in $\Ccal$. For the bottom two terms this is provided by \cite[Proposition 1.1.4.6]{HA}. 
    As for the top right corner, recall that there is a commutative square
    $$
        \begin{tikzcd}
            \CatL\arrow[r]\arrow[d] & \int^{\Ical \in\Cat} \Fun(\Ical\op\times \Ical, \Spaces)\arrow[d] \\
            \CatEx\arrow[r] & \cat{Cat} \ ,
        \end{tikzcd}
    $$
    such that the vertical arrows are cocartesian fibrations and the top horizontal arrow, given by postcomposition by the loop finitary $\OmegaInfty\colon\Sp\to\Spaces$, preserves cocartesian arrows (since these are computed via a left Kan extension involving filtered comma categories and $\OmegaInfty$ preserves filtered colimits). Now the lower horizontal map preserves filtered colimits by \cite[Proposition 1.1.4.6]{HA} and the induced map on vertical fibres preserves filtered colimits, so that the bottom horizontal functor 
    $$
        \CatL\to\int^{\Ical\in\Cat} \Fun(\Ical\times \Ical\op, \Spaces) 
    $$ 
    also preserves filtered colimits. But this functor precisely sends $(\Ccal, M)$ to $(\Ccal, \OmegaInfty M\colon \Ccal\op\times\Ccal\longrightarrow\Spaces)$. Via unstraightening, this datum is equivalently encoded via the fibration $p\colon \int^{X \in \Ccal}_{Y \in \Ccal}\OmegaInfty M(X, Y)\to\Ccal\times\Ccal$ so that we have the wanted claim. \\
    
    We recall that a cocomplete category $\Ccal$ is generated under colimits by a collection of compact objects $\{X_i\}$ if $\Map(X_i, -)\colon\Ccal\to\Spaces$ jointly detect equivalences by \cite[Corollary 2.5]{Yanovski} and the remark that follows; moreover, if the $X_i$ form a set, then $\Ccal$ is presentable and compactly generated.

    It is a folklore result that $\CatEx$ is compactly generated by $\Spfin$, a proof of which is recorded in \cite{KrauseNikolausPuetzstueckSheavesOnManifolds}. 
    Therefore, given a laced functor $(f, \eta)\colon(\Ccal, M)\to(\Dcal, N)$ such that $\Lace(f, \eta)$ and $f$ induce equivalences of cores, then both $f$ and $\Lace(f, \eta)$ are equivalences of stable categories. But given the pullback square of Construction \ref{cons:laced-objects}, the fibre at $X\in\Ccal$ of
    $$
        \begin{tikzcd}
            \Fun_{\CatL}((\Spfin, \id), (\Ccal, M)) = \LaceEq(\Ccal, M)\arrow[r] & \Ccal^{\simeq} = \Fun_{\CatL}((\Spfin, 0), (\Ccal, M))
        \end{tikzcd}
    $$
    is precisely the space $\OmegaInfty M(X, X)$. The above argument then implies that $\eta\colon M\to N\circ(\catop{f}\times f)$ induces an equivalence on the diagonal. Since $M(X, Y)$ is a retract of $M(X\oplus Y, X\oplus Y)$, we deduce that $\eta$ is an equivalence at every point, which concludes.
\end{proof}

\subsection{Laced categories as the tangent bundle of \texorpdfstring{$\CatEx$}{CatEx}}
\label{subsec:tangent}%

Our goal in the present subsection is to prove that $\CatL$ is a model for the tangent bundle of $\CatEx$. Recall the functor $\L\colon \CatEx \to \CatL$ of Proposition~\ref{prop:left-adj-to-lace}. By its construction, it fits into a commutative square
\[
\begin{tikzcd}
\CatEx \ar[r,"\L"]\ar[d,equal] &\CatL \ar[d,"\pi"] \\
\CatEx \ar[r,equal] & \CatEx,
\end{tikzcd}
\]
which we may view as a natural transformation between the vertical arrows in $\CAT$. Since the individual components of this transformation are left adjoints, \cite[Theorem 4.6]{HaugsengLax} tells us that this transformation is itself a left adjoint when considered as a lax natural transformation, that is, as a morphism 
$\Ar^{\lax}(\CAT) = \Fun^{\lax}(\Delta^1,\CAT)$. In particular, its right adjoint will be a 
a lax natural transformation of the form
\[
    \begin{tikzcd}
        \CatEx \ar[d,equal] & \CatL \ar[dl,Rightarrow, shorten >= 2.8ex, shorten <= 2.8ex]\ar[l,"\Lace"']\ar[d,"\pi"] \\
        \CatEx \ar[r,equal] & \CatEx.
    \end{tikzcd}
\]
Explicitly, the 2-cell in the above square is given by the canonical projection $\Lace(\Ccal,M) \to \Ccal$ sending a laced object $(X,\alpha)$ to its underlying object $X$.
Passing to lax limits on both vertical arrows we obtain an induced adjunction
\begin{equation}\label{eq:first-adj}
    \begin{tikzcd}
        (\CatEx)^{\Delta^1} \arrow[r, shift left=2] & \arrow[l, shift left=2, "\perp"'] \CatL \times_{(\CatEx)^{\Delta^{\{0\}}}} (\CatEx)^{\Delta^1} 
    \end{tikzcd}
\end{equation}
whose left adjoint sends an arrow $f\colon \Ccal \to \Dcal$ to the tuple $((\Ccal,\id),f)$ and the right adjoint sends a tuple $((\Ccal,M),f\colon \Ccal \to \Dcal)$ to the composed arrow $\Lace(\Ccal,M) \to \Ccal \to \Dcal$. Let us now point out that the projection $\pi\colon \CatL \to \CatEx$ is not only a cartesian fibration, but also a cocartesian fibration, where for an exact functor $f\colon \Ccal \to \Dcal$ the cocartesian monodromy $\BiMod(\Ccal) \to \BiMod(\Dcal)$, which is left adjoint to the cartesian monodromy $\BiMod(\Dcal) \to \BiMod(\Ccal)$, is given here by left Kan extension along $f\op \times f\colon \Ccal\op \times \Ccal \to \Dcal\op \times \Dcal$. By~\cite[Lemma 2.27]{AyalaFrancis} we consequently have an adjunction
\begin{equation}\label{eq:second-adj}
    \begin{tikzcd}
        \CatL \times_{(\CatEx)^{\Delta^{\{0\}}}} (\CatEx)^{\Delta^1} \arrow[r, shift left=2] & \arrow[l, shift left=2, "\perp"'] \CatL
    \end{tikzcd}
\end{equation}
where the right adjoint is given by $(\Ccal,M) \mapsto ((\Ccal,M),\id_{\Ccal})$ and the left adjoint 
sends a pair $((\Ccal,M),f\colon \Ccal \to \Dcal)$ to the image $(f\op \times f)_!M$ of $M \in \BiMod(\Ccal)$ under the cocartesian monodromy along $f$. Composing~\eqref{eq:first-adj} and~\eqref{eq:second-adj} we hence obtain an adjunction
$$
    \begin{tikzcd}
          \widetilde{\L}\colon (\CatEx)^{\Delta^1}\arrow[r, shift left=2] & \arrow[l, shift left=2, "\perp"']\CatL \cocolon \widetilde{\Lace}
    \end{tikzcd}
$$
where the left adjoint $\widetilde{\L}$ sends an arrow $f\colon \Ccal \to \Dcal$ to the image of $\L(\Ccal) = (\Ccal,\id)$ under the cocartesian monodromy $(f\op \times f)_!\colon \BiMod(\Ccal) \to \BiMod(\Dcal)$, while the right adjoint $\widetilde{\Lace}$ sends a laced category $(\Ccal,M)$ to the exact functor $\Lace(\Ccal,M) \to \Ccal$ given by $(X,\alpha) \mapsto X$. 

\begin{prop}
    The adjunction $\widetilde{\L} \dashv \widetilde{\Lace}$ fits into a commutative diagram 
    \begin{equation}\label{eq:lace-to-ar}
        \begin{tikzcd}
            (\CatEx)^{\Delta^1} \ar[rr,phantom,"\perp"] \ar[rr, "\widetilde{\L}", bend left = 10]\ar[dr,"t"'] &&  \ar[dl,"\pi"]\ar[ll,bend left =10,"\widetilde{\Lace}"]\CatL \\
            & \CatEx & 
        \end{tikzcd}
    \end{equation}
    that is, both adjoints commute with the projection to $\CatEx$. In addition, this is an adjunction relative to $\CatEx$ in the sense of~\cite[Definition 7.3.2.2]{HA}, the right adjoint $\widetilde{\Lace}$ preserves cartesian edges and the left adjoint $\widetilde{\L}$ preserves cocartesian edges.
\end{prop}
\begin{proof}
    We first note that once it is established that $\widetilde{\L} \dashv \widetilde{\Lace}$ is a relative adjunction, the preservation of cocartesian edges by $\widetilde{\L}$ is equivalent to the preservation of cartesian edges by $\widetilde{\Lace}$.
    
    Indeed, write $\Mcal \to \Delta^1$ for the cartesian and cocartesian fibration classifying this adjunction. The fact that the adjunction is relative to $\CatEx$ implies that we have a well-defined projection $\rho\colon \Mcal \to \CatEx \times \Delta^1$ preserving both cartesian and cocartesian edges over $\Delta^1$. Since the restriction of $\rho$ to $\CatEx \times (\Delta^{\{0\}} \coprod \Delta^{\{1\}})$ is both a cartesian and cocartesian fibration, one concludes that $\rho$ itself is both a locally cartesian and a locally cocartesian fibration. The condition that $\widetilde{\L}$ preserves cocartesian edges is then equivalent to this fibration being cocartesian, while the condition that $\widetilde{\Lace}$ preserves cartesian edges is equivalent to this fibration being cartesian. Both of these are equivalent in this case to the fibration being exponentiable, see~\cite[Proposition 2.23]{AyalaFrancis}. 

    Now since the adjunction $\widetilde{\L} \dashv \widetilde{\Lace}$ was constructed as a composite of two adjunctions, it will suffice to show that each one of these refines to an adjunction relative to $\CatEx$ in which the left adjoint preserves cocartesian edges.     
    For the adjunction~\eqref{eq:first-adj}, let us simply point out that it is obtained from an adjunction in $\Ar^{\lax}(\CAT)$ upon passing to lax limits. There is a hence a natural transformation of adjunctions from~\eqref{eq:first-adj} to the image of this adjunction in $\Ar^{\lax}(\CAT)$ under the target projection $\Ar^{\lax}(\CAT) \to \CAT$. But this image is just the identity adjunction on $\CatEx$, and so we conclude that the adjunction~\eqref{eq:first-adj} refines to a relative adjunction
    \[
        \begin{tikzcd}
            (\CatEx)^{\Delta^1} \ar[rr,phantom,"\perp"] \ar[rr, bend left = 5]\ar[dr,"t"'] &&  \ar[dl,"\tau"]\ar[ll,bend left =5]\CatL \times_{(\CatEx)^{\Delta^{\{0\}}}} (\CatEx)^{\Delta^1}\\
            & \CatEx & 
        \end{tikzcd}
    \]
    where the left diagonal arrow is the target projection and the right diagonal arrow is composite $\CatL \times_{(\CatEx)^{\Delta^{\{0\}}}} (\CatEx)^{\Delta^1}\to (\CatEx)^{\Delta^1} \xrightarrow{t} \CatEx$. Let us also point out that both sides are (free) cocartesian fibrations where the cocartesian edges on the left hand side are the arrows which are sent to equivalences by the source projection $(\CatEx)^{\Delta^1} \to (\CatEx)^{\Delta^{\{0\}}}$, while the the cocartesian arrows on the right hand side are those which are sent to equivalences by the projection to $\CatL$. By construction these projections intertwine the left adjoint above with $\L$, and hence we see that this left adjoint preserves cocartesian edges. 
    
    Now consider the adjunction~\eqref{eq:second-adj}. As the right adjoint in~\eqref{eq:second-adj} visibly preserves the projection to $\CatEx$, by~\cite[Proposition 7.3.2.1]{HA} it will suffice to show that the components of the unit map are sent to equivalences in $\CatEx$. Indeed, this is a general fact about adjunctions arising from cocartesian fibrations in this manner, see~\cite[Lemma 2.27]{AyalaFrancis}. 
    Finally, the last part of \cite[Lemma 2.27]{AyalaFrancis} also stipulates that the left adjoint in question preserves cocartesian edges over $\CatEx$, and so the proof is complete.
\end{proof}

We now come to the main result of the present subsection.
\begin{prop}\label{prop:lace-is-tangent-bundle}
The diagram~\eqref{eq:lace-to-ar} exhibits $\CatL$ as the stable envelope of $(\CatEx)^{\Delta^1}$ over $\CatEx$ in the sense of \cite[Definition 7.3.1.1]{HA}. 
In particular, it exhibits $\CatL$ as the tangent bundle of $\CatEx$.
\end{prop}

The remainder of this subsection is devoted to the proof of Proposition~\ref{prop:lace-is-tangent-bundle}. For this, note first that since $\widetilde{\Lace}$ is a right adjoint and preserves cartesian edges the statement of Proposition~\ref{prop:lace-is-tangent-bundle} is local: we need to check that for every $\Ccal$, the induced functor $\Lace(\Ccal,-)\colon M \to [\Lace(\Ccal,M) \to \Ccal]$ exhibits $\BiMod(\Ccal)$ as the stabilization of $\CatEx_{/\Ccal}$. Since $\BiMod(\Ccal)$ is stable, this is equivalent to saying that the induced functor 
\[ \Lace^{\Sp}_{\Ccal}\colon \BiMod(\Ccal) \to \Sp(\CatEx_{/\Ccal}) \]
is an equivalence. 

\begin{defi}\label{defi:Q}
For $\Ccal \in \CatEx$, let us write 
\[ Q_{\Ccal}\colon \CatEx_{\Ccal//\Ccal} \to \BiMod(\Ccal) \]
for the functor sending a retract diagram $\Ccal \xrightarrow{f} \Dcal \xrightarrow{r} \Ccal$ to the bimodule 
\[ \Ccal\op \times \Ccal \to \Sp \quad\quad (X,Y) \mapsto \fib[\map_{\Dcal}(f(X),f(Y)) \xrightarrow{r_*} \map_{\Ccal}(X,Y)]. \]
We also write $Q^{\Sp}_{\Ccal}$ for the composite
\[ Q^{\Sp}_{\Ccal}\colon \Sp(\CatEx_{/\Ccal}) = \Sp(\CatEx_{\Ccal//\Ccal})  \xrightarrow{\Omega^{\infty}} \CatEx_{\Ccal//\Ccal} \xrightarrow{Q_{\Ccal}} \BiMod(\Ccal).\] 
\end{defi}

\begin{rmq}\label{rmq:Q}
In Definition~\ref{defi:Q}, under the identification $\BiMod(\Ccal) = \Fun^L(\Ind(\Ccal),\Ind(\Ccal))$, the $\Ccal$-bimodule $\map_{\Dcal} \circ (f\op \times f)$ corresponds to the colimit preserving endofunctor 
\[ \Ind(f)^{\ad} \circ \Ind(f)\colon \Ind(\Ccal) \to \Ind(\Ccal),\] 
where $\Ind(f)^{\ad}\colon \Ind(\Dcal) \to \Ind(\Ccal)$ is the right adjoint of $\Ind(f)$, and hence the $\Ccal$-bimodule
$Q_{\Ccal}(\Ccal \xrightarrow{f} \Dcal \xrightarrow{r} \Ccal)$ corresponds to the endofunctor $\fib[\Ind(f)^{\ad}\circ \Ind(f) \Rightarrow \id]$.
\end{rmq}

\begin{prop}\label{prop:Q-is-equiv}
For every stable category $\Ccal$ the functor $Q^{\Sp}_{\Ccal}$ of Definition~\ref{defi:Q} above is an equivalence.
\end{prop}

\begin{proof}[Proof of Proposition~\ref{prop:lace-is-tangent-bundle} given Proposition~\ref{prop:Q-is-equiv}]
Unwinding the definitions, for every stable category $\Ccal$ the composite
\[ Q^{\Sp}_{\Ccal}\circ \Lace^{\Sp}_{\Ccal}\colon \BiMod(\Ccal) \to \BiMod(\Ccal) \]
sends a bimodule $M$ to the bimodule 
\[ (X,Y) \longmapsto \fib\left(\map_{\Lace(\Ccal,M)}((X,0),(Y,0)) \longrightarrow \map_{\Ccal}(X,Y)\right) = \Om M(X,Y) ,\]
that is, $Q^{\Sp}_{\Ccal}\circ \Lace^{\Sp}_{\Ccal}$ is naturally equivalent to the loop functor on $\BiMod(\Ccal)$. In particular, this composite is an equivalence. Since $Q^{\Sp}_{\Ccal}$ is an equivalence by Proposition~\ref{prop:Q-is-equiv} it follows that $\Lace^{\Sp}_{\Ccal}$ is an equivalence. Since $\widetilde{\Lace}$ is a right adjoint and preserve cartesian edges it must therefore exhibit $\CatL$ as the stable envelope of $(\CatEx)^{\Delta^1}$, as desired.
\end{proof}

We now come to the proof of Proposition~\ref{prop:Q-is-equiv}. It will make use of the following lemma, which is an adaptation of~\cite[Corollary 2.4.9]{HarpazNuitenPrasmaTanModCat}.

\begin{lmm}\label{lmm:HNP-lemma}
    Let $L\colon \Ccal \adj \Dcal\colon R$ be an adjunction between pointed categories with finite limits such that the unit $u$ and counit $\nu$ both become equivalences after looping finitely many times (that is, there exists an $n$ such that $\Omega^nu$ and $\Omega^n\nu$ are natural equivalences). Then, the functor induced by $R$ on the stabilisations:
    $$
        \begin{tikzcd}
            \Sp(R)\colon\Sp(\Ccal)\longrightarrow\Sp(\Dcal) 
        \end{tikzcd}
    $$
    is an equivalence.
\end{lmm}
\begin{proof}
    Recall that if $\Ccal$ is a pointed category with finite limits, then, $\Sp(\Ccal)$ is also given by following limit of categories:
    $$
        \begin{tikzcd}
            \Sp(\Ccal):=\lim(...\arrow[r, "\Omega"] & \Ccal \arrow[r, "\Omega"] & \Ccal \arrow[r, "\Omega"] & \Ccal)
        \end{tikzcd}
    $$
    Since $R$ preserves limits, it induces a map of pro-objects, and so a map at the limit $\Sp(R)\colon\Sp(\Dcal)\to\Sp(\Ccal)$. The hypothesis on $L$ and $R$ can be reformulated to say that the following diagram commutes
    $$
        \begin{tikzcd}[row sep=large, column sep=large]
            \Dcal\arrow[r, "R"]\arrow[d, "\Omega^n"'] & \Ccal\arrow[d, "\Omega^n"]\arrow[dl, "\Omega^nL"']\\
            \Dcal\arrow[r, "R"'] & \Ccal
        \end{tikzcd}
    $$
    where the triangular fillers are provided by the unit and the counit of the adjunction, and the filling of the large square is also given by the homotopy witnessing that $R$ commutes with $\Omega^n$. In consequence, $R$ induces an equivalence of pro-object, and thus again an equivalence after taking the limit. 
\end{proof}

\begin{rmq}
    Note that the above proof also shows that $\hat{L}:=\Omega^n L$ induces a map of pro-objects and so passing to the limit, a functor $\hat{L}:\Sp(\Ccal)\to\Sp(\Dcal)$. In particular, the inverse of $\Sp(R)$ is given by $\Sp(L):=\Sigma^n\hat{L}$; under the equivalence $\Sp(\Ccal)\simeq\Exc_*(\SpacesFinStar, \Ccal)$ and suitable hypotheses on $\Ccal, \Dcal$, the explicit formula to compute excisive approximations in Goodwillie calculus shows that $\Sp(L)$ unravels to be
    $$
        \begin{tikzcd}
            \Exc_*(\SpacesFinStar, \Ccal)\arrow[r, "L_*"] & \Fun_*(\SpacesFinStar, \Dcal)\arrow[r, "\nexcPart{1}(-)"] & \Exc_*(\SpacesFinStar, \Dcal)
        \end{tikzcd}
    $$
    with $\nexcPart{1}(-)$ being the first Goodwillie derivative, i.e. the left adjoint to the inclusion.
\end{rmq}

\begin{proof}[Proof of Proposition~\ref{prop:Q-is-equiv}]
Given an exact functor $f\colon \Ccal \to \Dcal$ we write $\im^{\st}(f)$ for the smallest full subcategory of $\Dcal$ containing $f(\Ccal)$ and closed under finite limits and colimits, and refer to it as the \emph{stable image} of $f$. We say that $f$ is \emph{stably surjective} if $\im^{\st}(f) = \Dcal$. The collection of stably surjective functors is closed under composition, and we write $\Cat^{\st-\surj} \subseteq \CatEx$ for the wide subcategory consisting of the stably surjective functors. Observe that the collection of stably surjective functors is part of an orthogonal factorization system on $\CatEx$, whose right class is that of fully-faithful functors. In other words, in any commutative square
\[
    \begin{tikzcd}
        \Acal \ar[r]\ar[d] & \Bcal \ar[d] \\ 
        \Ccal \ar[r]\ar[ur,dotted] & \Dcal
    \end{tikzcd}
\]
in $\CatEx$ whose left vertical arrow is stably surjective and whose right vertical arrow is fully-faithful the space of dotted lifts is contractible, and in addition every morphism $\Ccal \to \Dcal$ admits a factorization $\Ccal \to \Ecal \to \Dcal$ into a stably surjective functor followed by a fully-faithful one. This factorization is then automatically essentially unique, and we see that essentially the only option is taking $\Ecal = \im^{\st}(f)$. 

For a given stable category $\Ccal$, it follows from the formalism of orthogonal factorization systems that the induced functor
\[ \Cat^{\st-\surj}_{\Ccal//\Ccal} \to \CatEx_{\Ccal//\Ccal} \]
is fully-faithful and admits a right adjoint
$\CatEx_{\Ccal//\Ccal} \to \Cat^{\st-\surj}_{\Ccal//\Ccal}$ given by
\[ [\Ccal \xrightarrow{f} \Dcal \to \Ccal] \mapsto [\Ccal \to \im^{\st}(f) \to \Ccal] .\]
Since the inclusion $\im^{\st}(i) \hookrightarrow \Dcal$ is fully-faithful, we have that for every $[\Ccal \xrightarrow{i} \Dcal \xrightarrow{r} \Ccal] \in \CatEx_{\Ccal//\Ccal}$ the induced functor 
\[ \Ccal \times_{\im^{\st}(i)} \Ccal \to \Ccal \times_{\Dcal} \Ccal \]
is an equivalence. It hence follows that the counit of the adjunction $\Cat^{\st-\surj}_{\Ccal//\Ccal} \adj \CatEx_{\Ccal//\Ccal}$ becomes an equivalence after looping one time. At the same time, the unit of this adjunction is an equivalence (since the left adjoint is the inclusion of $\Cat^{\st-\surj}_{\Ccal//\Ccal}$), and hence by Lemma~\ref{lmm:HNP-lemma} the right-adjoint induces an equivalence
\[ \Sp(\Cat^{\st-\surj}_{\Ccal//\Ccal}) \xrightarrow{\simeq} \Sp(\CatEx_{\Ccal//\Ccal}) \]
Finally, let us point out that the functor $Q_{\Ccal} \colon \CatEx_{\Ccal//\Ccal} \to \BiMod(\Ccal)$ sends arrows in $\CatEx_{\Ccal//\Ccal}$ whose middle component is fully-faithful to equivalences, and hence factors essentially uniquely as a composite
\[ \CatEx_{\Ccal//\Ccal} \to \Cat^{\st-\surj}_{\Ccal//\Ccal} \to \BiMod(\Ccal) ,\]
where the second map is just the restriction of $Q_{\Ccal}$ to $\Cat^{\st-\surj}_{\Ccal//\Ccal}$, which we will also denote by $Q_{\Ccal}$. It will hence suffice to show that the composite
\[ \Sp(\Cat^{\st-\surj}_{\Ccal//\Ccal}) \xrightarrow{\Om^{\infty}} \Cat^{\st-\surj}_{\Ccal//\Ccal} \xrightarrow{Q_{\Ccal}} \BiMod(\Ccal)\]
is an equivalence.

By the Lurie-Barr-Beck theorem \cite[Theorem 4.7.3.5]{HA},
for a given stable presentable category $\Ccal$, the operation that associates to a colimit-preserving monad $T \colon \Ccal \to \Ccal$ its category of $T$-algebras yields a fully faithful functor
\begin{align*}
    \begin{tikzcd}[ampersand replacement=\&]
        \Alg(\FunL(\Ccal,\Ccal)) \arrow[r] \& (\PrExCat{L})_{\Ccal/}
    \end{tikzcd} &&
    \begin{tikzcd}[ampersand replacement=\&]
        T \arrow[r, mapsto] \&{[\Ccal \xrightarrow{\Free_T} \Alg_T(\Ccal)]}
    \end{tikzcd}
\end{align*}
whose essential image consists of those left functors $\Ccal \to \Dcal$ among stable presentable categories whose right adjoint $\Dcal \to \Ccal$ is colimit-preserving and conservative. In other words, those whose right adjoint preserves and detects colimits. We note that if $\Ccal \xrightarrow{G} \Dcal \xrightarrow{G'} \Ecal$ are functors between cocomplete categories such that $G'$ preserves and detects colimits then $G$ preserves and detects colimits if and only if $G' \circ G$ preserves and detects colimits. The formation of $T$-algebras as above can hence be rewritten as an equivalence
\begin{align*}
    \begin{tikzcd}[ampersand replacement=\&]
        \Alg(\FunL(\Ccal,\Ccal)) \arrow[r, "\simeq"] \& (\Pr^{\mon}_{\st})_{\Ccal/}
    \end{tikzcd} &&
    \begin{tikzcd}[ampersand replacement=\&]
        T \arrow[r, mapsto] \&{[\Ccal \xrightarrow{\Free_T} \Alg_T(\Ccal)]}
    \end{tikzcd}
\end{align*}
where $\Pr^{\mon}_{\st}$ denotes the category whose objects are the stable presentable categories and whose morphisms are the left functors $\Ccal \to \Dcal$ whose right adjoints preserve and detect colimits. 

We note that if $\Ccal$ is compactly generated and $F\colon \Ccal \to \Dcal$ is a left adjoint functor whose right adjoint $G\colon \Dcal \to \Ccal$ preserves and detects colimits then $\Dcal$ is automatically also compactly generated and $F$ automatically preserves compact objects. On the other hand, if $f\colon \Ccal \to \Dcal$ is an exact functor between stable categories then $\Ind(f)\colon \Ind(\Ccal) \to \Ind(\Dcal)$ always has a colimit preserving right adjoint $\Ind(\Dcal) \to \Ind(\Ccal)$, and this right adjoint is conservative if and only if the stable image $\im^{\st}(f) \subseteq \Dcal$ in the above sense is dense in $\Dcal$. We conclude that if $\Ccal$ is a small stable category then associating to every colimit preserving monad on $\Ind(\Ccal)$ the stable image of the composite $\Ccal \to \Ind(\Ccal) \xrightarrow{\Free_T} \Alg_T(\Ind(\Ccal))$ yields an equivalence
\begin{align*}
    \begin{tikzcd}[ampersand replacement=\&]
        \Alg(\FunL(\Ind(\Ccal),\Ind(\Ccal))) \arrow[r, "\simeq"] \& (\Cat^{\st-\surj})_{\Ccal/}
    \end{tikzcd} &
    \begin{tikzcd}[ampersand replacement=\&]
        T \arrow[r, mapsto] \&{[\Ccal \to \im^{\st}_{\Free_T|_{\Ccal}} \subseteq \Alg_T(\Ind(\Ccal))]}
    \end{tikzcd}
\end{align*}
and hence an equivalence
\begin{align*}
    \begin{tikzcd}[ampersand replacement=\&]
        \Alg(\FunL(\Ind(\Ccal),\Ind(\Ccal)))_{/\id} \arrow[r, "\simeq"] \& (\Cat^{\st-\surj})_{\Ccal//\Ccal}
    \end{tikzcd} &
    \begin{tikzcd}[ampersand replacement=\&]
        T \arrow[r, mapsto] \&{[\Ccal \to \im^{\st}_{\Free_T|_{\Ccal}} \to \Ccal]}
    \end{tikzcd}
\end{align*}
By construction, the inverse of this equivalence sends $\Ccal \xrightarrow{f} \Dcal \to \Ccal$ to $\Ind(f)^{\ad} \circ \Ind(f) \Rightarrow \id$, where $\Ind(f)^{\ad}\colon \Ind(\Dcal) \to \Ind(\Ccal)$ is the right adjoint of $\Ind(f) \colon \Ind(\Ccal) \to \Ind(\Dcal)$.
Now by~\cite[Theorem 7.3.4.13]{HA} the composite
$$
    \begin{tikzcd}
        \Sp(\Alg(\Fun^L(\Ind(\Ccal),\Ind(\Ccal)))_{/\id}) \arrow[r, "\OmegaInfty"] & \Alg(\Fun^L(\Ind(\Ccal),\Ind(\Ccal)))_{/\id} \arrow[d, "{[p\colon T \Rightarrow \id] \mapsto \fib(p)}"] \\
        &\Fun^L(\Ind(\Ccal),\Ind(\Ccal))
    \end{tikzcd}
$$
is an equivalence. We hence conclude that the composed functor
$$
    \begin{tikzcd}
        \Sp(\Cat^{\st-\surj}_{\Ccal//\Ccal}) \arrow[r, "\OmegaInfty"] & \Cat^{\st-\surj}_{\Ccal//\Ccal} \arrow[r] & \Fun^L(\Ind(\Ccal),\Ind(\Ccal))
    \end{tikzcd}
$$
is an equivalence, where the second functor is given by 
\[ [\Ccal \xrightarrow{f} \Dcal \to \Ccal] \longmapsto \fib\left(\Ind(f)^{\ad} \circ \Ind(f) \Rightarrow \id\right)\] 
Under the identification $\Fun^L(\Ind(\Ccal),\Ind(\Ccal)) = \BiMod(\Ccal)$, this last functor is exactly $Q_{\Ccal}$ (see Remark~\ref{rmq:Q}), and so the proof is complete.
\end{proof}

\subsection{The symmetric monoidal structure on \texorpdfstring{$\CatL$}{CatLace}} 
\label{subsec:symmetric-monoida-tangent}

In this subsection we construct a symmetric monoidal structure on $\CatL$ and study its basic properties. To define it, let us recall a few facts about the Lurie tensor product, see~\cite[\S 4.8.1]{HA}.

Fix a class $\Kcal$ of simplicial sets (considered as ``diagram shapes''). Then the category $\Cat^{\Kcal}$ of $\Kcal$-cocomplete categories, that is, categories which admit $K$-colimits for every $K\in\Kcal$, and $\Kcal$-cocontinuous functors between them admits a tensor product, sometimes called the Lurie tensor product, where for two $\Kcal$-cocomplete categories $\Ccal,\Dcal$ their tensor product $\Ccal \otimes \Dcal$ is the universal recipient of a functor $\Ccal \times \Dcal \to \Ccal \otimes \Dcal$ which preserve $\Kcal$-indexed colimits in each variable separately.

For $\Kcal$ the collection of finite simplicial sets the symmetric monoidal structure on $\Cat^{\Kcal} =: \Cat^{\fin}$ is inherited by the full subcategory $\CatEx \subseteq \Cat^{\fin}$ of stable categories. In fact, the inclusion $\CatEx \subseteq \Cat^{\fin}$ is reflective, with left adjoint $\Cat^{\fin} \to \CatEx$ given by tensoring with the category $\Spfin$ of finite spectra. In particular, it is a smashing localisation and so $\CatEx$ inherits the Lurie tensor product in a manner that makes the localisation functor symmetric monoidal and the inclusion $\CatEx \subseteq \Cat^{\fin}$ a tensor ideal. For two stable categories $\Ccal,\Dcal$ the tensor product $\Ccal \otimes \Dcal$ is then the universal recipient of a biexact functor $\Ccal \times \Dcal \to \Ccal \otimes \Dcal$, and the unit of $\CatEx$ is $\Spfin$.

In the case where $\Kcal$ consists of all small simplicial sets and we consider large categories then Lurie shows that the associated tensor product preserves presentable categories, and so one obtains an induced symmetric monoidal structure on $\PrCat{L}$. As in the case of small categories, the inclusion $\PrExCat{L} \subseteq\PrCat{L}$ of stable presentable categories inside all presentable categories is reflective and the left adjoint $\PrCat{L}\to\PrExCat{L}$, given by $\Ccal\mapsto\Sp(\Ccal) = \mathrm{Sp}\otimes\Ccal$, is again a smashing localisation, so that $\PrExCat{L}$ inherits the Lurie tensor product in such a manner that the localisation $\PrCat{L} \to \PrExCat{L}$ is symmetric monoidal and $\PrExCat{L}$ is an ideal in $\PrCat{L}$. The unit of $\PrExCat{L}$ is $\Sp$. 

Finally, if $\Kcal \subseteq \Kcal'$ is a sequence of two classes of categories then the forgetful functor $\Cat^{\Kcal} \to \Cat^{\Kcal'}$ admits a left adjoint $\Pcal^{\Kcal}_{\Kcal'}\colon \Cat^{\Kcal} \to \Cat^{\Kcal'}$, and this left adjoint canonically refines to a symmetric monoidal functor with respect to the Lurie tensor product, see~\cite[Remark 4.8.1.8]{HA}. In the case where $\Kcal$ consists of finite simplicial sets and $\Kcal'$ of all small simplicial sets, this left adjoint coincides with the Ind-completion functor $\Ind(-)$. This functor sends small categories to presentable ones and $\Spaf$ to $\Sp$, and so fits into a square of symmetric monoidal functors
\[
    \begin{tikzcd}
        \Cat^{\fin}\ar[r, "\Ind"]\ar[d, "(-) \otimes \Spaf"'] & \PrCat{L} \ar[d,"(-) \otimes \Sp"]\\
        \CatEx \ar[r,"\Ind"] & \PrExCat{L}
    \end{tikzcd}
\]
in which the horizontal functors are given by Ind completion and the vertical ones by stabilization.

\begin{cons}\label{cons:sm-structure}
    Recall that the category $\CatL$ sits in a fibre square of the form
    \[
        \begin{tikzcd}
            \CatL \ar[rr]\ar[d] && \Ar^{\opl}(\PrExCat{L}) \ar[d] \\
            \CatEx \ar[rr, "{(\Ind,\Ind)}"] && \PrExCat{L} \times \PrExCat{L}.
        \end{tikzcd}
    \]
    Equipping $\CatEx$, $\PrExCat{L}$ and $\Ar^{\opl}(\PrExCat{L})$ with the symmetric monoidal product pointwise induced by the Lurie tensor product on $\CatEx$ and $\PrExCat{L}$ the bottom right part of this square acquires a symmetric monoidal refinement (see~\cite[Remark 4.8.1.8]{HA}), which then induces a symmetric monoidal structure on $\CatL$ such that the projection $\CatL \to \CatEx$ is symmetric monoidal. In addition, as the natural transformation of diagrams~\eqref{eq:transformation-of-diagrams} is compatible by construction with pointwise Lurie tensor products everywhere, we obtain that the functor $\L$ canonically refines to a symmetric monoidal functor
    \[ \L\colon (\CatEx)^{\otimes} \to (\CatL)^{\otimes} .\]
    Consequently, the unit of $(\CatL)^{\otimes}$ is given by $(\Spaf,\id) = \L(\Spaf)$. 
\end{cons}

\begin{rmq}
    Explicitly, the tensor product of two laced categories $(\Ccal,F)$ and $(\Dcal,G)$ is given by $(\Ccal \otimes \Dcal, F \boxtimes G)$, where $F \boxtimes G$ is the endo-functor of $\Ind(\Ccal \otimes \Dcal) = \Ind(\Ccal) \otimes \Ind(\Dcal)$ induced by $F$ and $G$. If we instead view bimodules as exact functors $F\colon \Ccal\op \otimes \Ccal \to \Sp$ and $G\colon \Dcal\op \otimes \Dcal \to \Sp$, then 
    \[ F \boxtimes G\colon (\catop{\Ccal}\otimes\catop{\Dcal})\otimes(\Ccal\otimes\Dcal) = (\catop{\Ccal}\otimes\Ccal)\otimes(\catop{\Dcal}\otimes\Dcal)\to\Sp\] 
    is induced by the biexact functor
        \[
            \begin{tikzcd}
                (\catop{\Ccal}\otimes\Ccal)\times(\catop{\Dcal}\otimes\Dcal)\arrow[r, "F\times G"]
                & \Sp\times\Sp\arrow[r, "\otimes"] & \Sp ,
            \end{tikzcd}
        \]
    where the last map is the tensor (smash) product of spectra.
\end{rmq}

By the adjunction of Proposition~\ref{prop:left-adj-to-lace}, the unit $(\Spfin, \id)$ of $\CatL$ corepresents the functor $\LaceEq$, where $\core$ denotes the core-groupoid functor $\CatEx\to \Spaces$. We can thus extract from the results of \cite{Nikolaus} another universal property for $\LaceEq$:

\begin{lmm} \label{InitialLaxMonoidalisSpLace}
    $\LaceEq$ is the initial lax-monoidal functor $\CatL\to\Spaces$. Consequently, $\SpLaceEq$ is the initial lax-monoidal functor $\CatL\to\Sp$. 
\end{lmm}
\begin{proof}
    The first part is an application of~\cite[Corollary 6.8]{Nikolaus}. The second part follows from point (4) of Corollary 6.9 of loc.\ cit.
\end{proof}

\begin{rmq}
    To avoid confusion, let us point out that for a fixed stable category $\Ccal$, the stable category $\BiMod(\Ccal) = \Fun^L(\Ind(\Ccal),\Ind(\Ccal))$ admits a (non-symmetric) monoidal structure given by composition of endo-functors. This structure is different and unrelated to the one of Construction~\ref{cons:sm-structure}.
\end{rmq}

Our next goal is to show that the symmetric monoidal structure of Construction~\ref{cons:sm-structure} is closed and give an explicit description of the internal mapping objects. 

\begin{defi}
    Fix $(\Ccal, F),(\Dcal, G)$ two laced categories and let $f,g\colon \Ccal\to\Dcal$ be two exact functors. We define the \textit{spectrum of $(F, G)$-linear natural transformations} from $f$ to $g$ by the formula:
    \[
        \Nat^F_G(f, g):=\Nat(F, (\catop{f}\otimes g)^*G),
    \]
    where the right hand side object is the spectrum given by the enrichment in $\Sp$ of the stable category $\BiMod(\Ccal)$. This is exact in both $f$ and $g$, covariant in $g$ and contravariant in $f$, hence defines a bimodule $\Nat^F_G$ on $\FunEx(\Ccal, \Dcal)$. We denote $\FunInt((\Ccal, F),(\Dcal, G))$ the associated laced category.
\end{defi}

Recall that the symmetric monoidal structure on $\CatEx$ is closed with internal mapping object $\FunEx(-,-)$. In particular, there is an exact, natural, evaluation functor $\ev\colon \Ccal\otimes\FunEx(\Ccal, \Dcal)\to\Dcal$ which enjoys the following universal property: for every stable category $\Ecal$ the composed map
\[ \Map_{\CatEx}(\Ecal,\FunEx(\Ccal,\Dcal)) \to \Map_{\CatEx}(\Ccal \otimes \Ecal,\Ccal \otimes \FunEx(\Ccal,\Dcal)) \xrightarrow{\ev_*} \Map_{\CatEx}(\Ccal \otimes \Ecal,\Dcal) \]
is an equivalence of spaces. In the presence of bimodules $F, G$ on $\Ccal$ and $\Dcal$ respectively, we claim that $\ev$ refines to a laced functor 
\[
    (\ev, \eta_{\ev})\colon (\Ccal, F)\otimes\FunInt((\Ccal, F),(\Dcal, G))\to(\Dcal, G).
\]
To obtain the structure map $\eta_{\ev}\colon F\boxtimes\Nat^F_G\Rightarrow(\catop{\ev}\times\ev)^*G$, consider the tensor-hom adjunction
    \[ L_F := F \otimes (-) \colon \Sp \adj \BiMod(\Ccal)\colon \Nat(F,-) =: R_F \]
    associated to $F$, which arises from the structure of $\BiMod(\Ccal)$ as tensored over $\Sp$. Then for every stable category $\Acal$ we have an induced adjunction
\begin{equation}\label{eq:tensor-hom-comp} 
L_F \circ (-) \colon \FunEx(\Acal,\Sp) \adj \FunEx(\Acal,\BiMod(\Ccal))\colon R_F \circ (-).
\end{equation}
    via post-composition. Let now 
\[ G' \colon \FunEx(\Ccal, \Dcal)\op\otimes\FunEx(\Ccal, \Dcal) \to \BiMod(\Ccal)\] 
be the functor associated to
$(\catop{\ev}\times\ev)^*G \in \BiMod(\Ccal \otimes \FunEx(\Ccal,\Dcal))$ via the currying equivalence
    \[ \BiMod(\Ccal \otimes \FunEx(\Ccal, \Dcal)) = \FunEx(\FunEx(\Ccal, \Dcal)\op \otimes \FunEx(\Ccal, \Dcal),\BiMod(\Ccal)).\] 
explicitly, $G'$ corresponds to the biexact functor $(f,g) \mapsto (f\op \times g)^*G \in \BiMod(\Ccal)$, and we have $R_F \circ G' = \Nat^F_G$ essentially by construction. The counit of $L_F \circ (-) \dashv R_F \circ (-)$ then provides a map of the form $L_F \circ R_F \circ G' \Rightarrow G'$, which after passing back via the currying equivalence gives the map
\[ \eta_{\ev}\colon F\boxtimes\Nat^F_G\Rightarrow(\catop{\ev}\times\ev)^*G \]
we use to define the laced enhancement of $\ev$.

\begin{prop}\label{ClosedMonoidalStructure}
    Let $(\Ccal, F)$ be a laced category. Then for every laced category $(\Ecal, H)$ the composed map
    $$
        \begin{tikzcd} \label{eq:composed-map}
            \Map_{\CatL}((\Ecal,H),\FunInt((\Ccal,F),(\Dcal,G)))\arrow[d] \\
            \Map_{\CatL}((\Ccal,F) \otimes (\Ecal,H),(\Ccal,F) \otimes \FunInt((\Ccal,F),(\Dcal,G)))\arrow[d] \\
            \Map_{\CatL}((\Ccal,F) \otimes (\Ecal,H),(\Dcal,G)) 
        \end{tikzcd}
    $$
    is an equivalence of spaces.
\end{prop}

\begin{cor}\label{cor:internal}
The symmetric monoidal structure on $\CatL$ is closed with internal mapping objects given by $\FunInt(-,-)$. In particular, the association $((\Ccal,F),(\Dcal,G)) \mapsto \FunInt((\Ccal,F),(\Dcal,G))$ canonically upgrades to a functor $\CatL \times \CatL \to \CatL$.
\end{cor}

Since $\L\colon \CatEx \to \CatL$ is symmetric monoidal we also conclude
\begin{cor}
    The category $\CatL$ is tensored, cotensored and enriched over $\CatEx$. The tensor and cotensor structures are given by
    \[ \Ccal \otimes (\Dcal,G) = (\Ccal,\id) \otimes (\Dcal,G) \quad\text{and}\quad (\Dcal,G)^{\Ccal} = \FunInt((\Ccal,\id),(\Dcal,G)),\]
    while the enrichment is given by
    \[ (\Ccal,F),(\Dcal,G) \mapsto \Lace\FunInt((\Ccal,F),(\Dcal,G))\]
\end{cor}

\begin{rmq}
    Since $(\Spfin,\id)$ is the unit of $\CatL$ we have
    \begin{align*} 
        \Map_{\CatL}((\Ccal,F),(\Dcal,G)) & = \Map_{\CatL}((\Spfin,\id),\FunInt((\Ccal,F),(\Dcal,G))) \\
                                          & = \LaceEq\FunInt((\Ccal,F),(\Dcal,G))
    \end{align*}
\end{rmq}

\begin{proof}[Proof of Proposition \ref{ClosedMonoidalStructure}]
    Since $\fgt\colon \CatL\to\CatEx$ is a bicartesian fibration the composed map of~\eqref{eq:composed-map} fits in the top row of a commutative square of spaces 
    \[
        \begin{tikzcd}
            \Map_{\CatL}((\Ecal, H),\FunInt((\Ccal, F), (\Dcal, G)))\arrow[r]\arrow[d] & \Map_{\CatL}((\Ccal, F)\otimes(\Ecal, H), (\Dcal, G))\arrow[d] \\
            \Map_{\CatEx}(\Ecal, \FunEx(\Ccal, \Dcal))\arrow[r, "\simeq"] & \Map_{\CatEx}(\Ccal\otimes\Ecal, \Dcal)
        \end{tikzcd}
    \]
    where the bottom horizontal arrow is an equivalence by the universal property of the underlying exact functor $\ev$.
    Hence, it remains to show that the induced map on vertical fibres is an equivalence.  
    
    Fix an exact functor $\phi\colon \Ecal\to\FunEx(\Ccal, \Dcal)$ and write $\overline{\phi}$ for the associate composite 
    \[\overline{\phi}\colon\Ccal \otimes \Ecal \xrightarrow{\id \otimes \phi} \Ccal \otimes \FunEx(\Ccal,\Dcal) \xrightarrow{\ev} \Dcal,\]  
    so that $\overline{\phi}$ is the image of $\phi$ in the right bottom corner.
    Then the induced map from the vertical fibre over $\phi$ to the vertical fibre over $\overline{\phi}$ is given by the composed map
 	\begin{align*}
 	\Nat(H, (\catop{\phi}\times\phi)^*\Nat^F_G) \to& \Nat(F \boxtimes H, F \boxtimes (\catop{\phi}\times\phi)^*\Nat^F_G) \\
 	\to& \Nat(F\boxtimes H, (\catop{\overline{\phi}}\times\overline{\phi})^*G)
 	\end{align*}
    where the last map is induced by the natural transformation $\eta_{\ev}\colon F\boxtimes\Nat^F_G\Rightarrow(\catop{\ev}\times\ev)^*G$ after restricting along $\id \otimes \phi\colon \Ccal \otimes \Ecal \to \Ccal \otimes \FunEx(\Ccal,\Dcal)$, 
    where we note that $\overline{\phi} = \ev \circ (\id \otimes \phi)$. 
    To see that this map is an equivalence, let us translate it through the currying equivalence 
    \[ \BiMod(\Ccal \otimes \Ecal) = \FunEx(\Ecal\op \otimes \Ecal,\BiMod(\Ccal)).\] 
    Write $G'\colon \Ecal\op \otimes \Ecal \to \BiMod(\Ccal)$ for the exact functor corresponding to $(\catop{\overline{\phi}}\times\overline{\phi})^*G$. Explicitly, $G'$ corresponds to the biexact functor $\Ecal\op \times \Ecal \ni (X,Y) \mapsto (\phi(X) \times \phi(Y))^*G \in \BiMod(\Ccal)$. 
    Let
    \[ L_F \circ (-) \colon \FunEx(\Ecal\op\otimes\Ecal,\Sp) \adj \FunEx(\Ecal\op \otimes \Ecal,\BiMod(\Ccal))\colon R_F \circ (-)\]
    be the adjunction determined by $F$ as in~\eqref{eq:tensor-hom-comp} (for $\Acal = \Ecal\op \otimes \Ecal$). Then the composite
    \[ \Ecal\op \otimes \Ecal \xrightarrow{G'} \BiMod(\Ccal) \xrightarrow{R_F \circ (-)} \Sp \] 
    is the $\Ecal$-bimodule $(\phi\op \times \phi)^*\Nat^F_G$ by construction, and the above composed map can be rewritten as
 	\begin{align*}
 	\Map_{\FunEx(\Ecal\op \otimes \Ecal,\Sp)}(H, R_F \circ G') \to& \Map_{\FunEx(\Ecal\op \times \Ecal,\BiMod(\Ccal))}(L_F \circ H, L_F \circ R_F \circ G') \\
 	\to& \Map_{\FunEx(\Ecal\op \times \Ecal,\BiMod(\Ccal))}(L_F \circ H, G')
 	\end{align*}
    where the second map is given by the counit of $L_F \circ (-) \adj R_F \circ (-)$. We conclude that this composite is indeed an equivalence by adjunction. 
\end{proof}

\subsection{Tensors and cotensors by unstable laced categories}

In this subsection we define an unstable version $\Catb$ of $\CatL$ and show that $\CatL$ is tensored and cotensored over $\Catb$. These operations will be essential in studying trace-like functors on $\CatL$ in \S\ref{sec:trace-like}.

Let $\Pcal\colon \Cat\to\PrCat{L}$ be the functor which associates to a small category $\Ical$ the presentable category $\Pcal(\Ical)$ of space-valued presheaves on $\Ical$. The Yoneda embedding $j\colon \Ical \to \Pcal(\Ical)$ exhibits $\Pcal(\Ical)$ as the universal cocomplete recipient of a functor from $\Ical$, that is, for every cocomplete category $\Ecal$, colimit-preserving functors $\Pcal(\Ical) \to \Ecal$ correspond, via restriction along $j$, to functors $\Ical \to \Ecal$. In addition, by~\cite[Remark 4.8.1.8]{HA}, $\Pcal$ refines to a symmetric monoidal functor, where $\Cat$ carries the cartesian product and $\PrCat{L}$ the Lurie tensor product.

\begin{cons}
    We define $(\Catb)^{\otimes}$ to be the symmetric monoidal category sitting in a fibre square
    \[
        \begin{tikzcd}
            (\Catb)^{\otimes} \ar[r]\ar[d] & \Ar^{\opl}(\PrCat{L})^{\otimes} \ar[d,"{(s,t)}"] \\
            \Cat^{\times} \ar[r, "{(\Pcal,\Pcal)}"] & (\PrCat{L})^{\otimes} \times (\PrCat{L})^{\otimes}
        \end{tikzcd}
    \]
    where the bottom left corner is endowed with the cartesian monoidal structure and the corners on the right with the pointwise Lurie symmetric monoidal structure.
    
    Explicitly, an object of $\Catb$ can be described as a pair $(\Ical,B)$ where $\Ical$ is a small category and $B\colon \Pcal(\Ical) \to \Pcal(\Ical)$ is a colimit preserving functor. Equivalently, $B$ can be described as the data of a functor $\Ical\op \times \Ical \to \Spaces$. We refer to such pairs as \emph{unstable laced categories}. 
    
    As in the case of laced categories, it follows form \cite[Theorem 7.21]{HaugsengHebestreitLinskensNuiten} that the right vertical arrow in the above square is the orthofibration classified by the bifunctor $\Fun^L(-,-)$, and since $\Pcal$ sends every arrow in $\Cat$ to a strongly continuous functor (that is, to an internal left adjoint in $\PrCat{L}$), the left vertical arrow is the cartesian and cocartesian fibration classified by $\Ical \mapsto \Fun(\Ical\op \times \Ical,\Spaces)$. 
\end{cons}

\begin{rmq}
    Explicitly, if $(\Ical,B)$ and $(\Jcal,C)$ are two unstable laced categories then their tensor product is given by $(\Ical \times \Jcal, B \boxtimes C)$, where $B \boxtimes C\colon (\Ical\op \times \Jcal\op) \times (\Ical \times \Jcal) \to \Spaces$ is given by $(B \boxtimes C)(i,j,i',j') = B(i,i') \times C(j,j')$. In fact, this is the cartesian symmetric monoidal structure on $\Catb$. Though we could have more easily defined this structure this way, the path we took was chosen in order to facilitate the construction of a symmetric monoidal functor $\Catb \to \CatL$ below. 
\end{rmq}

Our next step is to construct a symmetric monoidal adjunction $\Catb \adj \CatL$, that is an adjunction where the left adjoint carries a symmetric monoidal structure (and the right adjoint an induced lax symmetric monoidal structure).

For this, recall that the (non-full) inclusion $U\colon \CatEx \to \Cat$ admits a left adjoint $\St\colon \Cat \to \CatEx$ defined as follows. Given a category $\Ical$, the stable category $\St(\Ical)$ is given by the smallest stable subcategory of $\Fun(\Ical\op,\Sp)$ containing the presheaves $\Sig^{\infty}_+\Map_{\Ical}(-,i)$ for every $i \in \Ical$. Since the presheaves $\Sig^{\infty}_+\Map_{\Ical}(-,i)$ constitute a family of compact generators for $\Fun(\Ical\op,\Sp)$ we have in particular that 
\[ 
    \Ind(\St(\Ical)) = \Fun(\Ical\op,\Sp) = \Sp(\Pcal(\Ical)) = \Pcal(\Ical) \otimes \Sp 
\]
is the stabilization of $\Pcal(\Ical)$, and we have a commutative square
\[
    \begin{tikzcd}
        \Cat \ar[r, "\Pcal"]\ar[d,"\St"'] & \PrCat{L} \ar[d, "(-) \otimes \Sp"] \\
        \CatEx \ar[r,"\Ind"] & \PrExCat{L}
    \end{tikzcd}
\]
In addition, the functor $\St$ refines to a symmetric monoidal functor $\Cat^{\times} \to (\CatEx)^{\otimes}$, which is essentially because we can write it as a composite of symmetric monoidal left adjoints 
\[ \Cat^{\times} \xrightarrow{\Pcal_{\fin}} \Cat^{\fin} \xrightarrow{(-) \otimes \Spfin} \CatEx \]
where $\Pcal_{\fin}$ is the free cocompletion functor, which is symmetric monoidal by~\cite[Remark 4.8.1.8]{HA}.
The natural transformation of diagrams of symmetric monoidal categories
\[
    \begin{tikzcd}[row sep = 10pt]
        & \Ar^{\opl}(\PrCat{L})^{\otimes} \ar[dd,"{(s,t)}"]  &&&& \Ar^{\opl}(\PrExCat{L})^{\otimes} \ar[dd] \\
        && \Rightarrow &&& \\
        \Cat^{\times} \ar[r, "{(\Pcal,\Pcal)}"] & (\PrCat{L})^{\otimes} \times (\PrCat{L})^{\otimes} & & (\CatEx)^{\otimes} \ar[rr, "{(\Ind,\Ind)}"] && (\PrExCat{L})^{\otimes} \times (\PrExCat{L})^{\otimes}
    \end{tikzcd}
\]
whose bottom left component is $\St$ and the other two components are induced from the stabilization functor $\Sp \otimes (-)\colon \PrCat{L} \to \PrExCat{L}$ now induces a symmetric monoidal functor
\[ 
    \widetilde{\St} \colon (\Catb)^{\otimes} \to (\CatL)^{\otimes}
\] 
Explicitly, this functor sends $(\Ical,B)$ to $(\St(\Ical),\St(B))$, where 
$$
    \St(B):\St(\Ical)\op \otimes \St(\Ical) = \St(\Ical\op \times \Ical) \longrightarrow \Sp
$$
is the exact functor corresponding to the functor $\SigmaInftyP B\colon \Ical\op \times \Ical \to \Sp$ via the adjunction $\St \dashv U$. In particular, since $\Catb \to \Cat$ and $\CatL \to \CatEx$ are both cartesian fibrations, the adjunctions of $\St \dashv U$ and $\SigmaInftyP \dashv \OmegaInfty$ induce, for an unstable laced category $(\Ical,B)$ and a stable laced category $(\Ccal,M)$, a natural equivalence:
\[ 
    \Map_{\CatL}((\St(\Ical),\St(B)),(\Ccal,M)) \simeq \Map_{\Catb}((\Ical,B),(U(\Ccal),\Omega^{\infty}M)) 
\]
Hence $\widetilde{\St}$ admits a right adjoint $\widetilde{U}\colon \CatL \to \Catb$ given by $(\Ccal,M) \mapsto (U(\Ccal),\Omega^{\infty}M)$ and in particular, $\widetilde{U}$ inherits a lax symmetric monoidal structure and $\widetilde{\St} \dashv \widetilde{U}$ is a symmetric monoidal adjunction.

\begin{cor}
    The category $\CatL$ is tensored, cotensored and enriched over $\Catb$, where the tensor and cotensor operations are given by
    \[ 
        (\Ccal,M)_{(\Ical,B)} := (\Ccal,M) \otimes \widetilde{\St}(\Ical,B) \quad\text{and}\quad (\Ccal,M)^{(\Ical,B)} := \FunInt(\widetilde{\St}(\Ical,B),(\Ccal,M))
    \]
    and the enrichment is given by
    \[ 
        ((\Ccal,M),(\Dcal,N)) \mapsto \widetilde{U}\left(\FunInt((\Ccal,M),(\Dcal,N))\right)
    \]
\end{cor}

Let us now take a closer look at the cotensor construction. Given a laced category $(\Ccal,M)$ and an unstable laced category $(\Ical,B)$, the cotensor $(\Ccal,M)^{(\Ical,B)}$ is given by the internal mapping object $\FunInt(\widetilde{\St}(\Ical,B),(\Ccal,M)) = \FunInt((\St(\Ical),\St(B)),(\Ccal,M))$. As constructed explicitly in \S\ref{subsec:symmetric-monoida-tangent}, this has as underlying stable category the functor category
\[ 
    \FunEx(\St(\Ical),\Ccal) = \Fun(\Ical,\Ccal)
\]
and given two diagrams $\phi,\psi\colon \Ical \to \Ccal$ with associated exact functors $\widetilde{\phi},\widetilde{\psi}\colon \St(\Ical) \to \Ccal$, the associated bimodule $\Nat^{\St(B)}_{M}$ on $\FunInt((\St(\Ical),\St(B)),(\Ccal,M))$ is given by the equivalent expressions
\begin{align*}
    \Nat^{\St(B)}_{M}(\widetilde{\phi},\widetilde{\psi}) &:= \Nat(\St(B),(\widetilde{\phi}\op \times \widetilde{\psi})^*M) \\
    &\simeq \Nat(\Sigma^{\infty}_+B, (\phi\op \times \psi)^*M) \\
    &\simeq \Nat(B,(\phi\op \times \psi)^*\OmegaInfty M)
\end{align*}

\begin{ex} \label{CotensorPointEx}
   Suppose $B = \ast\colon \catop{\Ical}\times \Ical\to\Spaces$ is the terminal functor.    Given a laced category $(\Ccal, M)$ we have
    \[
        \Nat(*, (\catop{\phi}\times \psi)^*\Omega^{\infty} M)\simeq \lim_{(i,j)\in \catop{\Ical}\times \Ical}M(\phi(i), \psi(j))
    \]
    In particular, in the case where $\Ical=[n]$ is the linearly order poset with $n$ elements, we get that
    \[
        M^{([n], *)}(\phi, \psi)\simeq \Omega^{\infty}M(\phi(n), \psi(0))
    \]
    Viewing a diagram $\phi\colon [n]\to\Ccal$ as the data of a sequence $X_0 \to ... \to X_n$, the above bimodule is simply $M(X_n, Y_0)$ for two sequences $(X_i)$ and $(Y_i)$.
\end{ex}

\begin{ex} \label{CotensorMapEx}
    Suppose $\Ical=\Delta^1$ and take $B:=\Map_{\Delta^1}$ to be the mapping space bifunctor. Let $\phi\colon \Delta^1 \to \Ccal$ and $\psi\colon \Delta^1 \to \Ccal$ be two arrows in $\Ccal$, corresponding to morphisms $f\colon X \to Y$ and $g\colon X'\to Y'$ in $\Ccal$.
	Given a $\Ccal$-bimodule $M$ we then have an equivalence
    \[
        \Nat(\Map_{\Delta^1}, (\catop{\phi}\times \psi)^* M)\simeq M(X, X')\times_{M(X, Y')}M(Y, Y')
    \]
    where the maps of the pullback are induced by $f$ and $g$ under $M$. 

    More generally, if $\Ical=[n]$ and $B=\Map_{[n]}$ its mapping space, then we have
    \[
        \Nat(\Map_{[n]}, (\catop{\phi}\times \psi)^* M)\simeq \lim_{(i\leq j)\in\catop{\TwAr([n])}} M(\phi(i), \psi(j)),
    \]
    which recovers the above when $n=1$.
\end{ex}

%%%%%%%%%%%%%%%%%%%%%%%%%%%%%%%%%%%%%%%%%%%%%%%%%%%%%%%%%%%%%%%%%%%%%%%%%%%%%%%%%%%%%%%%%%%%%%%%%%%%%%%%%%%%%%%%

\section{The K-theory of laced categories}
\subsection{Laced semi-orthogonal decompositions}
\label{lace-additivity}%

Recall that a \emph{semi-orthogonal decomposition} of a stable category $\Ccal$ is the datum of a pair of full subcategories $\Acal \subseteq \Ccal \supseteq \Bcal$ satisfying the following two conditions:
\begin{itemize}[align=left]
    \item[(Decomposition)] Every $X\in\Ccal$ fits in an exact sequence $A\to X\to B$ with $A\in\Acal$ and $B\in\Bcal$.
    \item[(Semi-orthogonality)] For every $A\in\Acal$ and $B\in\Bcal$, $\map_{\Ccal}(A, B)\simeq 0$
\end{itemize}
The semi-orthogonality condition ensures that the decomposition of the first condition is unique, and thus functorial. Sending $X$ to the first term in its decomposition sequence then gives a right adjoint to the inclusion of $\Acal$, and sending it to the last term gives a left adjoint to the inclusion of $\Bcal$. 

The data of a semi-orthogonal decomposition is hence equivalent to the a-priori stronger notion a right-split Verdier sequence, i.e., a null-composite sequence
\[
    \begin{tikzcd}
        \Acal\arrow[r, "i"] & \ar[l,bend left, "\perp"']\Ccal\arrow[r, "p"] & \Bcal \ar[l, bend left, "\perp"']
    \end{tikzcd}
\]
such that $i$ is fully-faithful and has a right adjoint and $p$ has a fully-faithful right adjoint, as studied in~\cite[Appendix A.2]{HermKII}. Passing to right adjoints gives a left-split Verdier sequence encoding the same data; the formalism of semi-orthogonal decompositions can then be considered as a symmetric manner to express one-side-split Verdier sequences. 

Note however that the condition of being a semi-orthogonal decomposition remains asymmetric in the couple $(\Acal, \Bcal)$.

\begin{rmq}\label{rmq:recollement}
    A semi-orthogonal decomposition in which the inclusion of $\Acal$ admits both adjoints is also called a stable recollement in~\cite{HA}. It is equivalent to the data of a split Verdier sequence, that is, a Verdier sequence in which the inclusion and the projection admit both adjoints, see~\cite[Appendix A.2]{HermKII}.
\end{rmq}

We now give a definition in the laced setting. By a laced full subcategory $(\Dcal,N) \subseteq (\Ccal,M)$ we will mean a full subcategory $\Dcal \subseteq \Ccal$ equipped with an identification $N = M|_{\Dcal\op \times \Dcal}$. Equivalently, this is the data of a laced functor $(\Dcal,N) \to (\Ccal,M)$ whose underlying exact functor is fully-faithful and such that the structure map $N \to M|_{\Dcal\op \times \Dcal}$ is an equivalence.

\begin{defi}
    Let $(\Ccal, M)$ be a laced category. A laced semi-orthogonal decomposition of $(\Ccal, M)$ is a pair of laced full subcategories $(\Acal, N) \subseteq (\Ccal,M) \supseteq (\Bcal, P)$ satisfying the following conditions: 
    \begin{itemize}[align=left]
        \item[(Underlying)] The underlying pair of stable subcategories $\Acal \subseteq \Ccal \supseteq \Bcal$ is a semi-orthogonal decomposition of $\Ccal$.
        \item[(Laced semi-orthogonality)] For every $A\in\Acal$ and $B\in\Bcal$ we have $M(A, B)\simeq 0$.
    \end{itemize}
\end{defi}

\begin{lmm}\label{lmm:laced-adjoint}
    Let $(\Acal, N) \subseteq (\Ccal,M) \supseteq (\Bcal, P)$ be a laced semi-orthogonal decomposition. Then, the left adjoint $p$ of the inclusion $j\colon \Bcal\to\Ccal$ canonically upgrades to a laced functor $(p, \eta)\colon (\Ccal, M)\to(\Bcal, P)$ in such a way that
    \[
        \begin{tikzcd}
            (\Acal, N)\arrow[r, "{(i, \alpha)}"] & (\Ccal, M)\arrow[r, "{(p, \eta)}"] & (\Bcal, N)
        \end{tikzcd}
    \]
    becomes a fibre-cofibre sequence in $\CatL$.
\end{lmm}
\begin{proof}
To construct a natural transformation $\eta\colon M \Rightarrow P \circ (p\op \times p)$ it suffices to construct its mate
$\hat{\eta}\colon (p\op \times p)_!M \Rightarrow P$. 
We can factor $p\op \times p$ as follows:
$$
    \begin{tikzcd}[column sep=large]
         p\op \times p:\Ccal\op \times \Ccal \arrow[r, "{\id \times p}"] & \Ccal\op \times \Bcal\arrow[r, "{\catop{p} \times \id}"] & \Bcal\op \times \Bcal
    \end{tikzcd}
$$
Left Kan extension is functorial, as the left adjoint of precomposition, hence in two steps, by first left Kan extending along $\id \times p$ and then along $p\op \times \id$. The first step can be implemented by restriction along the inclusion $j\colon \Bcal \hookrightarrow \Ccal$, since $j$ is right adjoint to $p$.

Let us now write $M'\colon \Bcal\op \times \Bcal \to \Sp$ for the left Kan extension of $M(-,j(-))$ along $p\op \times \id$. Since $p\colon \Ccal \to \Bcal$ is a semi-split Verdier projection with kernel $\Acal$, it exhibits $\Bcal$ as the localisation of $\Ccal$ by equivalences modulo $\Acal$, that is, by the arrows whose fibre lies in $\Acal$.
At the same time, by the laced semi-orthogonality assumption $M(i(-),j(-)) = 0$, so that $M(-,j(-))$ is invariant under equivalences modulo $\Acal$ in the first coordinate, and hence descends to a biexact functor $\Bcal\op \times \Bcal \to \Sp$ in an essentially unique manner. More precisely, the unit map
\[ M(-,j(-)) \Rightarrow M'(p(-),-) \] 
is an equivalence. The structure natural transformation $P \xrightarrow{\simeq} M \circ (j\op \times j)$ of $j$ then determines an equivalence $P \simeq M'$, and so we obtain an equivalence $\hat{\eta}\colon (p\op \times p)_!M \simeq P$, which refines $p$ to a laced functor.

To show that the resulting sequence of laced categories is both a fibre and a cofibre sequence, note that on the level of underlying stable categories the sequence
    \[
        \begin{tikzcd}
            \Acal\arrow[r, "i"] & \Ccal\arrow[r, "p"] & \Bcal
        \end{tikzcd}
    \]
is fibre and cofibre, and hence to check its enhancement to $\CatL$ is a fibre-cofibre, it suffices to check that $\alpha\colon N\to M\circ(\catop{i}\times i)$ is an equivalence and that the mate $\widehat{\eta}\colon (\catop{p}\times p)_!M\to N$ of $\eta$ is an equivalence. The former is by assumption and the latter by our construction of $\eta$ above.
\end{proof}

We now describe three types of examples of laced semi-orthogonal decompositions which will prove useful in later parts of the paper.

\begin{ex} \label{LpreservesSOD}
    If $\Acal \subseteq \Ccal \supseteq \Bcal$ is a semi-orthogonal decomposition then $(\Acal, \map_\Acal) \subseteq (\Ccal,\map_{\Ccal}) \supseteq (\Bcal, \map_\Bcal)$ is a laced semi-orthogonal decomposition. Indeed, the fact that $\Acal \subseteq \Ccal \supseteq \Bcal$ are full subcategory inclusions directly implies that $(\Acal, \map_\Acal) \subseteq (\Ccal,\map_{\Ccal}) \supseteq (\Bcal, \map_\Bcal)$ are laced full subcategory inclusions, and the laced semi-orthogonality condition is simply semi-orthogonality of the underlying stable decomposition in this case.
\end{ex}

\begin{ex}\label{ex:cotensor-decompose}
The cotensored laced category $(\Ccal, M)^{([1], \Map)}$ of Example~\ref{CotensorMapEx} can be obtained as the semi-orthogonal decomposition of two copies of $(\Ccal, M)$.
    Indeed, it is a classical fact that $\Ccal^{[1]}$ is the semi-orthogonal decomposition of two copies of $\Ccal$, the inclusions being given by $i\colon X\mapsto(X\to 0)$ and $j\colon X\mapsto \id_X$. As described in Example \ref{CotensorMapEx} the bimodule of $(\Ccal, M)^{([1], \Map)}$ fits in a cartesian square of the form
    \begin{equation}\label{eq:universal-laced-subcat}
        \begin{tikzcd}
            M^{([1], \Map)}(f\colon X\to Y, g\colon X'\to Y')\arrow[r]\arrow[d] & M(Y, Y')\arrow[d, "{M(f,-)}"] \\
            M(X, X')\arrow[r, "{M(-,g)}"] & M(X, Y').
        \end{tikzcd}
    \end{equation}
    Plugging in $Y=Y'=0$ makes the vertical maps into equivalences, and plugging $f, g$ being identities makes all the maps into equivalences, so that in both cases the left vertical map is an equivalence. We consequently obtain equivalences $\alpha\colon M \xrightarrow{\simeq} M^{([1], \Map)} \circ (i\op \times i)$ and $\beta\colon M \xrightarrow{\simeq} M^{([1], \Map)} \circ (j\op \times j)$, yielding laced full subcategory inclusions
\[
\begin{tikzcd}
(\Ccal,M) \ar[r,hook,"{(i,\alpha)}"] & (\Ccal,M)^{([1],\Map)} \ar[r,hookleftarrow,"{(j,\beta)}"] &(\Ccal,M)
\end{tikzcd} 
\]
To see that this is a laced semi-orthogonal decomposition note that if one takes $Y=0$ and $g=\id$ in~\eqref{eq:universal-laced-subcat} then the horizontal maps are equivalences and the top right corner vanishes, so that $M^{([1], \Map)}$ vanishes, yielding laced semi-orthogonality.
\end{ex}

\begin{ex}\label{CotensorPointSOD}
Let us now consider the cotensor $(\Ccal, M)^{([1], *)}$ of example \ref{CotensorPointEx}. The underlying stable category is the same as in Example~\ref{ex:cotensor-decompose}, and the bimodule is given by the formula
    \[
        M^{([1], *)}(f\colon X\to Y, g\colon X'\to Y')\simeq M(Y, X').
    \]
If $Y=Y'=0$ then $M^{([1], *)}$ vanishes whereas it recovers $M$ in the case where $f, g$ are identities, so that we have a pair of laced full subcategory inclusions
\[
\begin{tikzcd}
(\Ccal,0) \ar[r,hook,"{(i,\alpha)}"] & (\Ccal,M)^{([1],\Map)} \ar[r,hookleftarrow,"{(j,\beta)}"] &(\Ccal,M)
\end{tikzcd} 
\]
To see that this is a laced semi-orthogonal decomposition note that $M^{([1], *)}$ in fact vanishes as soon as $Y=0$ (regardless of $X',Y'$) so that we have laced semi-orthogonality.
\end{ex}

\begin{prop} \label{AdjPresSOD}
    The pair of adjoint functors $L\colon \CatEx \adj \CatL\cocolon \Lace$ both preserve semi-orthogonal decompositions.
\end{prop}
\begin{proof}
    For $L$, this is the content of Example \ref{LpreservesSOD}. Now let $(\Acal, N) \subseteq (\Ccal,M) \supseteq (\Bcal, P)$ be a laced semi-orthogonal decomposition. Then we have an induced pair of full subcategory inclusions $\Lace(\Acal,N) \subseteq \Lace(\Ccal, M) \supseteq \Lace(\Bcal,P)$ and we want to show that it constitutes a semi-orthogonal decomposition. 

Let $(X, f)$ be an object of $\Lace(\Ccal, M)$, so that $X\in\Ccal$ and $f\colon X\to M(X)$ is a morphism. Since $(\Acal, \Bcal)$ is an orthogonal decomposition of $\Ccal$, we have an exact sequence $q(X)\to X\to p(X)$ where $p\colon \Ccal\to\Bcal$ and $q\colon \Ccal\to\Acal$ are the corresponding reflection and coreflection functors and we omit the inclusions from the notation. By the laced semi-orthogonality condition we have $\map(q(X),M(p(X))) = 0$, and so the map $X \to M(X)$ extends in an essentially unique manner to a commutative diagram
    \[
        \begin{tikzcd}
            q(X)\arrow[r]\arrow[d] & X\arrow[r]\arrow[d] & p(X)\arrow[d] \\
            M(q(X))\arrow[r] & M(X)\arrow[r] & M(p(X))
        \end{tikzcd}
    \]
in which both rows are fibre sequences. 
At the same time, since $(\Acal,N) \subseteq (\Ccal,M) \supseteq (\Bcal,P)$ are laced full subcategory inclusions this is exactly the datum of an exact sequence in $\Lace(\Ccal, M)$ whose first term is in $\Lace(\Acal, N)$, last term is in $\Lace(\Bcal, P)$ and middle term is $(X, f)$, hence we have decompositions.

    For the semi-orthogonality, note that by the cartesian square~\eqref{eq:lace}, mapping spaces in $\Lace(\Ccal, M)$ can be expressed as equalizers;  
explicitly, we need to show that for every $(A, f)\in\Lace(\Acal, N)$, $(B, g)\in\Lace(\Bcal, P)$, the following spectrum vanishes:
    \[
        \Eq\left(\begin{tikzcd}[column sep=large]\map_\Ccal(A, B)\arrow[r, shift left=1, "M(-)\circ f"]\arrow[r, shift right=1, "M(g)\circ(-)"'] &\map_{\Ind(\Ccal)}(A, M(B))\end{tikzcd}\right)
    \]
    This follows from the fact that both $\map_\Ccal(A, B)$ and $\map(A, M(B))$ vanish, the former by semi-orthogonality of the underlying $(\Acal, \Bcal)$ and the latter by laced semi-orthogonality.
\end{proof}

\subsection{The universal property of laced K-theory}

Let $\Ecal$ be a stable $\infty$-category and $\Fcal \colon \CatEx \to \Ecal$ a functor. We say that $\Fcal$ is \emph{reduced} if it sends the zero stable category to the zero object. For a reduced functor $\Fcal\colon \CatEx \to \Ecal$, the following properties are equivalent (see~\cite[Proposition 2.4]{HebestreitLachmannSteimle}): 
\begin{itemize}
    \item $\Fcal$ sends semi-orthogonal decompositions to direct sums.
    \item $\Fcal$ sends left-split Verdier sequences to exact sequences.
    \item $\Fcal$ sends right-split Verdier sequence to exact sequences.
    \item $\Fcal$ sends split Verdier sequences to exact sequences.
    \item $\Fcal$ is extension splitting, that is, it sends the semi-orthogonal decomposition $\Ccal \subseteq \Ccal^{[1]} \supseteq \Ccal$ (which is actually a recollement, see Remark~\ref{rmq:recollement}) underlying Examples~\ref{ex:cotensor-decompose} and~\ref{CotensorPointSOD}, to a direct sum. 
\end{itemize}

A reduced functor $\Fcal\colon \CatEx \to \Ecal$  
satisfying any of these equivalent conditions is called \emph{additive}. 
A key example of an additive functor is the $\Kth$-theory functor $\Kth\colon \CatEx \to \Sp$. In fact, this functor enjoys a universal property with respect to additive functors, as established in \cite{BlumbergGepnerTabuada}:

\begin{prop}[Blumberg-Gepner-Tabuada] \label{BGTMainThm}
    The natural transformation $\SigmaInftyP\core\Rightarrow\Kth$ of functors $\CatEx\to\Sp$ exhibits algebraic K-theory as the initial additive invariant under $\SigmaInftyP\core$.
\end{prop}

We propose the following definition as a laced counterpart of additivity.
  
\begin{defi}
    Let $\Ecal$ be an additive category. We say that a functor $\Fcal\colon \CatL\to\Ecal$ is \textit{additive} if it is reduced and sends laced semi-orthogonal decompositions to direct sum decompositions in $\Ecal$. 
\end{defi}

\begin{defi}
    Write $\Klace\colon \CatL\to\Sp$ for the composite
    \[
    \begin{tikzcd}
        \CatL\arrow[r, "\Lace"] & \CatEx\arrow[r, "\Kth"] & \Sp.
    \end{tikzcd}
    \]
    It comes naturally equipped with a natural transformation $\Sigma^{\infty}_+\LaceEq\to \Klace$.
\end{defi}

By Theorem \ref{AdjPresSOD}, $\Klace$ is additive, since $\Kth\colon \CatEx\to\Sp$ is so and $\Lace$ preserves orthogonal decompositions.

\begin{thm}\label{KlaceIsUnivAdditive}
    The natural transformation $\SpLaceEq\Rightarrow\Klace$ of functors $\CatL\to\Sp$ exhibits laced K-theory as the initial additive invariant under $\SpLaceEq$.
\end{thm}
\begin{proof}
    First, as we mention above, since $\Lace$ preserves semi-orthogonal decompositions and $\Kth\colon \CatEx\to\Sp$ is additive, so is $\Klace$. Second, since $L$ is left adjoint to $\Lace$ the associated precomposition functor $L^*$ is right adjoint to the pre-composition functor $\Lace^*$. For every functor $F\colon \CatL\to\Sp$ we may then consider the commutative square
\[
\begin{tikzcd}
        \Nat(\Klace,F) \ar[r]\ar[d,"\simeq"] & \Nat(\SpLaceEq, F) \ar[d,"\simeq"] \\
        \Nat(K,F \circ L)\ar[r] & \Nat(\SigmaInftyP\core, F\circ L)
\end{tikzcd}
\]
whose vertical legs are equivalences.
But if $F$ is additive then so is $F\circ L$ by Propositions \ref{AdjPresSOD}, and hence by Proposition \ref{BGTMainThm} the bottom horizontal map is an equivalence, so that the top horizontal map is an equivalence as well.
\end{proof}

Recall that the main result of \cite{BlumbergGepnerTabuadaMultiplicative} shows that the K-theory 
is the unit of the symmetric monoidal category $\Fun^{\omega,\add}(\CatEx, \Sp)^\otimes$ of finitary additive functors on $\CatEx$ (with the symmetric monoidal structure induced by Day convolution). Consequently, it carries a canonical lax symmetric monoidal structure under which it is the initial lax symmetric monoidal finitary additive functor $\CatEx\to\Sp$. In fact, this also holds if we drop the adjective ``finitary'', since the collection of finitary functors form a coreflective subcategory. In addition, the natural transformation
\[ \SigmaInftyP\core \Rightarrow \Kth ,\]
which can be viewed as the component at the unit of the unit of the multiplicative adjunction 
\[ \Fun^{\omega}(\CatEx,\Sp)^{\otimes} \adj \Fun^{\omega,\add}(\CatEx,\Sp)^{\otimes},\] 
is naturally compatible with the lax symmetric monoidal structures on both sides.

In Construction \ref{cons:sm-structure}, we have argued that the cotangent complex $L$ refines to a symmetric monoidal functor and therefore its right adjoint admits a lax symmetric monoidal structure. This endows laced K-theory with a lax symmetric monoidal structure of its own, and refines the natural transformation
\[ \SigmaInftyP\LaceEq \Rightarrow \Klace \]
to one of lax symmetric monoidal functors.

\begin{prop} \label{KlaceInitialLaxMonoidal}
    The functor $\Klace$ is the initial lax symmetric monoidal additive functor $\CatL\to\Sp$.
\end{prop}
\begin{proof}
    By the above discussion, the adjunction between $L$ and $\Lace$ induces, by means of pre-composition, an adjunction     \[ \Lace^*\colon \Fun^{\lax-\otimes,\add}(\CatEx,\Sp) \adj \Fun^{\lax-\otimes,\add}(\CatL,\Sp) \cocolon L^* \]
    on the level of lax symmetric monoidal additive functors to $\Sp$.
In particular, $\Lace^*$ preserves the initial object, which is K-theory, so that the claim follows.
\end{proof}

\begin{rmq} \label{RemarkLacedMotives}
    In the laced setting, the construction of non-commutative additive motives of \cite{BlumbergGepnerTabuada} can be adapted to give a presentable stable $\Mcal^{\lace}_{\add}$ and a universal finitary additive 
    $$
        U^{\lace}_{\add}\colon \CatL\to\Mcal^{\lace}_{\add}.
    $$ 
    As in \textit{loc. cit.}, this is done using Proposition 5.3.6.2 of \cite{HTT}, and then taking the Spanier-Whitehead stabilization. 
    Since this category of motives will not play a role in the present paper, we will refrain from pursuing these ideas here.
\end{rmq}

%%%%%%%%%%%%%%%%%%%%%%%%%%%%%%%%%%%%%%%%%%%%%%%%%%%%%%%%%%%%%%%%%%%%%%%%%%%%%%%%%%%%%%%%%%%%%%%%%%%%%%%%%%%%%%%%

\section{Trace-like functors and topological Hochschild homology}
\label{sec:trace-like}%

In the previous sections, we have built a category $\CatL$ and extended algebraic K-theory into a functor $\Klace\colon \CatL\to\Sp$ which we have characterized by a universal property. In this section we add to our discussion the functor $\THH$ which associates to a stable category its Hochschild homology. Extending $\THH$ to laced categories, we show that this contexts affords $\THH$ a new universal property, expressed via the notion of trace invariance. Unlike the case of $\Kth$-theory, this universal property is only visible when passing to the laced setting. Using this universal property we show in Section~\ref{subsec:THH-derivative} that $\THH$ with coefficients is the first Goodwillie derivative of laced $\Kth$-theory.

\subsection{Topological Hochschlid homology and the cyclic bar construction}

In this section we recall Hochschild homology of stable categories and explain how to extend the construction to the laced setting. For this, let us first recall the unstable version of the construction.

\begin{defi}
    Let $\Ccal$ be an category. The \textit{unstable topological Hochschild homology} of $\Ccal$ is the space $\uTHH(\Ccal)$ given by the coend formula
    \[
        \uTHH(\Ccal) = \int^{X\in\Ccal} \Map_\Ccal(X,X) := \colim_{f\colon X\to Y\in\TwAr(\Ccal)} \Map_\Ccal(Y, X),
    \]
    where $\Map_\Ccal$ denote the mapping space of $\Ccal$ and $\TwAr(\Ccal)$ is the twisted arrow category of $\Ccal$. This construction defines a functor $\uTHH\colon \Cat\to\Spaces$. 
\end{defi}

\begin{rmq} \label{uTHHexplicit}
	If $\Ccal = \Gcal$ is an $\infty$-groupoid then $\TwAr(\Gcal) = \Ar(\Gcal) = \Gcal$ and we have a natural equivalence 
	\[ \uTHH(\Gcal) = \Map(S^1,\Gcal).\] 
	In particular, the following composite
	\[
    		 \Ccal^\simeq\to\Fun(S^1, \Ccal^\simeq)=\uTHH(\Ccal^\simeq)\to\uTHH(\Ccal)
	\]
	provides a natural transformation $\core\Rightarrow\uTHH$.
\end{rmq}

To pass from unstable to stable $\THH$, recall that any stable category $\Ccal$ is canonically enriched in spectra, and this enrichment, denoted $\map_{\Ccal}(-,-)$, is characterized to the unique functor exact in both variables such that
$$
    \Map_{\Ccal}(-,-) = \OmegaInfty\map_{\Ccal}(-,-)
$$
where the left hand side is the mapping space functor.
\begin{defi} \label{THHinCatEx}
    Let $\Ccal$ be a stable category.
    The \textit{topological Hochschild homology} of $\Ccal$ is the spectrum $\THH(\Ccal)$ given by the coend formula
    \[
        \THH(\Ccal) = \int^{X\in\Ccal} \map_\Ccal(X,X) = \colim_{f\colon X\to Y\in\TwAr(\Ccal)} \map_\Ccal(Y, X).
    \]
    This construction defines a functor $\THH\colon \CatEx\to\Sp$, and the above observation endows it with a canonical natural transformation $\uTHH\Rightarrow\LoopInfty\THH$ or equivalently, $\SigmaInftyP\uTHH\Rightarrow \THH$. 
\end{defi}

The above definitions can be upgraded to the laced setting in a straightforward manner.

\begin{defi}\label{defi:laced-thh}
    Let $(\Ccal, M)$ be a laced category. The \textit{topological Hochschild homology} of $(\Ccal, M)$ is the spectrum  $\THH(\Ccal, M)$ given by the coend formula
    \[
        \THH(\Ccal, M) = \int^{X\in\Ccal} M(X, X) = \colim_{f\colon X\to Y\in\TwAr(\Ccal)}M(Y, X).
    \]
    This construction defines a functor $\THH\colon \CatL\to\Sp$. Similarly, for pairs $(\Ical, F)$ with $I$ a category and $F$ a functor $F\colon \catop{\Ical}\times \Ical\to\Spaces$, we can also define the \textit{unstable topological Hochschild homology} by:
    \[
        \uTHH(\Ical, F) = \int^{X\in \Ical} F(X, X) = \colim_{f\colon X\to Y\in\TwAr(\Ical)}F(Y, X)
    \]
    The functor $\uTHH\colon \Catb\to\Spaces$ is such that, when restricted to $\CatL$, we have a natural transformation $\SigmaInftyP \uTHH\Rightarrow \THH$.
\end{defi}

\begin{ex}\label{ex:primitive-bimodules}
The simplest types of bimodules one can consider on a given stable $\infty$-category $\Ccal$ are of the form $M_{A,B}(X,Y) = \map(B,Y)\otimes\map(X,A)$ for some $A,B \in \Ccal$. We claim that $\THH(\Ccal,M_{A,B}) \simeq \map(B,A)$. In fact, there are comparison maps
\[ \map(B,A) \to \colim_{[f\colon X \to Y]}\map(B,X)\otimes\map(Y,A) \to \map(B,A) \]
in both directions: the left map is induced by the inclusion $\{\id_B\} \subseteq \TwAr(\Ccal)$ together with the map 
\[ \map(B,A) \xrightarrow{\{\id_B\}\otimes \id} \map(B,B) \otimes \map(B,A)  = M_{A,B}(B,B)\] 
while the one on the right is induced by the natural transformation to the constant diagram on $\map(B,A)$ which on $[f\colon X \to Y] \in \TwAr(\Ccal)$ implements the composite
\[ \map(B,X)\otimes\map(Y,A)  \xrightarrow{\id\otimes \{f\} \otimes \id} \map(B,X) \otimes \map(X,Y)\otimes \map(Y,A) \to \map(B,A) .\]
We claim that both comparisons are equivalences. Indeed, their composite is the identity, and so it is enough to show this for the first one. Mapping into a test object $E \in \Sp$ then yields, using the end formula for natural transformations, the map
\[ \Nat(\map(B,-),\map(\map(-,A),E)) \to \map(\map(B,A),E) \]
induced by the evaluation at $B$, and this map is an equivalence by the stable Yoneda lemma. 

Finally, note that the comparison map on the right is visibly natural in the pair $(A,B) \in \Ccal\times \Ccal\op$. In particular, this argument shows that the composite functor
\[ \Ccal \times \Ccal\op \xrightarrow{(A,B) \mapsto M_{A,B}} \BiMod(\Ccal) \xrightarrow{\THH(\Ccal,-)} \Sp \]
is equivalent to the functor $(A,B) \mapsto \map(B,A)$. Identifying $\BiMod(\Ccal)$ with $\Ind(\Ccal \otimes \Ccal\op)$, this completely characterizes $\THH(\Ccal,-)\colon \BiMod(\Ccal) \to \Sp$ as the unique colimit preserving functor whose restriction to $\Ccal \times \Ccal\op$ is $\map(-,-)$.
\end{ex}

\begin{ex}\label{ex:thh-of-spaf}
For the $\infty$-category $\Spaf$ of finite spectra, the map 
\[ M(\bS,\bS) \to \colim_{[X \to Y] \in \TwAr(\Spaf)}M(Y,X) = \THH(\Spaf,M)  \]
induced by the inclusion $\{[\id\colon \bS \to \bS]\} \subseteq \TwAr(\Spaf)$, is an equivalence. Indeed, every bimodule on $\Spaf$ is of the form $M(X,Y) = E \otimes \map(X,Y) =\map(\bS,Y) \otimes \map(X,E)$ for some $E \in \Sp$. When $E$ is a finite spectrum the claim is an instance of Example~\ref{ex:primitive-bimodules}, and the general case follows since both sides are colimit preserving in $M$.
\end{ex}

\begin{ex}\label{ex:thh-of-rings}
Let $R$ be an $\E_1$-ring spectrum, $M$ an $R$-bimodule and $F_M(-) = M \otimes_R -\colon \Perf(R) \to \Mod(R)$ the associated bimodule on $\Perf(R)$, see Example~\ref{ExampleRingSp1}. Then we have a natural equivalence
\[ \THH(\Perf(R),F_M) \simeq \THH(R;M)\]
where $\THH(R;M) = M \otimes_{R\otimes R\op} R$ is the algebraically defined topological Hochschild homology. Indeed, both are colimit preserving functors in $M$ whose restriction to bimodules of the form $M = P \otimes Q$ for $P \in \Perf(R)$ and $Q \in \Perf(R\op)$ is naturally equivalent to the underlying spectrum of $Q \otimes_R P$, see the characterization at the end of Example~\ref{ex:primitive-bimodules}.
\end{ex}

\begin{rmq}\label{rmq:unusual}
    Definition~\ref{defi:laced-thh} may appear different than some of the definitions for $\THH$ of stable or spectrally enriched categories previously considered in the literature, and which typically use some form of the cyclic bar construction, see, e.g., \cite[\S 10]{BlumbergGepnerTabuada}, \cite[Definition III.2.3]{NikolausScholze}, \cite{HoyoisScherotzkeSibilla}, or the more recent \cite{KrauseMcCandlessNikolaus} for higher categorical constructions, and \cite[Definition 1.3.6]{DundasMcCarthyTHHRingFunctors} or \cite{Bokstedt} for classical constructions (the former applies to spectrally enriched categories, referred to as ring functors in loc.\ cit., and the latter to rings, or ring spectra). Nonetheless, for all these construction it is typically standard to show that their value on representable bimodules of the form $M_{A,B}$ is given by $\map(B,A)$ naturally in $A,B$, and so the comparison with Definition~\ref{defi:laced-thh} is automatic from the characterization in the end of Example~\ref{ex:primitive-bimodules} (cf. Example~\ref{ex:thh-of-rings}).
\end{rmq}

We now seek to show that the $\THH$ we defined can also be realize as some version of a cyclic Bar construction, i.e., the geometric realization of a cyclic object which looks like the Bar construction where some terms have been modified to be cyclic.

\begin{defi}
    Let $(\Ccal, M)$ be a laced category, we denote $\baBar_{n}(\Ccal, M)$ the following space :
    \[
        \baBar_{n}(\Ccal, M):=\LaceEq((\Ccal, M)^{([n], *)})
    \]
    By functoriality of the cotensor, those spaces assemble into a functor $\baBar_\bullet\colon \CatL\to\Scal^{\catop{\Delta}}$. We call it the \textit{cyclic bar construction} of the functor $\LaceEq$.
\end{defi}

Unwinding the above definition, the objects of the cotensored category $(\Ccal, M)^{([n], *)}$ are chains $X_0\to ...\to X_n$ in $\Ccal$, while a lacing on such an object corresponds to a point in $\LoopInfty M(X_n, X_0)$, which we can equivalently see as an arrow $X_n\to M(X_0)$ in $\Ind(\Ccal)$. Hence, $\Lace((\Ccal, M)^{([n], *)})$ is the category of $n$-cycles where $n-1$ terms are simply arrows $X_i\to X_{i+1}$, i.e, points in $\Map(X_i, X_{i+1})$, and the last term is ``cycling back'', but in a twisted way, giving a point in $\Map(X_n, M(X_0))$.

\begin{prop} \label{CyclicBarInTCatEx}
    Let $(\Ccal, M)\in \CatL$. We have a natural equivalence
    \[
        \uTHH(\Ccal, M)\simeq |\baBar_{\bullet}(\Ccal, M)|
    \]
\end{prop}
\begin{proof}    
    By definition, $\uTHH(\Ccal, M)$ is given by the colimit of the composite $\LoopInfty M\circ p$, where 
    \[
        p: \TwAr(\Ccal)\to\Ccal\times\catop{\Ccal}
    \] 
    is the right fibration classifying the mapping space of $\Ccal$. 
    Let $T_n:=\Map(\Delta^n,\TwAr(\Ccal))$. Then the inclusion $T_0\simeq\TwAr(\Ccal)^\simeq\to\TwAr(\Ccal)$ yields a composed map 
    \[
        \begin{tikzcd}[column sep=huge]
           \phi_n: T_n\arrow[r, "d_0"] & T_0\arrow[r] & \TwAr(\Ccal)\arrow[r, "\LoopInfty M\circ\; p"] & \Scal.
        \end{tikzcd}
    \]
    By Corollary 12.5 of \cite{Shah} (the \textit{Bousfield-Kan formula}), taking the colimit produces a simplicial object $X_\bullet = \colim_{T_{\bullet}}\phi_\bullet$ whose geometric realization recovers $\uTHH(\Ccal, M)$. 
    
    We now identify each $X_n$ in terms of the cyclic bar construction. Recall from \cite[\S 2.2]{BarwickQConstr} that the edgewise subdivision is the functor $e\colon \Delta\to\Delta$ given by  
    $[n]\mapsto [n]\ast\catop{[n]}\simeq[2n+1]$.  
    The functor $\catop{e}\colon \catop{\Delta}\to\catop{\Delta}$ is cofinal, so that we are reduced to show that there is an equivalence $X_n\simeq\baBar_{e(n)}(\Ccal, M)$ natural in $[n]\in\Delta$. 
    
    By definition, $\baBar_{e(n)}(\Ccal, M)\simeq \LaceEq((\Ccal, M)^{(\catop{[n]}\ast[n], *)})$. Remark that on the underlying stable categories, we have $\Ccal^{\catop{[n]}\ast[n]}\simeq\TwAr(\Ccal)^{[n]}$; we can also rewrite the bimodule under this equivalence, and we claim the following holds:
    \[
        (\Ccal, M)^{(\catop{[n]}\ast[n], *)}\simeq(\TwAr(\Ccal), M\circ p)^{([n], *)}
    \]
    This is the case when $n=0$ by Example \ref{CotensorPointEx}. More generally, recall that if $I=\{i_0 < ... < i_\kappa\}$ is linearly ordered, the formula for $M^{(I, *)}$ is the evaluation $M(X(i_\kappa), Y(i_0))$, again by Example \ref{CotensorPointEx}. For $I=\catop{[n]}\ast [n]$, this is the following diagram where we colored in red the rightmost-vertical arrow, which corresponds to the first point when passing to the twisted-arrow-category point of view:
    \[
    \begin{tikzcd}
        \bullet\arrow[d]& \arrow[l]\bullet\arrow[d] & ...\arrow[l] & \bullet\arrow[l]\arrow[d] & \color{red}\bullet\arrow[l]\arrow[d, red] \\
        \bullet\arrow[r] & \bullet\arrow[r] & ...\arrow[r] & \bullet\arrow[r] & \color{red}\bullet
    \end{tikzcd}
    \]
    In particular, we deduce that $M^{(\catop{[n]}\ast [n], *)}\simeq (M\circ p)^{([n], *)}$ as wanted. Finally, to recover the colimit defining $X_n$, remark that there is an equivalence
    \[
        \LaceEq(\Dcal, N)\simeq \colim_{X\in\Dcal^\simeq}\LoopInfty N(X, X)
    \]
    since $\LaceEq(\Dcal, N)\to\Dcal^\simeq$ is the left fibration classifying $\LoopInfty N\circ\Delta$, where $\Delta\colon \Ccal^\simeq\to(\catop{\Ccal})^\simeq\times\Ccal^\simeq$ is the diagonal (see \cite[Lemma 3.7]{SaunierHeart} for a proof). Applied to the rewriting we produced, we get an equivalence
    \[
        \baBar_{e(n)}(\Ccal, M)\simeq \colim_{Y\in(\TwAr(\Ccal)^{[n]})^\simeq} \LoopInfty M\circ p(Y(0)) 
    \]
    where we recognize on the right hand side the definition of $X_n$. This concludes.
\end{proof}

Let us mention that Haugseng derived in a similar manner a general formula for (co)ends in \cite[Theorem 1.1]{HaugsengCoend}. We thank Fabian Hebestreit for pointing us to this reference. We also note that we will prove a variant of the above for $\THH$ as well (see Remark~\ref{rmq:thh-bar-construction}).

\subsection{Trace-like functors}
The inclusion of the 0-simplices in the cyclic Bar construction yields a natural transformation 
$\begin{tikzcd}[cramped]\LaceEq\arrow[r] & \uTHH\end{tikzcd}$.
Our goal is to show that this natural transformation exhibits its target as the initial functor under $\LaceEq$ for a certain property, which we call being trace-like. In fact, we will show more generally that this happens for the natural transformation $F\to\cyc(F)$, where $\cyc(F)$ is cyclic bar construction adapted to a general functor $F\colon \CatL\to\Sp$, taking the role of $\LaceEq$.

\begin{defi}\label{defi:trace-equivalence}
     Let $f, g\colon (\Ccal, M)\to(\Dcal, N)$ be two laced functors. A \textit{trace homotopy} from $f$ to $g$ is a functor $H\colon (\Ccal, M)\to(\Dcal, N)^{([1], *)}$ such that the following diagram commutes:
     \[
        \begin{tikzcd}
            & & (\Dcal, N) \\
            (\Ccal, M)\arrow[r, "H"]\arrow[rrd, bend right=15, "g"]\arrow[rru, bend left=15, "f"'] & (\Dcal, N)^{([1], *)}\arrow[ru, "d_0"']\arrow[rd, "d_1"] .\\
            & & (\Dcal, N)
        \end{tikzcd}
     \]
     A laced functor $f\colon (\Ccal, M)\to(\Dcal, N)$ is a \textit{strict trace equivalence} if there exists a laced functor $g\colon (\Dcal, N)\to(\Ccal, M)$ such that $g\circ f$ and $f\circ g$ are trace homotopic to the respective identities.
\end{defi}

By the adjunction of tensor and cotensor, a laced homotopy from $f$ to $g$ can equivalently by encoded by a laced functor $(\Ccal,M)_{([1],*)} \to (\Dcal,N)$ fitting into a commutative diagram
 \[
    \begin{tikzcd}
        & & (\Ccal, M) \arrow[lld, bend right=15, "g"]\arrow[ld, "\widehat{d_0}"] \\
        (\Dcal, N) & \arrow[l, "H"'](\Ccal, M)_{([1], *)}\\
        & & (\Ccal, M)\arrow[llu, bend left=15, "f"']\arrow[lu, "\widehat{d_1}"'] .
    \end{tikzcd}
 \]
Equivalently, this is the data of a 1-simplex in the simplicial space
\[
    \baBar_\bullet\FunInt((\Ccal, M), (\Dcal, N))\simeq\Map_{\CatL}((\Ccal, M), (\Dcal, N)^{([\bullet],*)}) = \Map_{\CatL}((\Ccal, M)_{([\bullet],*)}, (\Dcal, N)).
\]

\begin{ex}\label{ex:d-is-trace-equivalence}
    Both $d_0$ and $d_1$ are strict trace equivalences, with the same trace inverse $s\colon (\Ccal, M)\to(\Ccal, M)^{([1], *)}$. Indeed, since $s\circ d_0 = s\circ d_1 = \id$ already, it suffices to find a laced arrow
    \[
        \begin{tikzcd}
            H\colon (\Ccal, M)^{([1], *)}\arrow[r] & (\Ccal, M)^{([1]\times[1], *)}
        \end{tikzcd}
    \]
    such that $d_0\circ H = \id$ and $d_1\circ H = d_0\circ s$. We can simply pick $H$ to be induced by the first projection $[1]\times[1]\to[1]$. Similarly, using the dual characterization of laced homotopies we see that $\widehat{d_0},\widehat{d_1}\colon (\Ccal,M) \to (\Ccal,M)_{([1],*)}$ are both laced equivalences with the same trace inverse $\hat{s}\colon (\Ccal,M)_{([1],*)} \to (\Ccal,M)$.
\end{ex}

Let us give a non-trivial example of a strict trace equivalence:

\begin{lmm} \label{TraceEquivalencesFromAdjunction}
    Let $L\colon \Ccal \adj \Dcal \cocolon R$    
be an adjunction between stable categories, and let $M\colon \catop{\Dcal}\otimes\Ccal\to\Sp$ be exact. Then, the unit $\eta$ and the counit $\varepsilon$ of the adjunction promote $L$ and $R$ to laced functors $L_M:=(L, M\circ(\catop{\id}\times\eta))$ and $R_M:=(R, M\circ(\catop{\varepsilon}\times\id))$ such that
    \[
        \begin{tikzcd}[column sep=large]
            (\Ccal, M\circ(\catop{L}\times\id))\arrow[rr, shift left=2, "{L_M}"] & &(\Dcal, M\circ(\catop{\id}\times R))\arrow[ll, shift left=2, "{R_M}"]
        \end{tikzcd}
    \]
    are trace inverses to one another.
\end{lmm}
\begin{proof}
    What we have described is clearly a pair of laced functors, so it remains to check that both composite are trace equivalent to the identity. The arguments for the two composites are dual to each other, and so let us simply provide a commutative diagram
    \[
        \begin{tikzcd}
            & & (\Ccal, M\circ(\catop{L}\times\id)) \\
            (\Ccal, M\circ(\catop{L}\times\id))\arrow[r, "H"]\arrow[rrd, bend right=10, "R_M\circ L_M"]\arrow[rru, bend left=10, "\id"'] & (\Ccal, M\circ(\catop{L}\times\id))^{([1], *)}\arrow[ru, "d_0"']\arrow[rd, "d_1"] \\
            & & (\Ccal, M\circ(\catop{L}\times\id))
        \end{tikzcd}
    \]
    exhibiting $R_M\circ L_M$ as trace homotopic to $\id$. 
    
    On underlying stable categories, we let $H\colon \Ccal\to\Ccal^{[1]}$ be the functor sending $X$ to $\eta_X\colon (X\to RL(X))$, and unpacking the definitions via Example \ref{CotensorPointEx}, the natural transformation we have to supply is given on objects $X, Y\in\Ccal$ is of the form
    \[
        \begin{tikzcd}
            M(L(X), Y)\arrow[r] & M(LRL(X), Y),
        \end{tikzcd}
    \]
    and so we take the map induced by pulling back in the first argument along the counit 
    \[
        \begin{tikzcd}
        \varepsilon_{L(X)}\colon LRL(X) \arrow[r] & L(X)
        \end{tikzcd}
    \]
    of $L(X)$. 
    The commutativity of the lower triangle is then by construction and the commutativity of the upper triangle is provided by the triangle identities of the adjunction $L\dashv R$. 
\end{proof}

\begin{defi}
    A functor $F\colon \CatL\to\Ecal$ is \textit{trace-like} if it inverts strict trace equivalences.
\end{defi}

Our first result shows a simpler criterion to check that a functor is trace-like.

\begin{prop} \label{TestSampleTrLike}
    Let $F\colon \CatL\to\Ecal$ be such that either of the two following conditions is met:
    \begin{itemize}
        \setlength\itemsep{0em}
        \item For every laced category $(\Ccal, M)$, $F$ sends $(\Ccal, M)^{([1],*)}\xrightarrow{d_0}(\Ccal, M)$ to an equivalence.
        \item For every laced category $(\Ccal, M)$, $F$ sends $(\Ccal, M)\xrightarrow{\widehat{d_0}}(\Ccal, M)_{([1],*)}$ to an equivalence.
    \end{itemize}
    Then, $F$ is trace-like. In particular, both of the above are realized.
\end{prop}
\begin{proof}
    First note that both conditions are implied by being trace-like, since both $d_0$ and $\widehat{d_0}$ are strict trace equivalences (see Example~\ref{ex:d-is-trace-equivalence}). For the inverse implication, let us treat just the first case, the second being dual. Let $f\colon (\Ccal, M)\to(\Dcal, N)$ be a strict trace equivalence with trace inverse $g\colon (\Dcal, N)\to(\Ccal, M)$, and let $H$ and $H'$ be the two trace-homotopies to the identity. Since $F$ maps $d_0$ to an equivalence by assumption and $d_0\circ H =\id$ it follows that $F$ maps $H$ to an equivalence, and so also maps $d_1\circ H = g\circ f$ to an equivalence. The same argument applied to $H'$ shows $F$ also inverts $f\circ g$. It follows that $F(f)$ is an equivalence and so $F$ is trace-like. 
\end{proof}

\begin{lmm} \label{CycIsMadeOfTraceEq}
    Let $(\Ccal, M)$ be a laced category and $\rho : [n]\to[m]$ a map in $\Delta$. The following maps are strict trace equivalences:
    \begin{enumerate}
        \setlength\itemsep{0.3em}
        \item $\rho^*:(\Ccal, M)^{([m],*)}\longrightarrow(\Ccal, M)^{([n],*)}$
        \item $\rho_*:(\Ccal, M)_{([n],*)}\longrightarrow(\Ccal, M)_{([m],*)}$
    \end{enumerate}
\end{lmm}
\begin{proof}
    We treat the first case, the second is dual. Strict trace equivalences are stable by composition and if $f$ is a strict trace equivalence and $g\circ f=\id$ or $f\circ g=\id$, then so is $g$. Hence, given the structure of $\Delta$, it is sufficient to treat the case of the injective maps $i_n\colon [0] \to [n]$ and $i_0\colon [0] \to [n]$ sending $0 \in [0]$ to $n \in [n]$ and $0 \in [n]$, respectively.  
    We again treat just the first case, as the second is dual.
    Consider the one-sided inverse 
    $p^*\colon (\Ccal, M)^{([0],*)}\to(\Ccal, M)^{([n],*)}$ 
    induced by the unique map $p\colon [n]\to[0]$.
    It then suffices to find a trace homotopy between the composite $p^* \circ i_n^*$ and $\id$. Explicitly, this amounts a laced functor $(\Ccal, M)^{([n], *)}\to(\Ccal, M)^{([n]\times[1], *)}$ fitting in a commutative diagram
     \[
        \begin{tikzcd}
            & & (\Ccal, M)^{([n], *)} \\
            (\Ccal, M)^{([n], *)}\arrow[r]\arrow[rrd, bend right=10, "p^* \circ i_n^*"]\arrow[rru, bend left=10, "\id"'] & (\Ccal, M)^{([n]\times[1], *)}\arrow[ru, "d_0"']\arrow[rd, "d_1"] \\
            & & (\Ccal, M)^{([n], *)}.
        \end{tikzcd}
     \]
	But there is already such a commutative diagram at the level of posets: 
    \[
        \begin{tikzcd}
            {}[n]\arrow[rd]\arrow[rrd, bend left=15, "\id"] \\
             & {}[n]\times[1]\arrow[r, "H"] & {}[n] \\
            {}[n]\arrow[ru]\arrow[rru, bend right=15, "i_n\circ p"']
        \end{tikzcd}
    \]
    where $H\colon [n]\times[1]\to[n]$ maps a tuple $(k, i)$ to $k$ if $i=0$ and to $n$ if $i=1$. 
    The functoriality of the cotensor then concludes.
\end{proof}

As a consequence of Lemma \ref{CycIsMadeOfTraceEq}, the simplicial object $(\Ccal, M)^{([\bullet],*)}$ has its faces and degeneracies being strict trace equivalences.
Thus Proposition \ref{TestSampleTrLike} implies that $F$ is trace-like if and only if it sends the simplicial object $(\Ccal, M)^{([\bullet],*)}$ to a constant simplicial object. 

\begin{defi}\label{defi:cyclic-bar}
Let $\Ecal$ be a category admitting geometric realizations and $F\colon \CatL\to\Ecal$ a functor. We let $\cyc(F)$ denote the functor given pointwise by the geometric realization
    \[
        \cyc(F)(\Ccal, M):=\left|F((\Ccal, M)^{([\bullet],*)})\right|.
    \]
\end{defi}

The association $F\mapsto\cyc(F)$ upgrades to a functor $\cyc\colon \Fun(\CatL, \Ecal)\to \Fun(\CatL, \Ecal)$, which we call the \textit{cyclic bar construction functor}. The inclusion of 0-simplices provides a natural transformation $F\to\cyc(F)$, which is itself natural in $F$, hence providing a natural transformation $\id\Rightarrow\cyc$.

\begin{rmq}\label{bar-is-bar}
We may reformulate Proposition \ref{CyclicBarInTCatEx} as saying that $\uTHH\simeq\cyc(\LaceEq)$.
\end{rmq}

\begin{prop} \label{CycIsTRLike}
Let $\Ecal$ be a category admitting geometric realizations and $F\colon \CatL\to\Ecal$ a functor. 
Then $\cyc(F)$ is trace-like.
\end{prop}
\begin{proof} 
    Using Proposition \ref{TestSampleTrLike}, we are reduced to showing that the simplicial map 
    \[
        \begin{tikzcd}
          F((\Ccal, M)^{([1]\times[\bullet],*)}) \ar[r]  &  F((\Ccal, M)^{([\bullet],*)}),
        \end{tikzcd}
    \]
    induced by the inclusions $[n] = \{0\} \times [n] \subseteq [1] \times [n]$, is an equivalence on geometric realizations.
    This map has a one-sided inverse induced by the projections $[1] \times [n] \to [n]$, and so it suffices to show that the associated composite
    \[
        \begin{tikzcd}
            F((\Ccal, M)^{([1]\times[\bullet],*)})\arrow[r] & F((\Ccal, M)^{([\bullet],*)})\arrow[r] & F((\Ccal, M)^{([1]\times[\bullet],*)})
        \end{tikzcd}
    \]
    is simplicially homotopic to the identity. Now note that any levelwise induced functor on simplicial objects sends simplicial homotopies to simplicial homotopies, and so it will suffice to verify that the corresponding composite map $[1] \times [\bullet] \to [\bullet] \to [1] \times [\bullet]$ of simplicial objects in $\Cat\op$ is simplicially homotopic to the identity.     
     Explicitly, this means providing maps $h_i\colon [1]\times[n+1] \to [1]\times[n]$ for $i=0,...,n$ satisfying suitable boundary relations. Unwinding the definitions, one can take 
     \[ h_i(\epsilon,k) = \left\{\begin{matrix} (0,k) & k \leq i \\ (\epsilon,k-1) & k > i, \end{matrix}\right. \]
     thus concluding the proof.
\end{proof}

Applying the above proposition to $F = \LaceEq$ and using Remark~\ref{bar-is-bar} we get
\begin{cor}
    The functor $\uTHH\colon \CatEx \to \Spaces$ is trace-like.
\end{cor}

We now establish the universal property of the cyclic bar construction.

\begin{thm}\label{thm:universal-bar}
    Let $\Ecal$ be a category admitting geometric realizations and $F\colon \CatL\to\Ecal$ a functor. Then the natural transformation $\eta_F\colon F\Rightarrow\cyc(F)$ exhibits its target as the initial trace-like functor under $F$.
\end{thm}

\begin{cor} \label{uTHHUP}
    The natural transformation $\LaceEq\Rightarrow\uTHH$ exhibits the latter as the initial trace-like invariant under $\LaceEq$. 
\end{cor}

\begin{rmq}\label{uTHHUP-2}
    Since $\SigmaInftyP\colon\Spaces\to\Sp$ commutes with the formation of the cyclic-bar construction, the same applies to natural transformation of spectrum valued functors $\SpLaceEq\Rightarrow\SigmaInftyP\uTHH$.
\end{rmq}

\begin{proof}[Proof of Theorem~\ref{thm:universal-bar}]
    We first claim that for any $F$, the two arrows $\eta_{\cyc(F)}$ and $\cyc(\eta_F)$ are equivalent in the arrow category of $\Fun(\CatL,\Ecal)$. For this, remark that 
    \[
        ((\Ccal, M)^{([m], *)})^{([n],*)}=(\Ccal, M)^{([m] \times [n],*)}
    \]
    so that $\cyc(\cyc(F))$ can be viewed as the colimit of the bisimplicial object $F(((\Ccal, M)^{([\bullet_1]\times[\bullet_2],*)})$, with the two maps $\eta_{\cyc(F)}$ and $\cyc(\eta_F)$ are simply induced by the horizontal and vertical inclusions in $\Delta\times\Delta$. Thus, post-composing with the self equivalence induced by the flip involution 
    $\Delta\times\Delta\to\Delta\times\Delta$     
    switches between $\eta_{\cyc(F)}$ and $\cyc(\eta_F)$, and so the two are equivalent objects in $\Ar(\Fun(\CatL,\Ecal))$.  
        
        Now by Proposition \ref{CycIsTRLike} we have that $\cyc(-)$ takes values in trace-like functors and from Proposition~\ref{CycIsMadeOfTraceEq} we know that $\eta_F$ is an equivalence for $F$ trace-like. This implies, on the one hand, that the image of $\cyc$ consists exactly of trace-like functors, and on the other hand that the two equivalent arrows  
        $\eta_{\cyc(F)}$ and $\cyc(\eta_F)$ 
        are equivalences for any $F$. The functor $\cyc(-)$ and the natural transformation $\id \Rightarrow \cyc(-)$ thus satisfy the localisation criterion of~\cite[Proposition 5.2.7.4]{HTT}, and so the desired result follows. 
\end{proof}

\begin{prop} \label{uTHHInitialLaxMonoidal}
    The natural transformation $\LaceEq \Rightarrow \uTHH$ refines to one of lax symmetric monoidal functors, making     $\uTHH\colon \CatL\to\Spaces$ the initial lax symmetric monoidal trace-like functor.
\end{prop}
\begin{proof}
    Let $p\colon \CatL \to \mathcal{T}$ be the localisation of $\CatL$ at strict trace equivalences. Comparing universal properties it then follows from Theorem~\ref{thm:universal-bar} that
    \[ \Map_{\mathcal{T}}(p(\Ccal,M),p(\Dcal,N)) = \cyc(\Map((\Ccal,M),-))(\Dcal,N) = \left|\Map((\Ccal, M),(\Dcal,N)^{([\bullet],*)})\right| .\]
    In particular, $\mathcal{T}$ is locally small. 
    Now by Proposition \ref{TestSampleTrLike} the functor $p$ also exhibits $\mathcal{T}$ as the localisation at the collection of arrows $\{(\Ccal, M)_{([1],*)}\to(\Ccal, M)\}$. Since this collection is closed under tensoring with arbitrary $(\Dcal,N) \in \CatL$ the localisation $p$ is multiplicative, so that $\mathcal{T}$ inherits a symmetric monoidal structure under which $p$ is symmetric monoidal. For any symmetric monoidal category $\Ecal$, restriction along $p$ then induces a full inclusion among the associated the Day convolution operads 
    \[ \Fun(\mathcal{T},\Ecal)^{\otimes} \subseteq \Fun(\CatL,\Ecal)^{\otimes},\] 
    whose essential image is spanned by the trace-like functors. Passing to algebra objects we hence obtain an equivalence between the category of lax symmetric monoidal functors $\mathcal{T} \to \Ecal$ and that of trace-like lax symmetric monoidal functors $\CatL \to \Ecal$. The desired result now follows from from \cite[Corollary 6.8]{Nikolaus}, since, as a functor $\mathcal{T} \to \Spaces$, $\uTHH$ is corepresented by the unit (as can be seen by plugging in $(\Ccal,M) = (\Spaf,\map)$ in the above mapping space formula), while $\LaceEq$ is corepresented by the unit on $\CatL$.
\end{proof}

\subsection{The universal property of THH}\label{sec:UPTHH}

In the previous section, we have shown that the natural transformation $\LaceEq\Rightarrow\uTHH$ exhibits $\uTHH$ as the initial trace-like functor under $\LaceEq$. In order to translate this knowledge to stable $\THH$ we need to understand the natural transformation $\SigmaInftyP\uTHH\Rightarrow\THH$. 

Let us first deal with a simple case: recall from~\cite[Section 6]{HA} that if $F\colon \Ccal\to\Dcal$ is a (not necessarily exact) functor between stable categories such that $\Dcal$ admits sequential colimits, then there is a natural transformation $F\to\nexcPart{1}F$ whose target $\nexcPart{1}F$ is exact, and which is initial amongst such transformations. We call $\nexcPart{1}F$ is the \textit{exact approximation} of $F$.

\begin{lmm} \label{StableSigmaLoop}
    Let $\Ccal$ be a stable category, $\Dcal$ and $\infty$-category with finite limits, and $F\colon \Ccal\to\Sp(\Dcal)$ an exact functor. Then the natural transformation $\eta\colon \SigmaInftyP\OmegaInfty F\to F$ exhibits its target as the exact approximation of its source.
\end{lmm}
\begin{proof}
    Recall that the universal property of $\OmegaInfty\colon \Sp(\Dcal)\to\Dcal$ implies that
    \[
        \begin{tikzcd}
            (\OmegaInfty)_*\colon \FunEx(\Ccal, \Sp(\Dcal))\arrow[r, "\simeq"] & \Fun^{\mathrm{LEx}}(\Ccal, \Dcal)
        \end{tikzcd}
    \]
    is an equivalence, where $\Fun^{\mathrm{LEx}}(-,-)$ denotes the $\infty$-category of left exact (that is, finite limit preserving) functors. In particular, for every exact $G\colon \Ccal\to\Sp(\Dcal)$, the top horizontal arrow in the diagram
    \[
        \begin{tikzcd}
            \Nat(F, G)\arrow[r, "\simeq"]\arrow[d, "\eta^*"] & \Nat(\LoopInfty F, \LoopInfty G)\arrow[d]\ar[dr,equal] &  \\
            \Nat(\SigmaInftyP\OmegaInfty F, G)\arrow[r]\ar[rr,bend right=15,"\simeq"'] & \Nat(\LoopInfty\SigmaInftyP\OmegaInfty F, \LoopInfty G) \ar[r] & \Nat(\LoopInfty F, \LoopInfty G)
        \end{tikzcd}
    \]
is an equivalence, where we note that the commutative triangle on the right is given by the triangle identity of the adjunction $\SigmaInftyP \dashv \OmegaInfty$ and the bottom horizontal composite is an equivalence by adjunction. We conclude that $\eta^*$ is an equivalence, yielding the desired result.
\end{proof}

\begin{cor}\label{cor:THH-exact-approx}
For a fixed stable category $\Ccal$, the natural transformation
\[
    \SigmaInftyP\uTHH(\Ccal, -)\Rightarrow\THH(\Ccal, -)
\]
of functors $\BiMod(\Ccal) \to \Sp$ exhibits its target as the exact approximation of its source.
\end{cor}
\begin{proof}
By Lemma~\ref{StableSigmaLoop}, for each fixed $X, Y\in\Ccal$, the map 
\[ \SigmaInftyP\OmegaInfty M(X, Y)\to M(X, Y),\]
considered as a natural transformation of functors in the $M$ entry, exhibits its target as the exact approximation of its source (indeed, $\ev_{(X, Y)}:M\mapsto M(X, Y)$ is exact and pre-composition with exact functors preserves exact approximations). 

By definition, the natural transformation $\SigmaInftyP\uTHH(\Ccal, -)\Rightarrow\THH(\Ccal, -)$ is obtained by forming the colimit ranging over $(X\to Y)\in\TwAr(\Ccal)$ of the above natural transformation. The formation of exact approximations commutes with colimits since it is left adjoint to the inclusion of exact functors into all functors, hence $\THH(\Ccal, -)$ is the exact approximation of $\SigmaInftyP\uTHH(\Ccal, -)$, as desired. 
\end{proof}

To get the universal property when $\Ccal$ varies, we need to introduce a fibered version of the exact approximation. More formally, let us give the following definition:

\begin{defi}
Let $\Ecal$ be a category and $\Fcal\colon \CatL\to\Ecal$ a functor. We say that $\Fcal$ is 
    \begin{itemize}
        \item \textit{tangentially reduced} if $\Ecal$ is pointed and for every stable $\Ccal$, the restriction $\Fcal_\Ccal\colon \BiMod(\Ccal) \to\Ecal$ is reduced, i.e., sends the zero bimodule to zero, and
        \item \textit{tangentially exact} if $\Ecal$ is stable and for every stable $\Ccal$, the restriction $\Fcal_\Ccal\colon \BiMod(\Ccal)\to\Ecal$ is exact. 
    \end{itemize}
\end{defi}

In an earlier draft, we had used the adjective \textit{fibrewise} (or sometimes \textit{fiberwise}) in place of \textit{tangentially}. To align better this manuscript with its sequels, we decided against this name, since in the generalizations of $\CatL$ we will consider will not quite have a \textit{fibrewise} direction but still retain a \textit{tangential}.

\begin{ex}
    The functor $\THH\colon \CatL \to \Sp$ is tangentially exact. 
\end{ex}

An important example of a tangentially reduced functor (which is not tangentially exact) is the following.
\begin{defi}
The \emph{cyclic K-theory functor} $\Kth^{\cyc}\colon \CatL \to \Sp$ is defined by  
\[ \Kth^{\cyc}(\Ccal,M) := \cofib[\Kth(\Ccal) \to \Klace(\Ccal,M)].\]
\end{defi}
In fact, the passage from $\Klace$ to $\Kth^{\cyc}$ is an instance of a general procedure which yields a universal tangentially reduced replacement of a given functor on $\CatL$.

\begin{defi}
    Let $\Ecal$ be a stable category and $\eta\colon \Fcal \Rightarrow \Fcal'$ a natural transformation of functors $\CatL \to \Ecal$. We say that $\eta$ exhibits $\Fcal'$ as the tangentially reduced (resp. tangentially exact) approximation of $\Fcal$ if $\Fcal'$ is tangentially reduced (resp. tangentially exact) and the initial such functor under $\Fcal$.
\end{defi}

In classical Goodwillie calculus, reduced approximations are realized by taking the complement $\overline{F}(X)$ of the direct summand $F(0)$ of $F(X)$, and exact approximation are realized by first performing the reduced approximation $\overline{F}$ and then taking excisive approximation via the formula
$\colim_n\Omega^n \overline{F}(\Sigma^nX)$. This requires the target to admit sequential colimits, and also works if the target $\Ecal$ is not stable, as long as sequential colimits commute with finite limits (what is known as being differentiable, see 
\cite[Definition 6.1.1.6]{HA}). This idea still works in the bundled version, as we now explain.

\begin{lmm}\label{lmm:fibrewise-reduced}
    Let $\Ecal$ be a stable category and $\Fcal\colon \CatL\to\Ecal$ a functor.  
    Then, $\Fcal$ admits a tangentially reduced approximation  
    $\Fcal\Rightarrow \overline{\Fcal}$.
    Moreover, this natural transformation splits in the following way, for every $(\Ccal, M)\in\CatL$:
    \[
        \Fcal(\Ccal, M)\simeq \overline{\Fcal}(\Ccal, M)\oplus \Fcal(\Ccal, 0)
    \]
\end{lmm}

\begin{lmm} \label{FiberwiseDerivative}
    Let $\Ecal$ be a stable category with sequential colimits and $\Fcal\colon \CatL\to\Ecal$ a functor. 
    Then, $\Fcal$ admits  
    a tangentially exact approximation  
    $\Fcal\Rightarrow \fbwnexcPart{1}\Fcal$.  
	Moreover, for each stable category $\Ccal$, the induced natural transformation $\Fcal|_{\BiMod(\Ccal)} \Rightarrow \fbwnexcPart{1}\Fcal|_{\BiMod(\Ccal)}$ exhibits $\fbwnexcPart{1}\Fcal|_{\BiMod(\Ccal)}$ as the exact approximation of $\Fcal$. 
\end{lmm}

\begin{cor}\label{cor:fiberwise-derivative}
The inclusion $\Fun^{\bex}(\CatL,\Ecal) \subseteq \Fun(\CatL,\Ecal)$ of tangentially exact functors admits a left adjoint
\[ \fbwnexcPart{1}\colon \Fun(\CatL,\Ecal) \to \Fun^{\bex}(\CatL,\Ecal), \]
and this left adjoint is compatible with restriction along $\BiMod(\Ccal) \subseteq \CatL$ for every stable category $\Ccal$.
\end{cor}

\begin{rmq}\label{rmq:summary-lemmas}
Lemmas~\ref{lmm:fibrewise-reduced} and \ref{FiberwiseDerivative} tell us in particular that if $\Fcal\colon \CatL \to \Ecal$ is a functor to a stable category $\Ecal$ with sequential colimits then $\Fcal$ admits a tangentially reduced approximation $\overline{\Fcal}$ and a tangentially exact approximation $\fbwnexcPart{1}\Fcal$ satisfying
\[ \Fcal(\Ccal,M) = \Fcal(\Ccal,0) \oplus \overline{\Fcal}(\Ccal,M) \]
and
\[
    \fbwnexcPart{1}\Fcal(\Ccal, M)=\colim_n \Omega^n \overline{\Fcal}(\Ccal, \Sigma^n M).
\]
\end{rmq}

\begin{proof}[Proof of Lemmas~\ref{lmm:fibrewise-reduced} and \ref{FiberwiseDerivative}]
    Consider the inclusion of the full subcategory spanned by tangentially reduced functors:
    \[
        \begin{tikzcd}
            \Fun_*(\CatL, \Ecal)\arrow[r, hook] & \Fun(\CatL, \Ecal)
        \end{tikzcd}
    \]
    We want to show this functor admits a left adjoint. 
    Since $\fgt\colon \CatL\to\CatEx$ is a cocartesian fibration classifying $\Bimod(-)$, it follows from combining \cite[Corollary 3.2.2.13]{HTT} and \cite[Proposition 7.3]{GepnerHaugsengNikolaus} that this inclusion can be rewritten
    \[
        \begin{tikzcd}
            \Gamma(\Un^{\cart}(\Fun_*(\Bimod(-), \Ecal)))\arrow[r] & \Gamma(\Un^{\cart}(\Fun(\Bimod(-), \Ecal)))
        \end{tikzcd}
    \]
    where $\Un^{\cart}(F)$ denotes the cartesian fibration classifying a functor $F$, and $\Gamma$ denotes the categories of sections of a given fibration. 
    Applying the dual version of \cite[Proposition 5.1]{HorevYanovski} for cartesian fibrations, we get that the existence of a left adjoint follows from the existence of a left adjoint for each fibre and that the restriction of the global left adjoint recovers the fiberwise adjoint. The claim now follows from the usual formula for the reduced approximation of a functor $F$ among stable categories, given by taking the cofibre of the split inclusion $F(0) \to F(-)$. 
        
    The proof of Lemma~\ref{FiberwiseDerivative} follow the exact same logic as the argument above, considering instead tangentially exact functors instead of tangentially reduced. Again, the pointwise claim and the formula are classical and can be found as the case $n=1$ of \cite[Lemma 6.1.1.33]{HA}.
\end{proof}

\begin{cor}\label{GDPreserveTRL}
    Let $\Fcal \colon \CatL \to \Ecal$ be a functor to a stable category (and for the statements involving $\fbwnexcPart{1}\Fcal$, assume in addition that $\Ecal$ has sequential colimits). Then
    \begin{itemize}
        \item If $\Fcal$ is additive then its tangentially reduced approximation $\overline{\Fcal}$ and its tangentially exact approximation $\fbwnexcPart{1}\Fcal$ are again additive.
        \item If $\Fcal$ is trace-like then its tangentially reduced approximation $\overline{\Fcal}$ and its tangentially exact approximation $\fbwnexcPart{1}\Fcal$ are again trace-like.
    \end{itemize}
\end{cor}
\begin{proof}
    In both cases this follows from the formulas in Remark~\ref{rmq:summary-lemmas}. 
    
    For (1), observe that if $(\Acal, N) \subseteq (\Ccal,M) \supseteq (\Bcal, P)$ is a laced semi-orthogonal decomposition then $(\Acal,0) \subseteq (\Ccal,0) \supseteq (\Bcal,0)$ is a laced semi-orthogonal decomposition and $(\Acal, \Sigma^n N) \subseteq (\Ccal,\Sigma^n M) \supseteq (\Bcal,\Sigma^n P)$ is a laced semi-orthogonal decomposition for every $n \in \mathbb{Z}$. 
    
    Similarly, for (2) we point out that if $(\Ccal,M) \to (\Dcal,M)$ is a strict trace equivalence then $(\Ccal,0) \to (\Dcal,0)$ is a strict trace equivalence and $(\Ccal,\Sigma^n M) \to (\Dcal,\Sigma^n N)$ is a strict trace equivalence for every $n \in \mathbb{Z}$.
\end{proof}

Combining Lemma~\ref{FiberwiseDerivative} with Corollary~\ref{cor:THH-exact-approx}, we deduce:

\begin{cor} \label{StableuTHHisTHH}
    The natural transformation $\SigmaInftyP\uTHH\Rightarrow\THH$ exhibits $\THH$ 
    as the tangentially exact approximation of $\SigmaInftyP\uTHH$. 
\end{cor}

Combining Corollary~\ref{uTHHUP-2}, Corollary~\ref{StableuTHHisTHH}, and Corollary~\ref{GDPreserveTRL}, we now conclude:
  
\begin{cor} \label{StableuTHHisTHH-2}
The tangentially exact functor $\THH$ is trace-like, and the natural transformation $\SpLaceEq\Rightarrow\THH$ exhibits $\THH$ as the initial tangentially exact trace-like functor under its source.
\end{cor}

\begin{rmq}\label{rmq:thh-bar-construction}
Let $\Ecal$ be a stable category with sequential colimits.
By Corollary~\ref{GDPreserveTRL}, the tangentially exact approximation operation on $\Fun(\CatL,\Ecal)$ preserves trace-like functors. At the same time, it follow from directly from its explicit formula that the cyclic Bar construction $\cyc$ of Definition~\ref{defi:cyclic-bar} preserves tangentially exact functors. We conclude that, as two left Bousfield localisations of $\Fun(\CatL,\Ecal)$, tangentially exact approximation and cyclic Bar construction commute with each other, and their composite $\cyc \circ \fbwnexcPart{1} = \fbwnexcPart{1} \circ \cyc$ is a left Bousfield localisation of $\Fun(\CatL,\Ecal)$ whose local objects are the tangentially exact trace-like functors. 
By Corollary~\ref{StableuTHHisTHH-2} we then have that
\[ \THH = \fbwnexcPart{1}\cyc(\SpLaceEq) = \cyc\fbwnexcPart{1}(\SpLaceEq) .\]
Explicitly, $\fbwnexcPart{1}(\SpLaceEq)$ is given by 
\[ \fbwnexcPart{1}\SpLaceEq(\Ccal,M) = \colim_{X \in \Ccal^\simeq}M(X,X)\] 
and the identification $\THH = \cyc\fbwnexcPart{1}(\SpLaceEq)$ yields a bar construction type formula 
    \[
        \THH(\Ccal, M)\simeq \left|\begin{tikzcd}...\arrow[r, shift left=2]\arrow[r, shift right=2]\arrow[r] & \displaystyle\colim_{[f\colon X \to Y]\in(\Ccal^{\Delta^1})^\simeq} M(Y, X)\arrow[r, shift left=1]\arrow[r, shift right=1] & \displaystyle\colim_{X\in\Ccal^\simeq}M(X, X)\end{tikzcd}\right|
    \]
for $\THH$.
We may view $\fbwnexcPart{1}(\SpLaceEq)$ as some naive categorification of the trace of a matrix, whereas $\THH$ is an actual trace: the trace of $M$ seen as an endomorphism of the dualizable $\Ind(\Ccal)$ in $\PrExCat{L}$, by \cite[Proposition 4.5]{HoyoisScherotzkeSibilla}.
\end{rmq}

\subsection{The trace map and the Dundas-McCarthy theorem}\label{subsec:THH-derivative}

Our goal in this section is to construct a canonical trace map
\[ \Klace(\Ccal,M) \to \THH(\Ccal,M) \]
in the setting of laced categories and show that it exhibits $\THH(\Ccal,-)$ as the exact approximation (or first Goodwillie derivative) of $\Klace(\Ccal,-)$, thus generalizing the Dundas-McCarthy theorem from \cite{DundasMcCarthy} to the setting of stable categories.
 
As in the non-laced setting, the trace map can be constructed using the universal property of $\Klace$. More precisely, we have shown in Theorem~\ref{KlaceIsUnivAdditive} that the natural transformation $\SpLaceEq\Rightarrow\Klace$ exhibits laced K-theory as the initial additive invariant under $\SpLaceEq$. At the same time, we have a natural transformation $\SpLaceEq \Rightarrow \THH$, which we used in Corollary~\ref{StableuTHHisTHH-2} to endow $\THH$ with a universal property. To construct the trace map it will then suffice to verify that $\THH$ is additive (as a functor on $\CatL$). More generally, we show the following:

\begin{prop}\label{prop:trace-like-is-additive}
    Let $\Ecal$ be a stable category and $\Fcal\colon \CatL\to\Ecal$ a tangentially exact trace-like functor. Then $\Fcal$ is additive.
\end{prop}

\begin{rmq}
The proof of Proposition~\ref{prop:trace-like-is-additive} is an adaptation of a classical argument, notably found in \cite[Section 5.2]{Kaledin} who describes it as the gist of the proof of the localization theorem for Hochschild homology proven by \cite{Keller}. A close variant of this argument also appears in \cite[Theorem 3.4]{HoyoisScherotzkeSibilla}. 
\end{rmq}

\begin{cor}
The functor $\THH\colon \CatL \to \Sp$ is additive.
\end{cor}

\begin{proof}[Proof of Proposition~\ref{prop:trace-like-is-additive}]
	Let $(\Acal, N) \subseteq (\Ccal,M) \supseteq (\Bcal, P)$ be a laced semi-orthogonal decomposition with $i\colon \Acal\to\Ccal$ and $j\colon \Bcal\to\Ccal$ denoting the inclusions.
    Denote by $q$ the right adjoint of $i$; Lemma \ref{TraceEquivalencesFromAdjunction} applied to $i \dashv q$ and the bimodule $M\circ(\id\times i)$ gives a strict trace equivalence 
    \[
        \begin{tikzcd}
            (\Acal, N)\arrow[r] & (\Ccal, M\circ(\id\times iq))
        \end{tikzcd}
    \]
    and from the triangle identities we see that the laced full inclusion $(i,\alpha)\colon (\Acal, N)\to (\Ccal, M)$ factors as follows:
    \[
        \begin{tikzcd}[column sep=large]
            (\Acal, N)\ar[r] & (\Ccal, M\circ(\catop{\id}\times iq))\arrow[r, "{(\id, \hat{\alpha})}"] & (\Ccal, M)
        \end{tikzcd}
    \]
    where $\hat{\alpha}$ is induced by the counit $iq\to\id$.  
    Dually, if we write $p\colon \Ccal \to \Bcal$ for the left adjoint of $j$ then by Lemma~\ref{lmm:laced-adjoint} it canonically refines to a laced functor $(p,\eta)\colon (\Ccal,M) \to (\Bcal,P)$, and we have a similar factorization of $(p,\eta)$ as a composite 
	\[
        \begin{tikzcd}[column sep=large]
            (\Ccal, M)\arrow[r, "{(\id, \hat{\beta})}"]  & (\Ccal, M\circ(\catop{\id}\times jp))) \ar[r]& (\Bcal, P)
        \end{tikzcd}
    \]
     where the second functor is obtain from the construction of Lemma \ref{TraceEquivalencesFromAdjunction} and $\hat{\beta}$ is induced by the unit $\id \Rightarrow jp$.
	Putting these two factorizations together yields a commutative diagram 
    \[
        \begin{tikzcd}[column sep=large, row sep=large]
        (\Acal, N)\arrow[r, "{(i, \alpha)}"]\arrow[d] & (\Ccal, M)\arrow[r, "{(p, \eta)}"]\arrow[d, equal] & (\Bcal, P) \\
        (\Ccal, M\circ(\catop{\id}\times iq))\arrow[r, "{(\id, \hat{\alpha})}"] & (\Ccal, M)\arrow[r, "{(\id, \hat{\beta'})}"] & (\Ccal, M\circ(\catop{\id}\times jp))\arrow[u]
        \end{tikzcd}
    \]
	where the bottom horizontal sequence is exact in $\BiMod(\Ccal) = \CatL \times_{\CatEx} \{\Ccal\}$ (since the sequence $iq\to\id\to jp$ is exact)
    and the vertical maps are strict trace equivalences. Since $\Fcal$ is exact on each fibre and inverts strict trace equivalences, it follows that the top row is also sent to an exact sequence, which concludes.
 \end{proof}

Having established that $\THH$ is additive, the universal property of Theorem\ref{KlaceIsUnivAdditive} implies that the natural transformation $\SpLaceEq \Rightarrow \THH$ extends essentially uniquely to a natural transformation
\[ \tr\colon \Klace \Rightarrow \THH ,\]
which we call the laced trace map.

\begin{rmq}
    The original trace map on the level of $\CatEx$ can recovered from the one above as the composite
    \[ \Kth(\Ccal) \to \Klace(\Ccal,\map) \to \THH(\Ccal,\map) = \THH(\Ccal) .\]
    Indeed, this follows from the uniqueness of both trace maps together with the fact that the composed natural map
    \[ \SigmaInftyP\Ccal^\simeq \to \SpLaceEq(\Ccal,\map) \to \THH(\Ccal,\map) = \THH(\Ccal) \]
    is the one used to construct the trace map in~\cite{BlumbergGepnerTabuada}. 
\end{rmq}

We now remark that under even less hypotheses on $\Fcal$, additivity implies trace-like.

\begin{prop}\label{prop:reduced-additive-is-trace-like}
    Let $\Ecal$ be an additive category and $\Fcal\colon \CatL \to \Ecal$ a functor which is tangentially reduced and additive. Then $\Fcal$ is trace-like.
\end{prop}
\begin{proof}
    By Lemma \ref{CotensorPointSOD} we have a laced semi-orthogonal decomposition $(\Ccal,0) \subseteq (\Ccal, M)^{([1], *)}\supseteq (\Ccal,M)$, where the right hand category is embedded via the laced functor $s\colon (\Ccal,M) \to (\Ccal,M)^{([1],\ast)}$. In the composite
    \[ \Fcal(\Ccal,M) \to \Fcal(\Ccal,M) \oplus \Fcal(\Ccal,0) \to \Fcal((\Ccal,M)^{[1],\ast)}) \]
    the first map is an equivalence by the assumption that $\Fcal$ is tangentially reduced and the second an equivalence by the assumption that $\Fcal$ is additive. Hence the composite is an equivalence so that $\Fcal$ is trace-like by Proposition \ref{TestSampleTrLike}.
\end{proof}

Combining Proposition~\ref{prop:trace-like-is-additive} and Proposition~\ref{prop:reduced-additive-is-trace-like} we get:

\begin{cor}\label{cor:trace-like-iff-additive}
    Let $\Ecal$ be a stable category and $\Fcal\colon \CatL\to\Ecal$ a tangentially exact functor.  
    Then $\Fcal$ is additive if and only if $\Fcal$ is trace-like.
\end{cor}

The following generalization of the Dundas-McCarthy theorem is one of the main results of the present paper:
\begin{thm}[Stable K-theory is THH] \label{DundasMcCarthy}
    The trace map 
    $\tr\colon \Klace\Rightarrow\THH$ 
    exhibits $\THH$ as the initial tangentially exact functor under $\Klace$. In particular, on the fibre over each stable $\Ccal$, the first Goodwillie derivative of the functor $\Kth(\Lace(\Ccal, -))$ is $\THH(\Ccal, -)$.
\end{thm}
\begin{proof}
    Combining Proposition~\ref{prop:trace-like-is-additive} and Corollary~\ref{StableuTHHisTHH-2} and Theorem~\ref{KlaceIsUnivAdditive} we conclude that the map $\Klace \Rightarrow \THH$ exhibits $\THH$ as the initial tangentially exact additive functor under $\Klace$. This must coincide with the tangentially exact approximation of $\Klace$ since $\Kth$ is additive and tangentially exact approximation preserves additive functors, see Corollary~\ref{GDPreserveTRL}.
\end{proof}

One may also wonder if $\THH$ is also the universal trace-like functor under $\Kth$. This is false: the trace-like approximation of $\Kth$ actually coincides with its tangentially reduced approximation, namely, with cyclic K-theory.

\begin{prop}\label{prop:universal-K-cyc}
The cyclic K-theory functor $\Kth^{\cyc}$ is the initial trace-like functor under $\Klace$.
\end{prop}

As $\Kth^{\cyc}$ is the tangentially reduced approximation of $\Klace$ (see Lemma~\ref{lmm:fibrewise-reduced}) the last proposition is a special case of the following more general statement:

\begin{prop}\label{prop:universal-reduced}
Let $\Ecal$ be a stable category, $\Fcal\colon \CatL \to \Ecal$ an additive functor and $\overline{\Fcal}(\Ccal,M) = \cof[\Fcal(\Ccal,0) \to (\Ccal,M)]$ its tangentially reduced approximation. Then the map $\Fcal \Rightarrow \overline{\Fcal}$ exhibits $\overline{\Fcal}$ as the initial trace-like functor under $\Fcal$.
\end{prop}
\begin{proof}
First note that $\overline{\Fcal}$ is additive by Corollary~\ref{GDPreserveTRL} and hence 
trace-like by Proposition~\ref{prop:reduced-additive-is-trace-like}. Now let $\Gcal\colon \CatL \to \Ecal$ be some trace-like functor. We wish to show that the restriction map
\[ \Nat(\overline{\Fcal},\Gcal) \to \Nat(\Fcal,\Gcal) \]
is an equivalence. Since $\CatL$ admits a zero object $(0,0)$ given by the zero stable category equipped with the zero bimodule we may decompose $\Gcal$ canonically as a direct sum $\Gcal := \Gcal_0 \oplus E$ where $E := \Gcal(0,0) \in \Sp$ and $\Gcal_0$ vanishes on $(0,0)$. At the same time, both $\Fcal$ and $\overline{\Fcal}$ vanish on $(0,0)$, and hence map trivially to any constant functor. It will then suffice to show that the restriction map
\[ \Nat(\overline{\Fcal},\Gcal_0) \to \Nat(\Fcal,\Gcal_0) \]
is an equivalence. For this, note that since $\Gcal_0 \simeq \fib[\Gcal \Rightarrow \Gcal(0,0)]$ we have that $\Gcal_0$ is trace-like. But by Lemma~\ref{TraceEquivalencesFromAdjunction} any map in $\CatL$ of the form $(0,0) \to (\Ccal,0)$ is a strict trace equivalence, and since $\Gcal_0(0,0) = 0$ we conclude that $\Gcal_0$ is tangentially reduced, and so the desired result follows.
\end{proof}

\section{Laced theories}
\subsection{Flavours of Verdier sequences}

We dedicate this section to discussing some analogues of Verdier and Karoubi sequences in the laced setting. 

\begin{defi} \label{NaiveVerdierSequences}
    A sequence 
    $\begin{tikzcd}[cramped](\Acal, N)\arrow[r, "{(i, \alpha)}"]& (\Ccal, M)\arrow[r, "{(p, \beta)}"] &(\Bcal, P)\end{tikzcd}$ 
    is a \textit{naive Verdier sequence} if it is a fibre and a cofibre sequence in $\CatL$. 
\end{defi}

By Proposition~\ref{prop:limits-and-colimits} and Remark~\ref{rmq:explicit}, being a naive Verdier sequence amounts to the following conditions:
\begin{itemize}
    \item The sequence of stable categories $\Acal \to \Ccal \to \Bcal$
    is a Verdier sequence.
    \item The natural transformation $\alpha\colon  N\xrightarrow{\simeq} M\circ(\catop{i}\times i)$ and $\hat{\beta}\colon (\catop{p}\times p)_!M\xrightarrow{\simeq} P$ are equivalences, where $\hat{\beta}$ is the mate of $\beta$.
\end{itemize}

\begin{ex}\label{ex:map-is-naive}
    If $\Acal \to \Ccal \to \Bcal$ is a Verdier sequence in $\CatEx$ then 
    $$
        (\Acal,\map_{\Acal}) \longrightarrow (\Ccal,\map_{\Ccal}) \longrightarrow (\Bcal,\map_{\Bcal})
    $$
    is a naive Verdier sequence in $\CatL$.
\end{ex}
    
\begin{defi}
A functor $\CatL\to\Ecal$ with stable target is said to be \textit{Verdier localizing} if it is reduced and sends naive Verdier sequences to exact sequences in $\Ecal$.
\end{defi}    
    
Unlike in the split case, it need not be that $\Lace$ sends naive Verdier sequences to Verdier sequence, and so there is no reason to expect that $\Klace=K\circ\Lace$ would be Verdier localising.

\begin{defi} \label{VerdierSeqDef}
     A sequence $\begin{tikzcd}[cramped](\Acal, N)\arrow[r, "{(i, \alpha)}"]& (\Ccal, M)\arrow[r, "{(p, \beta)}"] &(\Bcal, P)\end{tikzcd}$
     is a \textit{fine Verdier sequence} if it is a naive Verdier sequence such that the following condition holds:
     \begin{itemize}[leftmargin=*]
         \item For every $n\geq 0$, the sequence 
         \[\begin{tikzcd}[cramped]\Lace(\Acal, \Sigma^nN)\arrow[r]& \Lace(\Ccal, \Sigma^nM)\arrow[r] &\Lace(\Bcal, \Sigma^nP)\end{tikzcd}\] is a Verdier sequence of stable categories.
     \end{itemize}
    A functor $\CatL\to\Ecal$ with stable target is said to be \textit{weakly Verdier localizing} if it is reduced and sends fine Verdier sequences to exact sequences in $\Ecal$.
\end{defi}

\begin{ex}\label{ex:orthogonal-is-fine-verdier}
    If $(\Acal, N) \subseteq (\Ccal,M) \supseteq (\Bcal, P)$ is a laced semi-orthogonal decomposition then the naive Verdier sequence 
    \[
        \begin{tikzcd}
            (\Acal, N)\arrow[r, "{(i, \alpha)}"] & (\Ccal, M)\arrow[r, "{(p, \eta)}"] & (\Bcal, P)
        \end{tikzcd}
    \]
    furnished by Lemma~\ref{lmm:laced-adjoint} is a fine Verdier sequence. Indeed, by Proposition~\ref{AdjPresSOD} (and its proof) this sequence is sent by $\Lace$ to a right-split Verdier sequence, and this persists under shifting since laced semi-orthogonal decompositions are stable under shifts. 
\end{ex}

By definition, every Verdier localising functor is weakly localising. In addition, by Example~\ref{ex:orthogonal-is-fine-verdier} any weakly Verdier localising functor is additive.

\begin{prop}
    The functor $\Klace\colon \CatL\to\Sp$ is weakly Verdier localizing. Consequently, by Theorem~\ref{KlaceIsUnivAdditive} it is also the initial such functor under $\SigmaInftyP\LaceEq$.
\end{prop}
\begin{proof}
    This follows directly from the fact that $\Kth\colon \CatEx\to\Sp$ is Verdier localizing (see for instance \cite[Theorem 3.13]{SaunierFundamental} or \cite[Theorem 6.1]{HebestreitLachmannSteimle}). 
\end{proof}

\begin{rmq}\label{rem:thh-is-loc}
    If $(\Acal,N) \to (\Ccal,M) \to (\Bcal,P)$ is a naive (resp. fine) Verdier sequence in $\CatL$ then $(\Acal,0) \to (\Ccal,0) \to (\Bcal,0)$ is a naive (resp. fine) Verdier sequence and $(\Acal,\Sig^mN) \to (\Ccal,\Sig^n M) \to (\Bcal,\Sig^nP)$ is a naive (resp. fine) Verdier sequence for every $n \in \mathbb{Z}$. It then follows from Remark~\ref{rmq:summary-lemmas} that the operations of tangentially reduced and tangentially exact approximation both preserve (weakly) Verdier localising functors.
\end{rmq}

\begin{cor}
    The functor $\THH\colon \CatL \to \Sp$ is weakly Verdier localising.
\end{cor}

We finish this subsection by discussing an example of fine Verdier projections not arising from laced semi-orthogonal decompositions.

\begin{cons}\label{cons:prototypical}
    Let $\Acal,\Bcal$ be two stable categories and $N\colon \Acal\op \times \Bcal \to \Sp$, $M \colon \Bcal\op \times \Acal \to \Sp$ two bimodules. Write $T_M\colon \Ind(\Acal) \to \Ind(\Bcal)$ and $T_N\colon \Ind(\Bcal) \to \Ind(\Acal)$ for the associated colimit preserving functors,  
    so that $M(b,a) = \map(b,T_M(a))$ and $N(a,b) = \map(a,T_N(b))$. 
    Consider the pairing category $\Pair(\Acal,\Bcal;N)$, which sits in a fibre square 
    \[ 
        \begin{tikzcd}
            \Pair(\Acal,\Bcal;N) \ar[r]\ar[d] & \Ar(\Ind(\Acal)) \ar[d] \\
            \Acal \times \Bcal \ar[r, "{(j,T_N)}"] & \Ind(\Acal) \times \Ind(\Acal)
        \end{tikzcd}
    \]
    It fits in a bifibration made of two Verdier projections
    \[ \Acal \xleftarrow{q} \Pair(\Acal,\Bcal,N) \xrightarrow{p} \Bcal, \]
    where $q$ is a cartesian fibration and $p$ a cocoartesian one. 
    Write $\widetilde{M} := (p\op \times q)^*M$. Since both $p$ are $q$ are Verdier projections the associated comonads $p_!p^*$ and $q_!q^*$ are equivalent to the respective identities, and so  
    \[ (p\op \times p)_!\widetilde{M} = (\id\op \times p)_!(\id\op \times q)^*M = M \circ N \in \BiMod(\Bcal) \] 
    and 
    \[ (q\op \times q)_!\widetilde{M} = (q\op \times \id)_!(p\op \times \id)^*M = N \circ M \in \BiMod(\Acal).\] 
    The functors $p$ and $q$ then upgrade to a pair of laced functors 
    \begin{equation}\label{eq:proto}
        (\Acal,N \circ M) \xleftarrow{\widetilde{q}} (\Pair(\Acal,\Bcal;N),\widetilde{M}) \xrightarrow{\widetilde{p}} (\Bcal,M \circ N) .
    \end{equation}
\end{cons}

\begin{prop}\label{prop:proto-fine}
    Both projections in~\eqref{eq:proto} are fine Verdier projections.
\end{prop}
\begin{proof}
    We prove the claim fro $\widetilde{p}$, the proof for $\widetilde{q}$ is dual. To begin, note that the underlying exact functor $p$ is a left-split Verdier projection and that its laced upgrade $\widetilde{p}$ is by construction a cocartesian arrow of $\CatL$.
    Extending $\widetilde{p}$ by its kernel $(\Acal,0)$ hence yields a naive Verdier sequence in $\CatL$.
    Up to shifting $M$, it will now suffice to show that the sequence
    \begin{equation}\label{eq:fine-verdier}
        \Acal = \Lace(\Acal,0) \to \Lace(\Pair(\Acal,\Bcal;N),\widetilde{M})) \to \Lace(\Bcal,M \circ N) 
    \end{equation}
    is a Verdier sequence for any choice of input data $\Acal,\Bcal,M,N$. Write 
    \[ N_{\Ind} = \map_{\Ind(\Bcal)}(-,T_N(-))\colon \Ind(\Acal)\op \times \Ind(\Bcal) \to \Sp \]
    and
    \[ M_{\Ind} = \map_{\Ind(\Bcal)}(-,T_M(-))\colon \Ind(\Bcal)\op \times \Ind(\Acal) \to \Sp \]
    for the corresponding Ind extensions of $M$ and $N$, so that we have a bifibration 
    \[ \Ind(\Acal) \xleftarrow{q_{\Ind}} \Pair(\Ind(\Acal),\Ind(\Bcal),N_{\Ind}) \xrightarrow{p_{\Ind}} \Ind(\Bcal) \]
    as above.
    Consider the commutative diagram
    \begin{equation}\label{eq:big-diagram}
        \begin{tikzcd}
            \Ind(\Acal) \ar[r]\ar[d,equal] & \Ind(\Lace(\Pair(\Acal,\Bcal;N),\widetilde{M})) \ar[r, "\Ind(\widetilde{p})"]\ar[d] & \Ind(\Lace(\Bcal,M \circ N)) \ar[d] \\
            \Ind(\Acal) \ar[r] & \Lace(\Pair(\Ind(\Acal),\Ind(\Bcal);N_{\Ind}),\widetilde{M}_{\Ind}) \ar[r,"\widetilde{p}_{\Ind}"] & \Lace(\Ind(\Bcal),M_{\Ind} \circ N_{\Ind})
        \end{tikzcd}
    \end{equation}
    where $\widetilde{M}_{\Ind} = (p^{\mathrm{op}}_{\Ind} \times q_{\Ind})^*M_{\Ind}$ is obtained from $M_{\Ind}$ in the same manner that 
    $\widetilde{M}$ was obtained from $M$ above.
    Here, both rows are fibre sequences and the vertical arrows are all fully-faithful embeddings (since underlying-compact laced objects are compact, and similarly for paired objects). In addition, the bottom row in this diagram is a right-split Verdier sequence: indeed, a fully-faithful right adjoint to $\widetilde{p}_{\Ind}$ is given by the functor 
    \[
        \begin{tikzcd}[row sep=5pt]
            \Lace(\Ind(\Bcal),M_{\Ind} \circ N_{\Ind}) \ar[r,"r"] & \Lace(\Pair(\Ind(\Acal),\Ind(\Bcal);N_{\Ind}),\widetilde{M}_{\Ind}) \\ 
            (b,\beta\colon b \to T_M(T_N(b))) \ar[r,mapsto] & ((T_N(b),b,\id),\beta).
        \end{tikzcd}
    \]
    We claim that $r$ sends the image of the right most vertical inclusion in~\eqref{eq:big-diagram} 
    to the image of the middle vertical inclusion, thus yielding a fully-faithful right adjoint to the top right horizontal map. This will imply that the sequence~\eqref{eq:fine-verdier}
    is a Karoubi sequence by~\cite[Appendix A]{HermKII}, and hence a Verdier sequence once we verify in addition that $\widetilde{p}$ is essentially surjective. 
        
    We now address the above claims about $r$ and $\widetilde{p}$. We establish both of them at the same time, by showing that for $(b,\beta\colon b \to T_M(T_N(b))) \in \Lace(\Bcal,M \circ N)$, the object $r(b) = ((T_N(b),b,\id),\beta)$ can be written as a filtered colimits of objects in $\Lace(\Pair(\Acal,\Bcal;N),\widetilde{M})$, in such a way that the image in $\Lace(\Bcal,M \circ N)$ of the corresponding filtered diagram is constant on $b$. For this, note that the pair $(T_N(b),\beta)$ determines an object $z \in \Zcal := \Ind(\Acal) \times_{\Ind(\Bcal)} \Ind(\Bcal)_{b/}$, and that to obtain the type of diagram we want it will suffice to write $z$ as a filtered colimit of objects in $\Zcal_0 := \Acal \times_{\Ind(\Bcal)} \Ind(\Bcal)_{b/}$. More precisely, we claim that the canonical map
    \[ \colim_{[z' \to z] \in (\Zcal_0)_{/z}} z' \to z \]
    is an equivalence. 
    Now since the projection $\Zcal \to \Ind(\Acal)$ is conservative and preserves filtered colimits and since every object in $\Ind(\Acal)$ is the colimit of its corresponding canonical diagram in $\Acal$ the desired result will now follow once we show that the projection $\pi\colon (\Zcal_0)_{/z} \to \Acal_{/T_N(b)}$ is cofinal. Now since $\Acal_{/T_N(b)}$ is filtered the projection $\Acal_{a//T_N(b)} \to \Acal_{/T_N(b)}$ is cofinal for every $a \to T_N(b) \in \Acal_{/T_N(b)}$, and since in addition $\pi$ is a left fibration the cofinality of $\pi$ is equivalence to the weak contractibility of $(\Zcal_0)_{/z}$. This, in turn, is equivalent to the statement that the inclusion $(\Zcal_0)_{/z} \subseteq \Zcal_{/z}$ is a weak homotopy equivalence, since $\Zcal_{/z}$ is itself weakly contractible. This inclusion fits in pullback square
    \[
        \begin{tikzcd}
            (\Zcal_0)_{/z} \ar[r]\ar[d] & \Zcal_{/z} \ar[d] \\
            \Acal_{/T_N(b)} \ar[r] & \Ind(\Acal)_{/T_N(b)}
        \end{tikzcd}
    \]
    where the vertical maps are left fibrations, both classified by the functor 
    \[ [a \to T_N(b)] \mapsto  \Map(b,T_N(a)) \times_{\Map(b,T_M(T_N(b)))} \{\beta\},\] 
    where for the left vertical map $a$ is understood to be in $\Acal$, while in the right vertical map it is taken in $\Ind(\Acal)$. Since this formula is compatible with filtered colimits in $a$ we see that the functor classifying the right vertical left fibration is left Kan extended from the one classifying the left vertical fibration. We hence conclude that the top horizontal map is a weak equivalence, as desired.
\end{proof}

\begin{rmq}
Proposition~\ref{prop:proto-fine} means in particular that the laced functors in~\eqref{eq:proto} induce Verdier projections after applying $\Lace$. Inspecting the proof shows something slightly more precise about the relation between these Verdier projections and the corresponding Verdier projections on underlying unlaced $\infty$-categories: in the commutative diagram
\[
\begin{tikzcd}
\Lace(\Acal,N \circ M) \ar[d] & \ar[l] \Lace(\Pair(\Acal,\Bcal;N),\widetilde{M}) \ar[r]\ar[d] & \Lace(\Bcal,M \circ N) \ar[d] \\
\Acal & \Pair(\Acal,\Bcal;N)\ar[l]\ar[r] & \Bcal
\end{tikzcd}
\]
whose horizontal arrows are all Verdier projections, the left square is horizontally Pro-adjointable and the right square is horizontally Ind-adjointable.
\end{rmq}

\subsection{\texorpdfstring{$\THH$}{THH} as a laced theory}

We have shown that $\THH$ is trace-like and from \cite{HoyoisScherotzkeSibilla}, we know that it is actually the trace in the category $\PrExCat{L}$ of presentable stable categories and left adjoint functors. Note however that being trace-like need not imply the usual trace invariance property. In this last section, we introduce and study the latter notion.

\begin{defi}\label{def:trace-invariant}
    Let $\Ecal$ be a category and $\Fcal\colon \CatL \to \Ecal$ a functor. We say that $\Fcal$ is a \emph{laced theory} if it sends the two laced functors in~\eqref{eq:proto} to equivalences for every $\Acal,\Bcal,N$ and $M$.
\end{defi}

In an earlier version of this draft, we had used the terms \emph{trace-invariant} to describe the above property, but while preparing the sequel of this paper, we decided against this name to facilitate unify the naming schemes within our work. If $\Fcal\colon \CatL \to \Ecal$ is a laced theory functor then for any $\Acal,\Bcal,M,N$ as in Construction~\ref{cons:prototypical} we have a distinguished equivalence
\[ 
    \Fcal(\Acal, N \circ M) \simeq \Fcal(\Bcal, M \circ N)
\]
More generally, by composing the $\Pair$ constructions, one gets distinguished equivalences showing that laced theories $\Fcal$ are invariant under cyclic permutations in the bimodule coordinate. Note that bimodules only depend of the idempotent completion of $\Ccal$, which we denote $\mathrm{Idem}(\Ccal)$; in consequence, in the situation where $\Acal=\Ccal$, $\Bcal=\mathrm{Idem}(\Ccal)$, $M$ being any $\Ccal$-bimodule and $N=\id$, the above unravels to an equivalence
$$
	\Fcal(\Ccal, M) \simeq \Fcal(\mathrm{Idem}(\Ccal), M)
$$
which shows that laced theories are invariant under passage to the idempotent completion.
 
\begin{rmq}\label{rem:invariant-is-like}
Taking $N$ to be the identity module in Construction~\ref{cons:prototypical}, 
it follows from Proposition~\ref{TestSampleTrLike} that every laced theory is trace-like.
\end{rmq}

\begin{defi}\label{def:coarse-equiv}
We will say that an arrow in $\CatL$ is a \emph{trace equivalence} if it is sent to an equivalence by any laced theory. 
\end{defi}

By Remark~\ref{rem:invariant-is-like} we have that any strict trace equivalence (in the sense of Definition~\ref{defi:trace-equivalence}) is a trace equivalence in the above sense.

\begin{rmq}\label{rem:trace-invariant-proto}
In the situation of Construction~\ref{cons:prototypical}, if the bimodule $N$ is left representable by an exact functor $f\colon \Acal \to \Bcal$ (that is, $N = \hat{Q}_f = \map_{\Bcal}(f(-),-)$), then the laced functor $\widetilde{q}$ in~\eqref{eq:proto} admits a section of the form 
\[ (\Acal,\hat{Q}_f \circ M) \to (\Pair(\Acal,\Bcal,\hat{Q}_f),\widetilde{M}) \quad\quad X \mapsto (X,f(X),\id) .\]
The composite of this section with $\widetilde{p}$ then yields a laced functor
\begin{equation}\label{eq:tilde-f} 
\widetilde{f}\colon (\Acal,\hat{Q}_f \circ M) \to (\Bcal,M \circ \hat{Q}_f) 
\end{equation}
whose underlying exact functor is $f$. We conclude that the resulting $\widetilde{f}$ is a trace equivalence.
Dually, if the bimodule $N$ is right representable by an exact functor $g\colon \Bcal \to \Acal$ (that is, $N = Q_g = \map_{\Acal}(-,g(-))$), then the laced functor $\widetilde{p}$ in~\eqref{eq:proto} admits a section of the form
\[ (\Bcal,M \circ Q_g) \to (\Pair(\Acal,\Bcal,Q_g),\widetilde{M}) \quad\quad Y \mapsto (g(Y),Y,\id) .\]
The composite of this section with $\widetilde{q}$ then yields a laced functor
\begin{equation}\label{eq:tilde-g} 
\widetilde{g}\colon (\Bcal,M \circ Q_g) \to (\Acal, Q_g \circ M) 
\end{equation}
whose underlying exact functor is $g$. We then similarly have that $\widetilde{g}$ is a trace equivalence. 
\end{rmq}

\begin{ex}\label{ex:primitive}
For a stable category $\Ccal$ and objects $A,B \in \Ccal$ let us write $M_{A,B} \in \BiMod(\Ccal)$ for the bimodule given by the formula $M_{A,B}(X,Y) = \map_{\Ccal}(X,A)\otimes \map_{\Ccal}(B,Y)$. Then $A$ and $B$ determine exact functors $(-)\otimes A, (-) \otimes B$ from $\Spaf$ to $\Ccal$ and we have $M_{A,B} = Q_A \circ \hat{Q}_B$ in the notation of Remark~\ref{rem:trace-invariant-proto}. Applying that remark we hence get that the associated laced functor
\[ ((-)\otimes B,\eta) \colon (\Spaf,\hat{Q}_B \circ Q_A) \to (\Ccal,Q_A \circ \hat{Q}_B) \]
is a trace equivalence, where $\eta$ is the natural transformation 
\begin{align*}
(\hat{Q}_B \circ Q_A)(E,E') = \map(E\otimes B,  E'\otimes A) =& \map(E\otimes B,  A)\otimes E' \\
\to& \map(E \otimes B,A)\otimes \map(B, E'\otimes B) \\
=& M_{A,B}(E \otimes B, E'\otimes B) 
\end{align*} 
induced by the unit $E'\to \map(B,E'\otimes B)$ of the adjunction $(-) \otimes B \dashv \map(B,-)$. At the same time, note that tensoring with bimodule $\map$ induces an equivalence $\Sp \simeq \BiMod(\Spaf)$, so that in the present situation the $\Spaf$-bimodule $\hat{Q}_B \circ Q_A$ corresponds to the spectrum $\map_{\Ccal}(B,A)$. We hence conclude that if $\Fcal\colon \CatL \to \Ecal$ is any laced theory then the above laced functor determines an equivalence
\[ \Fcal(\Spaf,\map_{\Ccal}(B,A)) \xrightarrow{\simeq} \Fcal(\Ccal,M_{A,B}) .\]
\end{ex}

\begin{rmq}\label{rmq:THH-of-primitive}
Combining Example~\ref{ex:primitive} with Example~\ref{ex:thh-of-spaf} we obtain another perspective on the equivalence
\[ \THH(\Ccal,M_{A,B}) \simeq \map_{\Ccal}(B,A) \]
of Example~\ref{ex:primitive-bimodules}, as arising from $\THH$ being a laced theory whose restriction to $\CatL \times_{\CatEx} \{\Spaf\} = \BiMod(\Spaf) = \Sp$ is the identify.
\end{rmq}

We now show that for tangentially exact invariants, the difference between weakly Verdier localising and Verdier localising vanish.

\begin{prop}\label{prop:loc-is-trace}
    Let $\Ecal$ be a stable category and $\Fcal\colon \CatL \to \Ecal$ a tangentially exact functor. Then the following are equivalent:
    \begin{enumerate}
        \item $\Fcal$ is Verdier localising.
        \item $\Fcal$ is weakly Verdier localising.
        \item $\Fcal$ is a laced theory.
    \end{enumerate}
    In fact, the implication $(2) \Rightarrow (3)$ holds for any tangentially reduced $\Fcal$.
\end{prop}
\begin{proof}
    Clearly $(1) \Rightarrow (2)$ for any $\Fcal$, and $(2) \Rightarrow (3)$ for any tangentially reduced $\Fcal$ since the two arrows in~\eqref{eq:proto} are fine Verdier projections (Proposition~\ref{prop:proto-fine}) whose kernels have zero bimodules. 
    We now show that $(3) \Rightarrow (1)$. 
    Let $F\colon \CatL \to \Ecal$ be a tangentially exact laced theory  
    and 
    \[ (\Acal,M) \xrightarrow{i} (\Bcal,N) \xrightarrow{p} (\Ccal,P)\] 
    a naive Verdier sequence in $\CatL$. 
    Then $P = (p\op \times p)_!N = Q_p \circ N \circ \hat{Q}_p$ and $M = (i\op \times i)^*N = \hat{Q}_i \circ N \circ Q_i$, where $Q_p = \map_{\Acal}(-,p(-)), \hat{Q}_p = \map_{\Acal}(p(-),-)$ are the two bimodules left and right represented by $p$, and similarly for $Q_i$ and $\hat{Q}_i$. We may then place this sequence in a commutative diagram of the form
    \[ 
        \begin{tikzcd}
            (\Acal,\hat{Q}_i \circ N \circ Q_i) \ar[r,"i"]\ar[d] & (\Bcal,N) \ar[d,equal]\ar[r, "p"] & (\Ccal,Q_p \circ N \circ \hat{Q}_p)  \\
            (\Bcal,N\circ Q_i \circ \hat{Q}_i) \ar[r] & (\Bcal,N) \ar[r] & (\Bcal,N \circ \hat{Q}_p \circ Q_p) \ar[u]
        \end{tikzcd}
    \]
    where the vertical arrows are trace equivalences by Remark~\ref{rem:trace-invariant-proto}, and are hence sent to equivalences by $\Fcal$. Since $\Fcal$ sends the lower sequence to an exact sequence by virtue of being tangentially exact we conclude that $\Fcal$ sends the top sequence to an exact sequence, as desired.
\end{proof}

\begin{cor}\label{cor:thh-is-trace-inv}
The functor $\THH\colon \CatL \to \Sp$ is a laced theory.
\end{cor}
\begin{proof}
Combine Proposition~\ref{prop:loc-is-trace} and Remark~\ref{rem:thh-is-loc}.
\end{proof}

By the universal property of Corollary~\ref{StableuTHHisTHH} we thus obtain:

\begin{cor}\label{cor:universal-thh-2}
The composed map 
\[ \Sig^{\infty}_+\iota\Lace \Rightarrow \Sig^{\infty}_+\uTHH \Rightarrow \THH\] 
exhibits $\THH$ as initial among tangentially exact laced theories under $\Sig^{\infty}\iota\Lace_+$.
\end{cor}

Similarly, combining Proposition~\ref{prop:loc-is-trace}, Remark~\ref{rem:thh-is-loc} and Proposition~\ref{prop:universal-K-cyc}, we conclude:
\begin{cor}
The cyclic K-theory functor $\Kth^{\cyc}$ is a laced theory and the map $\Klace \Rightarrow \Kth^{\cyc}$ exhibits it as the initial laced theory under $\Kth$.
\end{cor}

\subsection{Trace maps for arbitrary Verdier localising functors}

In this section we generalize the construction of the trace map
$\tr\colon \Kth\Rightarrow\THH$ 
to arbitrary Verdier localising functors $\Fcal\colon \CatEx \to \Ecal$ with stable presentable targets.

\begin{defi}
Let $\Ecal$ be a stable presentable category and $\Fcal\colon \CatEx \to \Ecal$ a functor. We let $\Fcal^{\lace} \colon \CatL \to \Ecal$ be the functor given by $\Fcal^{\lace}(\Ccal,M) = \Fcal(\Lace(\Ccal,M))$, and write
\[ d\Fcal\colon \CatL \to \Ecal \]
for the tangentially exact approximation of $\Fcal^{\lace}$.
\end{defi}

\begin{prop}\label{prop:derivative-is-trace-invariant}
Let $\Ecal$ be a stable presentable category. If $\Fcal\colon \CatEx \to \Ecal$ is a Verdier localising functor then $d\Fcal\colon \CatL \to \Ecal$ is a laced theory. On the other hand, if $\Gcal\colon \CatL \to \Ecal$ is a tangentially exact laced theory then the functor $\Ccal \mapsto \Gcal(\Ccal,\map)$ is Verdier localising.
\end{prop}
\begin{proof}
We begin with the first claim. Since $\Fcal$ is Verdier localising we have that $\Fcal^{\lace}$ is weakly Verdier localising. By Remark~\ref{rem:thh-is-loc} we then have that $d\Fcal$ is weakly Verdier localising, and so also a laced theory by Proposition~\ref{prop:loc-is-trace}.

Now let $\Gcal\colon \CatL \to \Ecal$ be a tangentially exact laced theory. Then by Proposition~\ref{prop:loc-is-trace} we have that $\Gcal$ is Verdier localising, and so $\Gcal(-,\map)$ is Verdier localising by Example~\ref{ex:map-is-naive}.
\end{proof}

Proposition~\ref{prop:derivative-is-trace-invariant} implies that the composite
\[ \Fun(\CatEx,\Ecal) \stackbin[\L^*]{\Lace^*}{\adj}
 \Fun(\CatL,\Ecal) \stackrel{\fbwnexcPart{1}}{\adj} \Fun^{\bex}(\CatL,\Ecal)\]
of the adjunction induced via pre-composition from the adjunction $\L \dashv \Lace$ of Proposition~\ref{prop:left-adj-to-lace}, and the adjunction of Corollary~\ref{cor:fiberwise-derivative},
restricts to an adjunction
\[
d\colon \Fun^{\vloc}(\CatEx,\Ecal) \adj \Fun^{\tr,\bex}(\CatL,\Ecal) \cocolon \L^* 
\]
where on the left we have the category of Verdier localising functors $\CatEx \to \Ecal$ and on the right the category of  tangentially exact laced theories $\CatL \to \Ecal$. We then write
\[ \partial := \L^*\circ d\colon \Fun^{\vloc}(\CatEx,\Ecal) \to \Fun^{\vloc}(\CatEx,\Ecal) \]
for the composite of the adjoint pair. By construction, $\partial$ is a monad on $\Fun^{\vloc}(\CatEx,\Ecal)$.

\begin{defi}
Given a Verdier localising functor $\Fcal\colon \CatEx \to \Ecal$, we call the unit 
\[ \tr_{\Fcal}\colon \Fcal \Rightarrow \partial\Fcal \]
the \emph{trace map} associated to $\Fcal$.
\end{defi}

Finally, given a Verdier localising functor $\Fcal\colon \CatEx \to \Ecal$, let us write $\Fcal^{\cyc}(\Ccal,M) := \cof[\Fcal^{\lace}(\Ccal,0) \to \Fcal^{\lace}(\Ccal,M)]$ for the tangentially reduced approximation of $\Fcal^{\lace}$.
 
\begin{prop}
The natural transformations $\Fcal^{\lace} \Rightarrow \Fcal^{\cyc} \Rightarrow d\Fcal$ exhibits their targets respectively as the initial laced theory and initial tangentially exact laced theory under $\Fcal^{\lace}$.
\end{prop}
\begin{proof}
For $d\Fcal$ this follows from Proposition~\ref{prop:derivative-is-trace-invariant}. For $\Fcal^{\cyc}$ combine Proposition~\ref{prop:loc-is-trace} and Remark~\ref{rem:thh-is-loc} to deduce that $\Fcal^{\cyc}$ is a laced theory and then apply Proposition~\ref{prop:universal-reduced}.
\end{proof}

\subsection{Multiplicativity of \texorpdfstring{$\THH$}{THH}}

Assembling Theorem~\ref{KlaceIsUnivAdditive}, Remark~\ref{uTHHUP-2}, Corollary~\ref{cor:universal-thh-2} and we obtain a commutative square 
\begin{equation}\label{eq:square-uni}
\begin{tikzcd}
\SigmaInftyP\LaceEq\ar[r,Rightarrow]\ar[d,Rightarrow] & \Klace \ar[d,Rightarrow] \\
\SigmaInftyP\uTHH \ar[r,Rightarrow] & \THH
\end{tikzcd}
\end{equation}
of functors $\CatL \to \Sp$, where the maps out of the top left corner exhibit the other three corners as the initial additive, initial trace-like, and initial trace-like tangentially exact functors under $\SigmaInftyP\LaceEq$, respectively. In addition, by Proposition~\ref{KlaceInitialLaxMonoidal} and Proposition~\ref{uTHHInitialLaxMonoidal}, the top horizontal and left vertical maps refine to ones of lax symmetric monoidal functors from the top left corner, which is the initial lax symmetric monoidal functor $\CatL \to \Sp$ by \cite[Corollary 6.8]{Nikolaus}, to $\Klace$ and $\SigmaInftyP\uTHH$ which are, respectively, the initial lax symmetric monoidal additive and trace-like such functors.

\begin{prop} \label{THHInitialLaxMonoidal}
    The lax symmetric monoidal refinement of the top left span in~\eqref{eq:square-uni} extends to a refinement of the entire square to one of lax symmetric monoidal functors, 
rendering $\THH\colon \CatL\to\Sp$ the initial lax symmetric monoidal tangentially exact trace-like functor.
\end{prop}
\begin{proof}
Let $\widehat{\Sp}$ denote the category of possibly large spectra.
    We begin by proving that the left Bousfield localisation 
    \[ \Fun(\CatL, \widehat{\Sp})\locadjL{\fbwnexcPart{1}}\Fun^{\bex}(\CatL, \widehat{\Sp})\]
    is multiplicative with respect to Day convolution, 
	and hence induces a left Bousfield localisation
    \[ \Fun^{\otimes-\lax}(\CatL,\widehat{\Sp}) \locadjL{(\fbwnexcPart{1})_*} \Fun^{\otimes-\lax,\bex}(\CatL,\widehat{\Sp}) \]
    on the level of lax symmetric monoidal functors by passing to algebra objects. To avoid confusion, let us explain that working with $\widehat{\Sp}$ is simply a technical solution to the fact that $\CatL$ is not small, and hence Day convolution with coefficients in $\Sp$ is only an operad; a different would be to work with multiplicative localisations of operads, see \cite[Appendix A.1]{HermKIV}.
    Now since $\fbwnexcPart{1}$ is a localization at the collection of arrows $\eta\colon F(-)\Rightarrow T_1F\colon =\Omega F(\Sigma -)$ , it suffices to check that $\eta\otimes_{\mathrm{Day}}G$ is again inverted by the localization for any $G\colon \CatL\to\widehat{\Sp}$, where $\otimes_{\mathrm{Day}}$ denotes the Day convolution. For this, it suffices to note that $\eta\otimes_{\mathrm{Day}}G$ fits in the following commutative diagram:
    \[
        \begin{tikzcd}[column sep=large]
            F\otimes_{\mathrm{Day}}G\arrow[rd, "\eta_{F\otimes_{\mathrm{Day}}G}"']\arrow[r, "\eta\otimes_{\mathrm{Day}}G"] & T_1(F)\otimes_{\mathrm{Day}}G \arrow[rd, "\eta_{T_1(F)\otimes_{\mathrm{Day}}G}"]\arrow[d] \\
            & T_1(F\otimes_{\mathrm{Day}}G)\arrow[r] & T_1(T_1(F)\otimes_{\mathrm{Day}}G)
        \end{tikzcd}
    \]
    Hence, since arrows inverted by a localization satisfy 2-out-of-6, we have the wanted claim. 
    
    Since $\SigmaInftyP\uTHH$ is the initial trace-like lax symmetric monoidal functor $\CatL \to \Sp$ by Proposition \ref{uTHHInitialLaxMonoidal}, applying $(\fbwnexcPart{1})_*$ to $\SigmaInftyP\uTHH$ now yields a lax symmetric monoidal refinement of $\THH$ which, given that $\THH$ is also trace-like, exhibits $\THH$ as the initial lax symmetric monoidal trace-like and tangentially exact functor. In addition, the unit $\SigmaInftyP\uTHH \Rightarrow \THH$ is now a map of lax symmetric monoidal functors. In particular, the composite
    \[ \SigmaInftyP\LaceEq \Rightarrow \SigmaInftyP\uTHH \Rightarrow \THH \]
    is also a map of lax symmetric monoidal functors. Since $\SigmaInftyP\LaceEq$ is the initial lax symmetric monoidal functor and $\Klace$ is initial both as an additive functor under $\SigmaInftyP\LaceEq$ and as a lax symmetric monoidal additive functor the map $\Klace \Rightarrow \THH$ uniquely refines to a map of lax symmetric monoidal functors completing the square~\eqref{eq:square-uni}.
    \end{proof}

\begin{prop}\label{prop:THH-is-symmetric-monoidal}
The lax symmetric monoidal functor $\THH\colon \CatL \to \Sp$ is actually symmetric monoidal.
\end{prop}
\begin{proof}
It is enough to show that the lax symmetric monoidal structure maps
\[ \bS \to \THH(\Spaf,\map) \quad\text{and}\quad \THH(\Ccal,M) \otimes \THH(\Dcal,N) \to \THH((\Ccal,M) \otimes (\Dcal,N)) \]
are equivalences. For the unit, this is just Example~\ref{ex:thh-of-spaf} for $M=\map$.
As for the second structure map, note that for fixed $\Ccal$ and $\Dcal$ this map is a natural transformation between two colimit preserving functors in $M$ and $N$, while, at the same time, $\BiMod(\Ccal)$ is generated under colimits by the primitive bimodules of the form $M_{A,B}(X,Y) = \map_{\Ccal}(X,A)\otimes \map_{\Ccal}(B,Y)$, and similarly for $\BiMod(\Dcal)$. It will hence suffice to prove the claim for $M=M_{A,B}$ and $N =N_{C,D}$ for $A,B \in \Ccal$ and $C,D \in \Dcal$. By
Example~\ref{ex:primitive-bimodules} the map in question then becomes the map
\[ \map_{\Ccal}(B,A) \otimes \map_{\Dcal}(D,C) \to \map_{\Ccal \otimes \Dcal}(B \otimes D, A \otimes C) ,\]
which is indeed an equivalence, essentially by the construction of the tensor product in $\CatEx$.
\end{proof}

\small
\bibliographystyle{alpha}
\bibliography{bibliographie}

\end{document}